\documentclass[1 [leqno,11pt]{amsart}
\usepackage{amssymb, amsmath,latexsym,amsfonts,amsbsy, amsthm,mathtools,graphicx,CJKutf8,CJKnumb,CJKulem,color}
\usepackage{float}
\usepackage{hyperref}
\usepackage{environ}
\usepackage{tikz} 
\numberwithin{equation}{section}

\allowdisplaybreaks

\let\al=\alpha

\let\f=\frac

\let\om=\omega

\def\R{\mathbb R}

\newcommand{\beq}{\begin{equation}}
\newcommand{\eeq}{\end{equation}}
\newcommand{\ben}{\begin{eqnarray}}
\newcommand{\een}{\end{eqnarray}}
\newcommand{\beno}{\begin{eqnarray*}}
\newcommand{\eeno}{\end{eqnarray*}}

\newtheorem{theorem}{Theorem}[section]

\newtheorem{lemma}[theorem]{Lemma}
\newtheorem{proposition}[theorem]{Proposition}

\newtheorem{remark}[theorem]{Remark}

 \NewEnviron{elaboration}{

\par

\begin{tikzpicture}

\node[rectangle,minimum width=0.9\textwidth] (m) {\begin{minipage}{0.85\textwidth}\BODY\end{minipage}};

\draw[dashed] (m.south west) rectangle (m.north east);

\end{tikzpicture}

}

\begin{document}
\begin{CJK*}{UTF8}{gkai}
\title[Enhanced dissipation and nonlinear  stability of the Taylor-Couette flow]{Enhanced dissipation and nonlinear  \\   asymptotic stability of the Taylor-Couette flow \\ for the 2D Navier-Stokes equations}

\author{Xinliang AN}
\address{Department of Mathematics, National University of Singapore, Singapore}
\email{matax@nus.edu.sg}

\author{Taoran HE}
\address{Department of Mathematics, National University of Singapore, Singapore}
\email{taoran\underline{~}he@u.nus.edu}

\author{Te LI}
\address{Department of Mathematics, National University of Singapore, Singapore}
\email{matlit@nus.edu.sg}

\date{\today}

\maketitle

\begin{abstract}In this paper, we study the nonlinear stability of a steady circular flow created between two rotating concentric cylinders. The dynamics of the viscous fluid are described by 2D Navier-Stokes equations. We adopt scaling variables. For the rescaled equations, we prove that the steady flow (Taylor-Couette flow) is asymptotically stable up to a large perturbation of initial data. Back to the original 2D Navier-Stokes equations, this implies an improved transition threshold for the Taylor-Couette flow. The improvement is due to enhanced dissipation and new observations and constructions of weighted $L^2$ norms, which capture a hidden structure between the viscosity constant $\nu$ and (different) rotating speeds and locations of two coaxial cylinders. In particular, we allow the location of the outer cylinder to tend to infinity, which renders the initial fluid kinetic energy not uniformly bounded. Due to enhanced-dissipation effect, we also establish a sharp resolvent estimate, desired space-time bounds and optimal decaying estimates, which lead to the proof of nonlinear asymptotic stability of 2D Taylor-Couette flow.

\end{abstract}

\tableofcontents

\section{\textbf{Introduction}}
In this paper, we consider the two-dimensional (2D) incompressible Navier-Stokes equations:
\begin{align}
\label{full nonlinear equation}\left\{
\begin{aligned}
&\partial_tv-\nu\Delta v+v\cdot \nabla v+\nabla p=0,\quad\text{div } v=0,\\
&v(0,x)=v_0(x)
\end{aligned}
\right.
\end{align}
with $x=(x_1,x_2)\in\Omega=\mathbb{R}^2$ being the space variables and $t\geq0$ being the time variable. Here the constant $\nu>0$ is called the kinematic viscosity. The unknowns to \eqref{full nonlinear equation} are the velocity field $v(t,x)=(v^1(t,x),v^2(t,x))$ and the pressure $p(t,x)\in\mathbb{R}$.

If taking the 2D curl to the above Navier-Stokes equations, then \eqref{full nonlinear equation} is transferred to
\begin{align}
\label{full nonlinear equation ns vor}
&\partial_t\omega-\nu\Delta\omega+v\cdot \nabla \omega=0.
\end{align}
Here $\omega=\text{curl }v=\partial_1v^2-\partial_2v^1$ is the vorticity field. For 2D case, it is a scalar. And \eqref{full nonlinear equation ns vor} is called the vorticity equation formulation. In this paper, we will use \eqref{full nonlinear equation ns vor} to understand the asymptotics of solutions to \eqref{full nonlinear equation}.

Given $\omega$ being a solution to \eqref{full nonlinear equation ns vor}, the velocity field $v:\mathbb{R}_{+}\times\mathbb{R}^2\to\mathbb{R}^2$ can be derived via solving the below elliptic system
\begin{align*}
\text{div }v=0, \quad\text{curl }v=\omega.
\end{align*}
 This leads to the 2D Biot-Savart law
\begin{align}
\label{BS law}v(t,x)=\f{1}{2\pi}\int_{\mathbb{R}^2}\f{(x-y)^{\perp}}{|x-y|^2}\omega(t,y)dy.
\end{align}
For notational simplicity, we write $v(t,\cdot)=K_{BS}\ast \omega(t,\cdot)$, where
$K_{BS}(x)=\f{1}{2\pi} \f{x^{\perp}}{|x|^2}$ is the Biot-Savart kernel.

\subsection{Taylor-Couette flow}
In this article, we focus on studying the nonlinear stability and long-time dynamics around the $circular\ flows$ with the vorticity $\omega(x)=\omega(|x|)$ being radial. In below, we denote $\tilde{r}$ to be $|x|$.

An important physical circular-flow solution to \eqref{full nonlinear equation ns vor} is called the \textbf{Taylor-Couette flow}. It describes a steady (circular) flow of viscous fluid bounded between two rotating infinitely long coaxial cylinders and has wide applications ranging from desalination to viscometric analysis. Now we derive its expression. 

Denote $\omega(x)=\omega(\tilde{r})$ and $\phi(x)=\phi(\tilde{r})$ to be the radial vorticity and the stream function, respectively. Then for 2D case we have
\begin{align}
\label{radical vor}
&\omega(\tilde{r})=\Delta\phi=\phi''(\tilde{r})+\f{1}{\tilde{r}}\phi'(\tilde{r}),\\
&v(x,y)=\left(
  \begin{array}{ccc}
   -\partial_2\phi \\
   \partial_1\phi \\
  \end{array}
\right)=\partial_{\tilde{r}}\phi e_{\tilde{\theta}}=\left(
  \begin{array}{ccc}
   -x_2 \\
   x_1 \\
  \end{array}
\right)\f{\phi'(\tilde{r})}{\tilde{r}}.
\end{align}
One can easily check $\nabla\cdot v=0$ and via $\partial_1=\cos\tilde{\theta}\partial_{\tilde{r}}-\f{1}{\tilde{r}}\sin\tilde{\theta}\partial_{\tilde{\theta}}$ , $\partial_2=\sin\tilde{\theta}\partial_{\tilde{r}}+\f{1}{\tilde{r}}\cos\tilde{\theta}\partial_{\tilde{\theta}}$, one also has
\begin{align*}
v\cdot\nabla\omega=-\sin\tilde{\theta}\phi'\partial_1\omega+\cos\tilde{\theta}\phi'\partial_2\omega=0.
\end{align*}
Hence, a radial vorticity $\omega(\tilde{r})$ in \eqref{radical vor} is a stationary solution to the below 2D incompressible Euler equation
\begin{align}
\label{full nonlinear equation euler vor}
&\partial_t\omega+v\cdot \nabla \omega=0.
\end{align}
In particular, if we choose $\omega=\textrm{const}$, then we have $\Delta\omega=0$ and $\omega$ is also a solution to the 2D incompressible Navier-Stokes equation \eqref{full nonlinear equation ns vor}. For this case, by \eqref{radical vor} the stream function $\phi(\tilde{r})$ reads
\begin{align}
\label{Taylor-Couette flow stream function}\phi''(\tilde{r})+\f{1}{\tilde{r}}\phi'(\tilde{r})=2A_1. 
\end{align}
where $A_1,A_2$ are constants. This implies $(\tilde{r}\phi')'=2A_1\tilde{r}$ and   $\f{\phi'}{\tilde{r}}=A_1+\f{A_2}{\tilde{r}^2}$.

Via \eqref{Taylor-Couette flow stream function} we now derive a solution $v(x,y)$ to \eqref{full nonlinear equation}. And with polar coordinate, we rename $v(x,y)$ and $\omega(r)$ to be $U(\tilde{r},\tilde{\theta})$ and $\Omega(\tilde{r})$, respectively. Now we have
\begin{align}
\label{Taylor-Couette flow}U(\tilde{r},\tilde{\theta})=\left(
  \begin{array}{ccc}
   U^1 \\
   U^2\\
  \end{array}
\right)=\left(
  \begin{array}{ccc}
   -\sin\tilde{\theta} \\
   \cos\tilde{\theta}\\
  \end{array}
\right)(A_1\tilde{r}+\f{A_2}{\tilde{r}})=v_{\tilde{\theta}}e_{\tilde{\theta}},\quad \Omega(\tilde{r})=2A_1,
\end{align}
where $A_1,A_2$ are constants, $e_{\tilde{\theta}}=\left(
  \begin{array}{ccc}
   -\sin\tilde{\theta} \\
   \cos\tilde{\theta}\\
  \end{array}
\right)$, and $v_{\tilde{\theta}}=A_1\tilde{r}+\f{A_2}{\tilde{r}}$ is the azimuthal velocity component.  Note that $U(\tilde{r},\tilde{\theta})$ in \eqref{Taylor-Couette flow} is a stationary solution to \eqref{full nonlinear equation} and it is called the \textbf{Taylor-Couette flow}.

\subsection{Connection and comparison to the Lamb-Oseen vortex}It is well known that the 2D Navier-Stokes equations \eqref{full nonlinear equation ns vor} have a family of self-similar solutions called the Lamb-Oseen vortices, which are of the following form
\ben
\label{Lamb-Oseen vortices}\om(t,x)=\f \al {\nu t}G\Big(\f x {\sqrt{\nu t}}\Big),\quad v(t,x)=\f \al {\sqrt{\nu t}}v^G\Big(\f x {\sqrt{\nu t}}\Big).
\een
Here the vorticity profile and the velocity profile are given by
\ben
\label{Lamb-Oseen vortices-rescal}G(\xi)=\f 1 {4\pi}e^{-|\xi|^2/4},\quad v^G(\xi)=\f 1 {2\pi}\f {\xi^\perp} {|\xi|^2}\Big(1-e^{-|\xi|^2/4}\Big).
\een
In particular we have $v^G=K_{BS}\ast G$ with $K_{BS}$ being the Biot-Savart kernel. And we further denote space
\begin{align}
    Y:=L^2(\mathbb{R}^2,G^{-1}d\xi).
\end{align}

An important property of flow obeying \eqref{full nonlinear equation ns vor} is that it preserves the mass in $L^1(\mathbb{R}^2)$, i.e.,
\begin{align}
\label{L1 intergable}\int_{\R^2}\om(t,x)dx=\int_{\R^2}\om(0,x)dx=\al\quad \textrm{for any}\ t>0.
 \end{align}
 The parameter $\al\in \R$ is called the circulation Reynolds number. People studying fluid mechanics are especially interested in rapidly rotating vortices, where the circulation $|\al|$ is much larger compared with the kinematic viscosity $\nu$. This is the regime most
relevant with applications to 2D turbulent flows. Recently, there are exciting progress on the linear and nonlinear stability of the Lamb-Oseen vortex. The results in this article are parallel to these developments. And here we review some important results about the Lamb-Oseen vortex.

In 2005, to study the long-time dynamics of the 2D Navier-Stokes equations, Gallay and Wayne \cite{GW} introduced the so called \textit{scaling variables} or \textit{similarity variables}:
\begin{align}
\label{similarity variables}\xi=\f{x}{\sqrt{\nu t}},\quad \tau=\log t.
\end{align}
And they work under the ansatz
\begin{align}
\label{rewritten form}\omega(x,t)=\f{1}{t}w(\f{x}{\sqrt{\nu t}},\log t),\quad v(x,t)=\sqrt{\f{\nu}{t}}u(\f{x}{\sqrt{\nu t}},\log t).
\end{align}
The rescaled vorticity $w(\xi,\tau)$ now satisfies the evolution equation
\begin{align}
\label{self-similar transform}\partial_\tau w-( \Delta_\xi+\f12\xi\cdot\nabla_\xi+1)w+u\cdot\nabla_\xi w=0.
\end{align}
When the initial vorticity is integrable, i.e., when $\alpha$ defined in \eqref{L1 intergable} is finite, Gallay and Wayne \cite{GW0,GW} proved that the long-time dynamical behaviours of the 2D Navier-Stokes equations can be described and approximated by the Lamb-Oseen vortex.
Let $G$ be defined  as in \eqref{Lamb-Oseen vortices-rescal}. They proved
\begin{proposition}[Gallay-Wayne \cite{GW}]\label{Prop 1.1} For any $w_0\in Y$, the rescaled vorticity equation \eqref{self-similar transform} admits a unique global mild solution $w\in C^0([0,\infty),Y)$ with $w(0)=w_0$. This solution satisfies $\|w(\tau)-\alpha G\|_{Y}\to0$ as $\tau\to+\infty$, where $\alpha=\int_{\mathbb{R}^2} w_0(\xi)d\xi$.
\end{proposition}

Later in 2008, Gallay and Rodrigues \cite{Ga0, GR} showed that for any finite value of circulation parameter $\alpha$, the equilibrium $\alpha G$ is asymptotically stable in $Y$.

\begin{proposition} [Gallay-Rodrigues \cite{Ga0, GR}]\label{Prop 1.2} There exists $\epsilon>0$ such that, for all $\alpha\in\mathbb{R}$ being finite and for all $w_0 \in \alpha G+ Y_0$ with $Y_0=\{w\in Y| \int_{\mathbb{R}^2}w(\xi)d\xi=0\}$, if requiring $\|w_0-\alpha G\|_{Y}\leq\epsilon$, then the unique solution to \eqref{self-similar transform} with initial data $w_0$ satisfies
\begin{align*}
\|w(\cdot,\tau)-\alpha G\|_{Y}\leq \min(1, 2 e^{-\f{\tau}{2}})\|w_0- \alpha G\|_{Y} \quad \textrm{for any} \ \tau\geq 0.
\end{align*}
\end{proposition}

A feature associated with \eqref{full nonlinear equation ns vor} that Proposition \ref{Prop 1.2} did not address is the enhanced dissipation effect when the circulation Reynolds numbers $|\alpha|$ is large. Gallay proposed a conjecture on the optimal resolvent estimate for the Oseen vortex. And this conjecture was solved by Li, Wei and Zhang in \cite{LWZ-O}. Let $G$ and $v^G$ be defined  as in \eqref{Lamb-Oseen vortices-rescal}. Set $L$ to be $\Delta_\xi+\f12\xi\cdot\nabla_\xi+1$
and operator $\Lambda$ to satisfy
\begin{align*}
\Lambda w= v^G\cdot \nabla w+u\cdot\nabla G \quad  \textrm{with} \ u=K_{BS}\ast w.
\end{align*}
Note that $\Lambda$ is a nonlocal linear operator. Then the following statement holds

\begin{proposition}[Li-Wei-Zhang \cite{LWZ-O}]\label{Prop 1.3-0} Denote $L_{\bot}$ and $\Lambda_{\bot}$ to be $L$ and $\Lambda$'s restriction to the orthogonal complement of ker $\Lambda$ in Y. Define the below pseudospectral  bound
\begin{align*}
\Psi(\alpha):=(\sup_{\lambda\in\mathbb{R}}\|(L_{\bot}-\alpha\Lambda_{\bot}-i\lambda)^{-1}\|_{Y\to Y})^{-1}.
\end{align*}
Then as $|\alpha|\to+\infty$, there exists $C>0$ independent of $\alpha$ satisfying
\begin{align*}
 C^{-1}|\alpha|^{\f13}\leq \Psi(\alpha)\leq C|\alpha|^{\f13}.
\end{align*}
\end{proposition}

Based on the linear resolvent estimate in \cite{LWZ-O}, Gallay \cite{Ga} proved an improved stability result  which allows the size of perturbation (size of the basin of attraction) of Oseen's vortex to be large when $|\alpha|\gg1$. Also, the non-zero frequency of vorticity obeys a faster decaying rate. Rewrite $w(\xi,\tau)$ as $w(\theta,r,\tau)$ with polar coordinate and let space $Y_1$ be $\{w\in Y| \int_{\mathbb{R}^2}\xi_iw(\xi)d\xi=0, i=1,2\}$. It holds
\begin{proposition}[Gallay \cite{Ga}]\label{Prop 1.3} There exist positive constants $C_1, C_2$, $\kappa$ such that, for all $\alpha\in\mathbb{R}$ and all initial data $w_0\in \alpha G + Y_1$ satisfying
\begin{align*}
\|w_0-\alpha G\|_{Y}\leq \f{C_1(1+|\alpha|)^{\f16}}{\log(2+|\alpha|)},
\end{align*}
we have that the unique solution to \eqref{self-similar transform} in $Y$ (guaranteed by Proposition \ref{Prop 1.1}) obeys, for any $\tau\geq0$,
\begin{align*}
&\|w(\cdot,\tau)-\alpha G\|_{Y}\leq C_2e^{-\tau}\|w_0-\alpha G\|_{Y},\\
&\|(1-P_r)w(\cdot,\tau)-\alpha G\|_{Y}\leq C_2\|w_0-\alpha G\|_{Y}\exp(-\f{\kappa(1+|\alpha|)^{\f13}\tau}{\log(2+|\alpha|)}),
\end{align*}
where 
\begin{align*}
P_rw(\theta,r,\tau)=\int_0^{2\pi}w(\theta,r,\tau)d\theta.
\end{align*}
\end{proposition}

\subsection{Main results}In this paper, we study fluid dynamics around the Taylor-Couette flow. We are in particular interested in the regime with a high Reynolds number ($\textrm{Re}=\f{1}{\nu}$).

The Taylor-Couette flow is the steady flow created between two rotating concentric cylinders. In this paper, we consider the 2D Taylor-Couette flow, and aim to prove its nonlinear asymptotic stability under the 2D incompressible Navier-Stokes equations \eqref{full nonlinear equation ns vor}. For the Taylor-Couette flow, if the inner cylinder with radius $r_1$ is rotating at constant angular velocity $\omega_1$ and the outer cylinder with radius $r_2$ is rotating at constant angular velocity $\omega_2$, then the azimuthal velocity component is given by
\begin{align}
\label{TC flow in physics}v_{\theta}=A_1r+\f{A_2}{r}=r\omega_{phy},\quad A_1=\omega_1\f{\mu-\eta^2}{1-\eta^2},\quad A_2=\omega_1r_1^2\f{1-\mu}{1-\eta^2},
\end{align}
where
\begin{align*}
\mu=\f{\omega_2}{\omega_1},\quad \eta=\f{r_1}{r_2},\quad \omega_{phy}:=A_1+\f{A_2}{r^2}.
\end{align*}
One can verify
\begin{align*}
\omega_1=\omega_{phy}(r_1),\quad \omega_2=\omega_{phy}(r_2).
\end{align*}
Then we can directly compute fluid vorticity from the velocity field as below
\begin{align*}
\Omega(r)=&\partial_1U^2-\partial_2U^1=\partial_1(x_1\omega_{phy})-\partial_2(-x_2\omega_{phy})\\
=&2\omega_{phy}+x_1\partial_1\omega_{phy}+x_2\partial_2\omega_{phy}=2\omega_{phy}+r\omega_{phy}'=2A_1.
\end{align*}
 Although $\omega_{phy}$ is singular at $r=0$, the above equality indicates that $\Omega(r)$ is regular at $r=0$.

Before presenting our main theorems, let's review some related results on the mathematical study of the Taylor-Couette flow or more general circular flows:
\begin{itemize}
\item In \cite{Z1}, for 2D incompressible Euler equations, Zillinger established linear inviscid damping with decay rates around the Taylor-Couette flow in an annular region. Note that this restriction on the annual region also avoid the issues of fluid dynamics at the origin and at infinity.

  \item In \cite{ZZ}, for 2D Euler equations, Zelati and Zillinger established linear inviscid damping and proved linear stability for a class of mildly degenerate flows. This class includes $U(r)\sim \f{1}{r}+r$, which is different from shear flows with strict monotonicity \cite{WZZ-E}.

\item Recently, Gallay and $\check{S}$ver$\acute{a}$k \cite{GV} employed variational principles (related to Arnold's stability criteria) to study orbital stability of steady circular solutions to both Euler's equations and the Navier-Stokes equation at low viscosity. They proved the stability for the algebraic vortex $\omega(x)=(1+|x|^2)^{-\kappa}$ with $\kappa>1$ and for the Lamb-Oseen vortex $\omega(x)=e^{-\f{|x|^2}{4}}$. Note that the vorticity of the Taylor-Couette flow does not decay when $|x|\to +\infty$. And it has not been addressed in \cite{GV}.
\end{itemize}

We proceed to state our main results. In this paper, we work with scaling variables:
\begin{align*}
\xi=\f{x}{\sqrt{\nu t}},\quad \tau=\log t.
\end{align*}
In \cite{GW}, Gallay and Wayne rewrite solutions to \eqref{full nonlinear equation ns vor} under the ansatz
\begin{align*}
\omega(x,t)=\f{1}{t}w(\f{x}{\sqrt{\nu t}},\log t),\quad v(x,t)=\sqrt{\f{\nu}{t}}u(\f{x}{\sqrt{\nu t}},\log t).
\end{align*}
Then the rescaled vorticity $w(\xi,\tau)$ satisfies the following evolution equation
\begin{align*}
\partial_\tau w-( \Delta_\xi+\f12\xi\cdot\nabla_\xi+1)w+u\cdot\nabla_\xi w=0.
\end{align*}
The rescaled velocity $u$ can be expressed via the Biot-Savart law,  namely $u(\cdot, \tau)=K_{BS} \ast w(\cdot,\tau)$. If we write $\xi$ in polar coordinates
\begin{align*}
\xi_1=r\cos\theta,\quad\xi_2=r\sin\theta\quad \textrm{with}\  r\in[0,\infty)\quad \textrm{and}\ \theta\in\mathbb{T}.
\end{align*}
Then we can take Fourier transform with respect to $\theta$. Denoting $w_k(\tau, r)=\f{1}{2\pi}\int_{\mathbb{T}} w(\tau, r, \theta) e^{-ik\theta} d \theta,$ we can write
\begin{align*}
w(\tau,\xi)=w(\tau,r,\theta)=\sum\limits_{k\in\mathbb{Z}} \omega_k(\tau,r)e^{ik\theta}.
\end{align*}

Recall that the fluid velocity $U$ and vorticity $\Omega$ of the Taylor-Couette flow \eqref{Taylor-Couette flow} are given by
\begin{align*}
U(\tilde{r},\tilde{\theta})=\left(
  \begin{array}{ccc}
   U^1 \\
   U^2\\
  \end{array}
\right)=\left(
  \begin{array}{ccc}
   -\sin\tilde{\theta} \\
   \cos\tilde{\theta}\\
  \end{array}
\right)(A_1\tilde{r}+\f{A_2}{\tilde{r}})=v_{\tilde{\theta}}e_{\tilde{\theta}},\quad \Omega(r)=2A_1.
\end{align*}
Here $A_1,A_2$ are constants, $e_{\tilde{\theta}}=\left(
  \begin{array}{ccc}
   -\sin\tilde{\theta} \\
   \cos\tilde{\theta}\\
  \end{array}
\right)$ and $v_{\tilde{\theta}}$ is the azimuthal velocity component. Via the self-similar transformation, the steady circular solutions \eqref{Taylor-Couette flow} to the original \eqref{full nonlinear equation euler vor} equation are concerted to
\begin{align}
\label{TC flow after self-similar transformation}V=\sqrt{\f{t}{\nu}}U=\left(
  \begin{array}{ccc}
   -\xi_2 \\
   \xi_1\\
  \end{array}
\right)(A_1e^{\tau}+\f{A_2}{\nu|\xi|^2}),\quad W=t\Omega=2A_1e^{\tau}.
\end{align}
For 2D Navier-Stokes equations, we work with scaling variables and its equivalent form \eqref{self-similar transform}. Thus the original problem is transformed into analyzing the stability of \eqref{TC flow after self-similar transformation} with
equation \eqref{self-similar transform}. For initial data, via Fourier series we have $w(0, r, \theta)=w_0(0, r)+\sum_{k\in\mathbb{Z}\backslash \{0\}}w_k(0, r)e^{ik\theta}$. For simplicity, we also employ $w(0), w_0(0), w_k(0)$ to replace $w(0, r, \theta), w_0(0, r), w_k(0, r)$, respectively.

For future use, we define
\begin{align}
\label{weighted space}\|w\|_{\mathcal{M}}^2:=\int_0^{\infty}re^{\f{r^2}{4}}|w(r)|^2dr,\quad\quad
\|w\|_{X}:=(\int_0^{\infty}\f{|w|^2}{r^2}dr)^{\f12}.
\end{align}
And we introduce an energy norm  $ \mathcal{E}(\tau)$ as below
\begin{align*}
 \mathcal{E}(\tau)=|\f{A_2}{\nu}|^{\f16}\|\f{\omega_0(\tau)-2A_1e^{\tau}}{r}\|_{\mathcal{M}}+\sum_{k\in\mathbb{Z}/\{0\}}(\|\omega_k(\tau)\|_{\mathcal{M}}+|k|^{\f16}|\f{A_2}{\nu}|^{\f16}\|\f{\omega_k(\tau)}{r}\|_{\mathcal{M}}).
  \end{align*} 
\noindent For initial data, via Fourier series we have $w(0, r, \theta)=w_0(0, r)+\sum_{k\in\mathbb{Z}\backslash \{0\}}w_k(0, r)e^{ik\theta}$. For simplicity, we also adapt $w(0), w_0(0), w_k(0)$ to denote $w(0, r, \theta), w_0(0, r), w_k(0, r)$, respectively. In our paper, we are interested in the regime when $\nu\leq |A_2|$. In particular, we allow $0<\nu \ll 1$. Our main result of this paper can be stated as follows:
\begin{theorem}[Main Theorem]\label{Main result-physical space} For any $|A_2|\geq\nu$, there exist constants $c_0,c,C>0$ independent of $A_1,A_2,\nu$ such that if the initial data $w(0)=w_0(0)+\sum_{k\in\mathbb{Z}\backslash \{0\}}w_k(0)e^{ik\theta}$ satisfy
\begin{align*}
 \mathcal{E}(0)=|\f{A_2}{\nu}|^{\f16}\|\f{\omega_0(0)-2A_1}{r}\|_{\mathcal{M}}+\sum_{k\in\mathbb{Z}/\{0\}}(\|\omega_k(0)\|_{\mathcal{M}}+|k|^{\f16}|\f{A_2}{\nu}|^{\f16}\|\f{\omega_k(0)}{r}\|_{\mathcal{M}})\leq c_0|\f{A_2}{\nu}|^\f13,
  \end{align*}
  then the solution $w$ of the full nonlinear vorticity formulation \eqref{self-similar transform} is global in time and the following stability estimate holds
  \begin{align*}
 &|\f{A_2}{\nu}|^{\f16}\|\f{e^{-\tau}\omega_0(\tau)-2A_1}{r}\|_{\mathcal{M}}\leq Ce^{-\tau} \mathcal{E}(0), \footnotemark[1]  
  \end{align*}
  \footnotetext[1]{Recall $\omega=2A_1 e^{\tau}$ in \eqref{TC flow after self-similar transformation}. In scaling variables, we expect  $e^{-\tau}\omega_0(\tau)-2A_1$ term obeys decaying estimate.}
 and
\begin{align*}
\sum_{k\in\mathbb{Z}/\{0\}}(\|\omega_k(\tau)\|_{\mathcal{M}}+|k|^{\f16}|\f{A_2}{\nu}|^{\f16}\|\f{\omega_k(\tau)}{r}\|_{\mathcal{M}})\leq C e^{-c|\f{A_2}{\nu}|^{\f13}\tau}\mathcal{E}(0).\end{align*}
Hence, it holds
\begin{align*}
  \mathcal{E}(\tau)\leq C  \mathcal{E}(0).
\end{align*}
\end{theorem}
 In our work, we keep tracking the relation between the enhanced dissipation effect and the coefficients $\nu$ and $A_1,A_2$. We are especially interested in the case with rapidly moving velocity, where the parameter $|A_2|$ is much larger than the kinematic viscosity $\nu$.
\begin{remark}
Compared with Proposition \ref{Prop 1.3} stated above by \cite{Ga}, here we allow larger perturbation. In \cite{Ga} for the Oseen's vortex, the size of the basin for attraction is of $|\alpha|^{\f16}$(where $\alpha\sim\nu^{-1})$, while here for Taylor-Couette flow, $|\f{A_2}{\nu}|$ is analogous to $|\alpha|$, and the size of perturbation is permitted to be of size $|\f{A_2}{\nu}|^{\f13}$. Moreover, to control the nonlinear terms in a sharper way, we construct and employ a designed weighted energy space $X$ with $\|w\|_X=(\int_0^{\infty}\f{|w|^2}{r^2}dr)^{\f12}$. This can be fulfilled due to a newly found connection about order of $|\f{kA_2}{\nu}|$ between two space-time bounds. Tentative space-time estimates leads us to consider the below two sums 
\begin{align}
\label{1-e}&\|e^{c|\f{kA_2}{\nu}|^{\f13}\tau}w_k\|_{L^{\infty}L^2}+|\f{kA_2}{\nu}|^{\f16}\boxed{\|e^{c|\f{kA_2}{\nu}|^{\f13}\tau}w_k\|_{L^2L^2}}+\|e^{c|\f{kA_2}{\nu}|^{\f13}\tau}(\f{|k|}{r}+r)w_k\|_{L^2L^2},\\
\label{2-e}&\|e^{c|\f{kA_2}{\nu}|^{\f13}\tau}w_k\|_{L^{\infty}X}+|\f{kA_2}{\nu}|^{\f16}\|e^{c|\f{kA_2}{\nu}|^{\f13}\tau}w_k\|_{L^2X}+\boxed{\|e^{c|\f{kA_2}{\nu}|^{\f13}\tau}rw_k\|_{L^2X}}.
\end{align}
And we expect \eqref{1-e} and $\eqref{2-e}\cdot |\f{kA_2}{\nu}|^{\f16}$ are of the same order. Note that in \eqref{2-e} the coefficient of $\|e^{c|\f{kA_2}{\nu}|^{\f13}\tau}rw_k\|_{L^2X}=\|e^{c|\f{kA_2}{\nu}|^{\f13}\tau}w_k\|_{L^2L^2}$ is $1$, and in \eqref{1-e} the coefficient in front of  $\|e^{c|\f{kA_2}{\nu}|^{\f13}\tau}w_k\|_{L^2L^2}$ is $|\f{kA_2}{\nu}|^{\f16}$. This observation motivates us to couple $L^2$ and $X$ space together to construct combined energy $E_k(|k|\geq1)$ of order $\eqref{1-e}+|\f{kA_2}{\nu}|^{\f16}\eqref{2-e}$ as below
\begin{align}
\label{basic energy}E_k=&\|e^{c|\f{kA_2}{\nu}|^{\f13}\tau}w_k\|_{L^{\infty}L^2}+|\f{kA_2}{\nu}|^{\f16}\Big(\|e^{c|\f{kA_2}{\nu}|^{\f13}\tau}w_k\|_{L^2L^2}+\|e^{c|\f{kA_2}{\nu}|^{\f13}\tau}w\|_{L^{\infty}X}\\
\nonumber &+|k|\|e^{c|\f{kA_2}{\nu}|^{\f13}\tau}\f{w_k}{r^2}\|_{L^2L^2}+|k|^{\f12}\|e^{c|\f{kA_2}{\nu}|^{\f13}\tau}\f{w_k}{r^{\f32}}\|_{L^2L^{\infty}}\Big)
+|\f{kA_2}{\nu}|^{\f13}\|e^{c|\f{kA_2}{\nu}|^{\f13}\tau}w\|_{L^2X}.
\end{align}
It turns out that $E_k$ is the correct energy, which enables us to prove the desired result with improvements. If we only adapt the natural $L^2$ energy \eqref{1-e} as in \cite{Ga}, we can prove a stability result allowing perturbation of initial data is less then $(\f{A_2}{\nu})^{\f16}$, while with $X$ space and $E_k$ energy, we allow the perturbation to be less than $|\f{A_2}{\nu}|^{\f13}$. By establishing sharp resolvent estimates and employing Gearhart-Pr$\ddot{u}$ss  type theorem in \cite{Wei}, we also avoid the loss of $\log|\alpha|$ in Proposition \ref{Prop 1.3}.

\end{remark}

\begin{remark}To prove the nonlinear result Theorem \ref{Main result-physical space}, we also obtain a sharp pseudospectral bound $(\f{A_2}{\nu})^{\f13}$ for the linearized operator around the Taylor-Couette flow in Section \ref{2-resolvnet}. This sharp bound is consistent with the estimates for the Lamb-Oseen vortices operator in \cite{LWZ-O} by Li-Wei-Zhang. The obtained pseudospectral bound then implies a sharp enhanced dissipation decaying rate $e^{-(\f{A_2}{\nu})^{\f13}\tau}$ for the vorticity equation. Not only we give resolvent estimates from $L^2$ space to $L^2$ space, we also establish the counterpart from $L^2$ space to $H^{-1}$ space and extend the optimal resolvent estimate to the weighted $L^2$ space. These estimates together infer desired space-time bounds for nonlinear terms in Section \ref{2-resolvnet}.

\end{remark}

Translated back to the original equations \eqref{full nonlinear equation ns vor}, initial data at $\tau=0$ is  corresponding to data at $t=1$. The above main theorem implies the below improved transition threshold result:
\begin{theorem}\label{improved transition threshold}(Transition Threshold) For any $|A_2|\geq\nu$, there exist constants $C_1,C_2>0$ independent of $A_1,A_2,\nu$ such that if the initial data at $t=1$ satisfies
\begin{align*}
\|\omega(1)-2A_1\|_{L^1}\leq C_1|A_2|^{\f13}\nu^{\f23},
  \end{align*}
  then the solution $\omega(t)$ to the nonlinear vorticity formulation \eqref{full nonlinear equation ns vor} is global in time and obeys the following stability estimate
 \begin{align*}
\|\omega(t)-2A_1\|_{L^1}\leq C_2|A_2|^{\f13}\nu^{\f23}.
  \end{align*}
\end{theorem}

\begin{remark}
In this theorem, $|A_2|^{\f12}\nu^{\f23}$ of order $\nu^{\f23}$ is the transition threshold. Recall on page 6 of \cite{Ga} Gallay proved (in $L^1$ norm) a transition threshold of order $\nu^{\f56}$ for the flow near the Lamb-Oseen vortices. By scaling consideration, we expect consistent transition thresholds for the Lamb-Oseen vortices and for the Taylor-Couette flow. Our improvement comes from the desired estimates in weighted norms obtained in Theorem \ref{Main result-physical space}. In particular, we overcome the difficulty for the cases when $r_1\to 0$ and $r_2\to +\infty$.

\end{remark}

\subsection{Derivation of equations}\label{1.4-Derivation of equations}

 For notational simplicity, in below we write $A_1$ and $\f{A_2}{\nu}$ as $A$ and $B$ respectively. And we also use polar coordinates
\begin{align*}
\xi_1=r\cos\theta,\quad\xi_2=r\sin\theta\quad \textrm{with} \ r\in[0,\infty) \quad  \textrm{and} \ \theta\in\mathbb{T}.
\end{align*}
Define $\tilde{w}=w-W$, $\tilde{u}=u-V$ with $W$ and $V$ given in \eqref{TC flow after self-similar transformation}. Via \eqref{self-similar transform} and $\tilde{u}\cdot\nabla_\xi W=0$, we have
\begin{align}
\label{pertubation of TC flow in ss trans}\partial_\tau\tilde{w}-( \Delta_\xi+\f12\xi\cdot\nabla_\xi+1)\tilde{w}+(\tilde{u}+V)\cdot\nabla_\xi\tilde{w}=0.
\end{align}
Since $\partial_{\xi_1}=\cos\theta\partial_{r}-\f{1}{r}\sin\theta\partial_{\theta}$ and $\partial_{\xi_2}=\sin\theta\partial_{r}+\f{1}{r}\cos\theta\partial_{\theta}$, we deduce
\begin{align*}
-\xi_2\partial_{\xi_1}+\xi_1\partial_{\xi_2}&=\partial_{\theta},\quad \xi_1\partial_{\xi_1}+\xi_2\partial_{\xi_2}=r\partial_{r},\quad\Delta_{\xi}=\partial_{r}^2+\f{1}{r}\partial_{r}+\f{1}{r^2}\partial_{\theta}^2.
\end{align*}
Thus, equation \eqref{pertubation of TC flow in ss trans} becomes
\begin{align}
\label{pertubation of TC flow in ss trans in polar}\partial_\tau\tilde{w}-[(\partial_{r}^2+\f{1}{r}\partial_{r}+\f{1}{r^2}\partial_{\theta}^2)+\f12r\partial_{r}+1]\tilde{w}+(Ae^{\tau}+\f{B}{r^2})\partial_{\theta}\tilde{w}+\tilde{u}\cdot\nabla_\xi\tilde{w}=0.
\end{align}
Recall we have $\tilde{u}=\left(
  \begin{array}{ccc}
   -\partial_{\xi_2}\varphi \\
   \partial_{\xi_1}\varphi \\
  \end{array}\right)$ with $\varphi$ being the stream function satisfying
  \begin{align*}
  \Delta_{\xi}\varphi=(\partial_{r}^2+\f{1}{r}\partial_{r}+\f{1}{r^2}\partial_{\theta}^2)\varphi=\tilde{w}. \end{align*}
  Then it follows 
\begin{align*}
\tilde{u}\cdot\nabla_\xi\tilde{w}=&-\partial_{\xi_2}\varphi\partial_{\xi_1}\tilde{w}+\partial_{\xi_1}\varphi\partial_{\xi_2}\tilde{w}\\
=&-(\sin\theta\partial_{r}\varphi+\f{1}{r}\cos\theta\partial_{\theta}\varphi)(\cos\theta\partial_{r}\tilde{w}-\f{1}{r}\sin\theta\partial_{\theta}\tilde{w})\\
&+(\cos\theta\partial_{r}\varphi-\f{1}{r}\sin\theta\partial_{\theta}\varphi)(\sin\theta\partial_{r}\tilde{w}+\f{1}{r}\cos\theta\partial_{\theta}\tilde{w})\\
=&\f{1}{r}(\partial_{r}\varphi\partial_{\theta}\tilde{w}-\partial_{\theta}\varphi\partial_{r}\tilde{w}).
\end{align*}
Together with
\begin{align*}
\partial_{r}\varphi\partial_{\theta}\tilde{w}-\partial_{\theta}\varphi\partial_{r}\tilde{w}=\partial_{r}(\varphi\partial_{\theta}\tilde{w})-\partial_{\theta}(\varphi\partial_{r}\tilde{w}), \end{align*}
we derive the below (nonlinear) perturbation equation \eqref{pertubation of TC flow in ss trans in polar} for $\tilde{w}$
\begin{align}
\label{pertubation of NS vor}
\partial_\tau\tilde{w}&-[ (\partial_{r}^2+\f{1}{r}\partial_{r}+\f{1}{r^2}\partial_{\theta}^2)+\f12r\partial_{r}+1]\tilde{w}+(Ae^{\tau}+\f{B}{r^2})\partial_{\theta}\tilde{w}\\
\nonumber&+\f{1}{r}[\partial_{r}(\varphi\partial_{\theta}\tilde{w})-\partial_{\theta}(\varphi\partial_{r}\tilde{w})]=0.
\end{align}
Take the Fourier transform in $\theta$ direction. Denote the Fourier coefficient  of $\tilde{w}$ and $\varphi$ to be $\hat{w}_k$ and $\hat{\varphi}_k$, respectively. Then equation \eqref{pertubation of NS vor} becomes
\begin{align}\label{pertubation of NS vor-fourier}\partial_\tau\hat{w}_k&-[(\partial_{r}^2+\f{1}{r}\partial_{r}-\f{k^2}{r^2})+\f12r\partial_{r}+1]\hat{w}_k+(Ae^{\tau}+\f{B}{r^2})ik\hat{w}_k\\
\nonumber&+\f{1}{r}[\sum_{l\in\mathbb{Z}}il\partial_{r}(\hat{w}_l\hat{\varphi}_{k-l})-ik\sum_{l\in\mathbb{Z}}\hat{\varphi}_{k-l}\partial_{r}\hat{w}_{l}]=0
\end{align}
with $\hat{\varphi}_k$ satisfying $(\partial_r^2+\f{1}{r}\partial_r-\f{k^2}{r^2})\hat{\varphi}_k=\hat{w}_k$.

Denote $\breve{w}_k:=e^{ikAe^{\tau}}\hat{w}_k$ and $\breve{\varphi}_k:=e^{ikAe^{\tau}}\hat{\varphi}_k$. We obtain an equivalent form for \eqref{pertubation of NS vor-fourier}
\begin{align*}
\partial_\tau\breve{w}_k&-[(\partial_{r}^2+\f{1}{r}\partial_r-\f{k^2}{r^2})+\f12r\partial_r+1]\breve{w}_k+\f{ikB}{r^2}\breve{w}_k\\
&+\f{1}{r}[\sum_{l\in\mathbb{Z}}il\partial_r(\breve{w}_l\breve{\varphi}_{k-l})-ik\sum_{l\in\mathbb{Z}}\breve{\varphi}_{k-l}\partial_r\breve{w}_{l}]=0
\end{align*}
with $\breve{\varphi}_k$  satisfying $(\partial_r^2+\f{1}{r}\partial_r-\f{k^2}{r^2})\breve{\varphi}_k=\breve{w}_k$.

\noindent Let us further define
$w_k:=f^{-1}\breve{w}_k$ with  $f=\f{e^{-\f{r^2}{8}}}{r^{\f12}}$.
This $f^{-1}$ weight is designed to make the first derivative $\f{1}{r}\partial_r+\f12r\partial_r$ vanish. With $w_k$ we deduce
\begin{align}
\label{scaling nonlinear}\partial_\tau w_k&- [\partial_r^2-(\f{k^2-\f14}{r^2}+\f{r^2}{16}-\f12)]w_k+\f{ikB}{r^2}w_k\\
\nonumber&+\f{f^{-1}}{r}[\sum_{l\in\mathbb{Z}}il\partial_r(\breve{w}_l\breve{\varphi}_{k-l})-ik\sum_{l\in\mathbb{Z}}\breve{\varphi}_{k-l}\partial_r\breve{w}_{l}]=0.
\end{align}

Next we give the explicit forms of nonlinear terms. We first write
\begin{align*}
&\f{f^{-1}}{r}[\sum_{l\in\mathbb{Z}}il\partial_r(\breve{w}_l\breve{\varphi}_{k-l})-ik\sum_{l\in\mathbb{Z}}\breve{\varphi}_{k-l}\partial_r\breve{w}_{l}]\\
=&\f{f^{-1}}{r}[\sum_{l\in\mathbb{Z}}il\partial_r(\breve{w}_l\breve{\varphi}_{k-l})-ik\sum_{l\in\mathbb{Z}}\partial_r(\breve{\varphi}_{k-l}\breve{w}_{l})+ik\sum_{l\in\mathbb{Z}}\breve{w}_{l}\partial_r\breve{\varphi}_{k-l}]\\
=&\f{f^{-1}}{r}[ik\sum_{l\in\mathbb{Z}}\breve{w}_{l}\partial_r\breve{\varphi}_{k-l}-\sum_{l\in\mathbb{Z}}i(k-l)\partial_r(\breve{w}_l\breve{\varphi}_{k-l})],
\end{align*}
Observing and utilizing below basic equalites
\begin{align*}
&\f{f^{-1}}{r}\partial_r(\breve{w}_l\breve{\varphi}_{k-l})=\partial_r(\f{f^{-1}}{\tilde{r}}\breve{w}_l\breve{\varphi}_{k-l})-\partial_r(\f{f^{-1}}{\tilde{r}})\breve{w}_l\breve{\varphi}_{k-l},\quad\partial_rf=-(\f12\f{1}{r}+\f14r)f,\\
&\partial_r(\f{f^{-1}}{r})=-\f{f'}{f^2}\f{1}{r}-\f{f^{-1}}{r^2}=(\f12\f{1}{r}+\f14r)f^{-1}\f{1}{r}-\f{f^{-1}}{r^2}=(\f14-\f12\f{1}{r^2})f^{-1},
\end{align*}
we obtain
\begin{align*}
&\f{f^{-1}}{r}[\sum_{l\in\mathbb{Z}}il\partial_r(\breve{w}_l\breve{\varphi}_{k-l})-ik\sum_{l\in\mathbb{Z}}\breve{\varphi}_{k-l}\partial_r\breve{w}_{l}]\\
=&\f{f^{-1}}{r}[ik\sum_{l\in\mathbb{Z}}\breve{w}_{l}\partial_r\breve{\varphi}_{k-l}-\sum_{l\in\mathbb{Z}}i(k-l)\partial_r(\breve{w}_l\breve{\varphi}_{k-l})]\\
=&[ik\sum_{l\in\mathbb{Z}}w_{l}\f{\partial_r\breve{\varphi}_{k-l}}{r}-\sum_{l\in\mathbb{Z}}i(k-l)(\partial_r(\f{w_l\breve{\varphi}_{k-l}}{r})-(\f14-\f12\f{1}{r^2})w_l\breve{\varphi}_{k-l})]\\
=&[ik\sum_{l\in\mathbb{Z}}w_{l}\f{\partial_r\breve{\varphi}_{k-l}}{r}+\sum_{l\in\mathbb{Z}}i(k-l)(\f14-\f12\f{1}{r^2})w_l\breve{\varphi}_{k-l}-\sum_{l\in\mathbb{Z}}i(k-l)\partial_r(\f{w_l\breve{\varphi}_{k-l}}{r})].
\end{align*}
We denote 
\begin{align*}
&f_1:=ik\sum_{l\in\mathbb{Z}}w_{l}\f{\partial_r\breve{\varphi}_{k-l}}{r}+\sum_{l\in\mathbb{Z}}i(k-l)(\f14-\f12\f{1}{r^2})w_l\breve{\varphi}_{k-l},\\
&f_2:=\sum_{l\in\mathbb{Z}}i(k-l)\f{w_l\breve{\varphi}_{k-l}}{r}.
\end{align*}
Then the nonlinear perturbation equation \eqref{scaling nonlinear} can be written as
\begin{align}
\label{main equation}\left\{
\begin{aligned}
&\partial_{\tau}w_k+L_kw_k+f_1-\partial_{r}f_2=0,\\
&w_k(0)=w_k|_{\tau=0},\quad w_k|_{\partial \Gamma}=0 
\end{aligned}
\right.
\end{align}
with $\partial\Gamma$ being the boundaries of the viscous fluid. Here
\begin{align*}
L_k=- [\partial_r^2-(\f{k^2-\f14}{r^2}+\f{r^2}{16}-\f12)]+i\f{kB}{r^2}.
\end{align*}

\subsection{Strategy and structure of the paper}
In Section \ref{2-resolvnet}, we establish the resolvent estimate for the linearized equation
\begin{align*}
&- [\partial_r^2-(\f{k^2-\f14}{r^2}+\f{r^2}{16}-\f12)]w+i\beta_k(\f{1}{r^2}-\lambda)w=F,\quad w|_{r=0}=w|_{r=\infty}=0
\end{align*}
 in both $L^2$ space and weighted $L^2$ space $X$. Here $X$ is defined in \eqref{weighted space}.
The resolvent estimate is inspired by \cite{LWZ-O}. A key difference to \cite{LWZ-O} is that, for here, not only we give the resolvent estimate from $L^2$ to $L^2$, we also establish the resolvent estimate from $L^2$ to $H^{-1}$. And we further generalize and extend these resolvent estimates to constructed weighted spaces as well. Our estimates in weighted spaces enable us to sharpen previous results and get the desired nonlinear theorem. In later sections to derive a sharp decaying estimate $e^{-c|kB|^{\f13}\tau}$ for the nonlinear problem, we also need to shift the linear operator to the left by $c_2|kB|^{\f13}$. This means that we establish bounds for $\|F-c_2|\beta_k|^{\f13}w\|_{L^2}$, $\|F-c_2|\beta_k|^{\f13}w\|_{H^{-1}}$,  $\|\f{F-c_2|\beta_k|^{\f13}w}{r}\|_{L^2}$ and $\|\f{F-c_2|\beta_k|^{\f13}w}{r}\|_{H^{-1}}$. Besides dodging a potential  $\log|B|$ loss, we also avoid proving the resolvent estimate between $\|w\|_{H^{-1}}$ and $\|F\|_{H^{-1}}$.

In Section \ref{3-space time}, we derive space-time estimates for the linearized Navier-Stokes equations. In the aim of applying these estimates to the nonlinear problem, we first study the following inhomogeneous equation \eqref{main equation} for $w_k$:
\begin{align*}
&\partial_{\tau}w_k+L_kw_k+f_1-\partial_{r}f_2=0,\quad w_k(0)=w_k|_{\tau=0}.
\end{align*}
We decompose $w_k$ into two parts with $w_k=w_k^l+w_k^{n}$. Here $w_k^l$ satisfies the below homogeneous linear equation \eqref{homogeneous linear equation} with initial data $w_k(0)$
\begin{align*}
&\partial_{\tau}w_k^l+L_kw_k^l=0,\quad w_k^l(0)=w_k(0),
\end{align*}
and $w_k^{n}$ verifies the inhomogeneous linear equation \eqref{nonzero mode nonlinear} with zero initial data
\begin{align*}
\partial_{\tau}w_k^{n}+L_kw_k^{n}+f_1-\partial_{r}f_2=0,\quad w_k^{n}(0)=0.
\end{align*}
We obtain the sharp bound for $w_k^l$ by using Gearhart-Pr$\ddot{u}$ss type theorem in \cite{Wei} to avoid $\log|B|$, since the linearized operator $L_k$ is accretive in both $L^2$ space and weighted $L^2$ space. Through Fourier transform and applying  proofs in Proposition \ref{resolvent estimate-1} and Proposition \ref{resolvent estimate-r1-1}, we derive a sharp semigroup bound for $w_k^n$ as well. Putting together, our space-time estimates take the below forms:
\begin{align*}
M(L^2):=&\|e^{c|kB|^{\f13}\tau}w_k\|_{L^{\infty}L^2}+|kB|^{\f16}\boxed{\|e^{c|kB|^{\f13}\tau}w_k\|_{L^2L^2}}\\
\nonumber&+\|e^{c|kB|^{\f13}\tau}\partial_{r}w_k\|_{L^2L^2}+\|e^{c|kB|^{\f13}\tau}(\f{|k|}{r}+r)w_k\|_{L^2L^2}\\
\nonumber\leq & C\Big(|kB|^{-\f16}\|e^{c|kB|^{\f13}\tau}f_1\|_{L^2L^2}+\|e^{c|kB|^{\f13}\tau}f_2\|_{L^2L^2}+\|w_k(0)\|_{L^2}\Big)
\end{align*}
and
\begin{align*}
M(X):=&\|e^{c|kB|^{\f13}\tau}w_k\|_{L^{\infty}X}+|kB|^{\f16}\|e^{c|kB|^{\f13}\tau}w_k\|_{L^2X}\\
\nonumber&+\|e^{c|kB|^{\f13}\tau}\partial_{r}w_k\|_{L^2X}+|k|\|e^{c|kB|^{\f13}\tau}\f{w_k}{r}\|_{L^2X}+\boxed{\|e^{c|kB|^{\f13}\tau}rw_k\|_{L^2X}}\\
\nonumber\leq&C\Big(|kB|^{-\f16}(\|e^{c|kB|^{\f13}\tau}f_1\|_{L^2X}+\|e^{c|kB|^{\f13}\tau}\f{f_2}{r}\|_{L^2X})+\|e^{c|kB|^{\f13}\tau}f_2\|_{L^2X}+\|w_k(0)\|_{X}\Big).
\end{align*}
The terms singled out in $M(L^2)$ and $M(X)$ will help us to design the correct ``energy norm'' $E_k$ used in Section \ref{4-Nonlinear stability}.

In Section \ref{4-Nonlinear stability}, we control the nonlinear terms in a sharpen way and employ a designed weighted energy space. Recall the expression of $M(L^2)$ and $M(X)$. In $M(X)$,  the coefficient of $\|e^{c|kB|^{\f13}\tau}rw_k\|_{L^2X}=\|e^{c|kB|^{\f13}\tau}w_k\|_{L^2L^2}$ is $1$, and in $M(L^2)$ the coefficient of $\|e^{c|kB|^{\f13}\tau}w_k\|_{L^2L^2}$ is $|kB|^{\f16}$. This suggests that a desired energy norm should be of order $M(L^2)+|kB|^{\f16}M(X)$. And it motivates us to couple $L^2$ and $X$ spaces together to construct the combined energy $E_k(k\in\mathbb{Z}/\{0\})$ as below
\begin{align*}
E_k:=&\|e^{c|kB|^{\f13}\tau}w_k\|_{L^{\infty}L^2}+|kB|^{\f16}\Big(\|e^{c|kB|^{\f13}\tau}w_k\|_{L^2L^2}+\|e^{c|kB|^{\f13}\tau}\f{w_k}{r}\|_{L^{\infty}L^2}\\
\nonumber &+|k|\|e^{c|kB|^{\f13}\tau}\f{w_k}{r^2}\|_{L^2L^2}+|k|^{\f12}\|e^{c|kB|^{\f13}\tau}\f{w_k}{r^{\f32}}\|_{L^2L^{\infty}}\Big)
+|kB|^{\f13}\|e^{c|kB|^{\f13}\tau}\f{w_k}{r}\|_{L^2L^2}.
\end{align*}
With these $E_k$, we prove our main conclusion-Theorem \ref{Nonlinear stability}.

Note that the last term $|kB|^{\f13}\|e^{c|kB|^{\f13}\tau}\f{w_k}{r}\|_{L^2L^2}$ in $E_k$ plays a crucial role when estimating nonlinear terms. It is because in $E_k$ we utilize $|kB|^{\f13}\|e^{c|kB|^{\f13}\tau}\f{w_k}{r}\|_{L^2L^2}$ rather than $|kB|^{\f16}\|e^{c|kB|^{\f13}\tau}w_k\|_{L^2L^2}$ to control nonlinear terms, that we allow larger $|B|^\f13$ (rather than $|B|^\f16$) perturbation in initial data. Furthermore, our argument is acompanied by including a  
$|kB|^{\f16}|k|^{\f12}\|e^{c|kB|^{\f13}\tau}\f{w_k}{r^{\f32}}\|_{L^2L^{\infty}}$ term in $E_k$, which helps to control weights in $k$.

\subsection{Other related results}
Here we list some other related works about enhanced dissipation and nonlinear asymptotic stability in 2D.

For 2D Navier-Stokes, enhanced dissipation effect was proved by  Beck-Wayne in \cite{BW} for a passive scalar advected by the Kolmogorov flow. The case of $u=(u(y), 0)$ with a finite number of critical points was treated by Bedrossian-Zelati in \cite{BZ}.  The case of the Kolmogorov flow was analyzed by Lin-Xu in \cite{LX}, Ibrahim-Maekawa-Masmoudi in \cite{IMM} and Wei-Zhang-Zhao in \cite{WZZ-NS}, and they obtained a time-scale of order $O(\nu^{-\f12})$. In \cite{ZDE}, Zelati-Delgadino-Elgindi explored its connection to mixing effect. In \cite{CWZ-2D-ided}, for the linearized Navier-Stokes system around monotone shear flows with non-slip condition, Chen-Wei-Zhang further developed estimates allowing the couple of the enhanced dissipation and the inviscid-damping effects.

For the 2D Boussinesq equations, in \cite{DWZ} Deng-Wu-Zhang proved enhanced dissipation via studying a non-self-adjoint operator. We also want to refer to a result achieving to establish nonlinear inviscid damping for 2D Euler equations by Ionescu-Jia \cite{IJ-NonE}. There they proved nonlinear asymptotic stability of monotonic shear flows in the channel.

\subsection{Acknowledgments}
The authors would like to thank Professor Zhifei Zhang for valueable discussions and correspondences.
\bigskip

\section{\textbf{Resolvent estimates in $L^2$ space and in weighted $L^2$ space}}	\label{2-resolvnet}
In order to establish decays of the linearized equation \eqref{main equation}, the first step is to study the resolvent equation under the following boundary condition in both $L^2$ space and weighted $L^2$ space. More precisely, we consider
\begin{align}\label{vorticity eqn}
&- [\partial_r^2-(\f{k^2-\f14}{r^2}+\f{r^2}{16}-\f12)]w+i\beta_k(\f{1}{r^2}-\lambda)w=F \quad \textrm{with} \ w|_{r=0,\infty}=0,
\end{align}
where $\lambda\in\mathbb{R}$. The domain is defined as
\begin{align*}
D_k=\{w\in H_{loc}^2(\mathbb{R}_{+},dr)\cap L^2(\mathbb{R}_{+},dr):- [\partial_r^2-(\f{k^2-\f14}{r^2}+\f{r^2}{16}-\f12)]w+i\f{\beta_k}{r^2}w\in L^2(\mathbb{R}_{+},dr) \}.
\end{align*}
Note that for any $|k|\geq1$, it holds
\begin{align*}
D_k&=\{w\in L^2(\mathbb{R}_{+},dr):\partial_r^2w,\f{w}{r^2},r^2w\in L^2(\mathbb{R}_{+},dr) \}.
\end{align*}

We start with deriving the coercive estimates for $\Re \langle F, w \rangle$ and $\Re \langle F, \f{w}{r^2}\rangle$, which are the real part of $\langle F, w \rangle$ and $\langle F, \f{w}{r^2}\rangle$, respectively.

\subsection{Coercive estimates of the real part}
\begin{lemma}\label{trivial w' lemma}For any $|k|\geq1$ and $w\in D_k$, it holds
  \begin{align}
\label{trivial w'}&\Re\langle F,w\rangle\gtrsim\|w'\|_{L^2}^2+\langle(\f{k^2}{r^2}+r^2)w,w\rangle.
\end{align}
\end{lemma}
\begin{proof}
We prove this lemma by considering the following two scenarios:
\begin{itemize}
  \item \textbf{When $|k|=1$}, one can check
\begin{align*}
-[\partial_r^2-(\f{k^2-\f14}{r^2}+\f{r^2}{16}-\f12)]w=-h^{-1}\partial_r[h^2\partial_r(h^{-1}w)]+\f12w,
\end{align*}
where $h=r^{\f32}e^{-\f{r^2}{8}}$. Via integration by parts we obtain
\begin{align*}
&\Re\langle F,w\rangle=\Re\langle -[\partial_r^2-(\f{k^2-\f14}{r^2}+\f{r^2}{16}-\f12)]w,w\rangle=\|h\partial_r(h^{-1}w)\|_{L^2}^2+\f12\|w\|_{L^2}^2,
\end{align*}
which implies
  \begin{align}
\label{PP1}&\Re\langle F,w\rangle\geq\f12\|w\|_{L^2}^2.
\end{align}
We also have
 \begin{align*}
&\Re\langle F,w\rangle=\|r^{-\f12}\partial_r(r^{\f12}w)\|_{L^2}^2+\langle(\f{1}{r^2}+\f{r^2}{16}-\f12)w,w\rangle\geq \langle(\f{1}{r^2}+\f{r^2}{16}-\f12)w,w\rangle.
\end{align*}
Adding these two inequalities, we deduce
\begin{align}
\label{Le1}2\Re\langle F,w\rangle\geq\langle(\f{1}{r^2}+\f{r^2}{16})w,w\rangle\gtrsim\langle(\f{1}{r^2}+r^2)w,w\rangle.
\end{align}
Applying integration by parts one also obtains
  \begin{align*}
\Re\langle F,w\rangle=&\|w'\|_{L^2}^2+\langle(\f{k^2-\f14}{r^2}+\f{r^2}{16}-\f12)w,w\rangle \\
\geq& \|w'\|_{L^2}^2-\langle(\f{1}{4r^2}+\f12)w,w\rangle.
\end{align*}
Thus we get
 \begin{align*}
 \|w'\|_{L^2}^2\leq\langle(\f{1}{4r^2}+\f12)w,w\rangle+\Re\langle F,w\rangle\lesssim\Re\langle F,w\rangle.
 \end{align*}
 Combining with \eqref{Le1}, we arrive at
  \begin{align*}
 \Re\langle F,w\rangle\gtrsim\|w'\|_{L^2}^2+\langle(\f{1}{r^2}+r^2)w,w\rangle.
 \end{align*}
  \item \textbf{When $|k|\geq2$}, observing
\begin{align*}
\f{k^2-\f14}{r^2}+\f{r^2}{16}-\f12\geq\f23\f{k^2}{r^2}+\f{r^2}{16}-\f12\gtrsim\f{k^2}{r^2}+r^2,
\end{align*}
and via employing integration by parts again, we conclude
  \begin{align*}
&\Re\langle F,w\rangle=\|w'\|_{L^2}^2+\langle(\f{k^2-\f14}{r^2}+\f{r^2}{16}-\f12)w,w\rangle\gtrsim\|w'\|_{L^2}^2+\langle(\f{k^2}{r^2}+r^2)w,w\rangle.
\end{align*}
\end{itemize}

This completes the proof of Lemma \ref{trivial w' lemma}.
\end{proof}

With Lemma \ref{trivial w' lemma}, we then can establish
\begin{lemma}\label{trivial w'-h1-2.2}For any $|k|\geq1$ and $w\in D_k$, it holds
  \begin{align}
  \label{trivial w'-h1-0}&|k|^{\f12}\|w\|_{H^1}+|k|\|w\|_{L^2}\lesssim\|F\|_{L^2},\\
\label{trivial w'-h1}&\|w\|_{H^1}+|k|^{\f12}\|w\|_{L^2}\lesssim\|F\|_{H^{-1}}.
\end{align}
\end{lemma}
\begin{proof}By Lemma \ref{trivial w' lemma}, one obtains
\begin{align*}
\|w'\|_{L^2}^2+|k|\|w\|_{L^2}^2\lesssim\Re\langle F,w\rangle,
\end{align*}
which implies
\begin{align*}
|k|\|w\|_{L^2}\lesssim\|F\|_{L^2}.
\end{align*}
Therefore, we deduce
\begin{align*}
\|w\|_{H^1}^2\lesssim\|w'\|_{L^2}^2+|k|\|w\|_{L^2}^2\lesssim\Re\langle F,w\rangle\lesssim|k|^{-1}\|F\|_{L^2}^2.
\end{align*}
This completes the proof of \eqref{trivial w'-h1-0}.

To prove (\ref{trivial w'-h1}), we apply Lemma \ref{trivial w' lemma} again
\begin{align*}
\|w\|_{H^1}^2\lesssim\|w'\|_{L^2}^2+|k|\|w\|_{L^2}^2\lesssim\Re\langle F,w\rangle.
\end{align*}
This renders
\begin{align*}
\|w\|_{H^1}\lesssim\|F\|_{H^{-1}}.
\end{align*}
Thus it follows
\begin{align*}
|k|\|w\|_{L^2}^2\lesssim\Re\langle F,w\rangle\lesssim\|F\|_{H^{-1}}^2.
\end{align*}
This completes the proof of \eqref{trivial w'-h1}.
\end{proof}

\begin{lemma}\label{trivial w'/r}For any $|k|\geq1$ and $w\in D_k$, it holds
  \begin{align}
\label{trivial w'/r1-0}&\Re\langle F,\f{w}{r^2}\rangle=\|r^{-1}h\partial_r(h^{-1}w)\|_{L^2}^2+(k^2-1)\|\f{w}{r^2}\|_{L^2}^2,
\end{align}
where $h=r^{\f32}e^{-\f{r^2}{8}}$. Moreover, we have the following estimate
  \begin{align}
\label{trivial w'/r1}&\Re\langle F,\f{w}{r^2}\rangle\gtrsim\|\f{w'}{r}\|_{L^2}+k^2\|\f{w}{r^2}\|_{L^2}^2+|k|\|\f{w}{r}\|_{L^2}^2+\|w\|_{L^2}^2.
\end{align}
\end{lemma}
\begin{proof}We start with proving \eqref{trivial w'/r1-0}. Denote $h=r^{\f32}e^{-\f{r^2}{8}}$. One can check
\begin{align*}
-[\partial_r^2-(\f{k^2-\f14}{r^2}+\f{r^2}{16}-\f12)]w=-r^2h^{-1}\partial_r[r^{-2}h^2\partial_r(h^{-1}w)]-\f{2}{r}w'+\f{k^2+2}{r^2}w.
\end{align*}
The identities below follow from integration by parts
\begin{align*}
\Re\langle F,\f{w}{r^2}\rangle=&\Re\langle -[\partial_r^2-(\f{k^2-\f14}{r^2}+\f{r^2}{16}-\f12)]w,\f{w}{r^2}\rangle\\=&\|r^{-1}h\partial_r(h^{-1}w)\|_{L^2}^2-\langle w',\f{2w}{r^3}\rangle+ \langle\f{k^2+2}{r^2}w,\f{w}{r^2}\rangle \\
=&\|r^{-1}h\partial_r(h^{-1}w)\|_{L^2}^2+(k^2-1)\|\f{w}{r^2}\|_{L^2}^2,
\end{align*}
where in the last equality we use
\begin{align}\label{basic equality-1}
-\Re\langle w',\f{2w}{r^3}\rangle=-\int_0^{\infty}\f{1}{r^3}d|w|^2=-\int_0^{\infty}\f{3|w|^2}{r^4}dr=-3\|\f{w}{r^2}\|_{L^2}^2.
\end{align}

Now we continue to prove \eqref{trivial w'/r1}. We first observe the following equality through integration by parts
\begin{align*}
&\langle -[\partial_r^2-(\f{k^2-\f14}{r^2}+\f{r^2}{16}-\f12)]w,\f{w}{r^2}\rangle=\|\f{w'}{r}\|_{L^2}^2-\langle w',\f{2w}{r^3}\rangle+ \langle(\f{k^2-\f14}{r^2}+\f{r^2}{16}-\f12)w,\f{w}{r^2}\rangle.
\end{align*}
Together with \eqref{basic equality-1} this gives
\begin{align}
\label{lemma-3-111}&\Re\langle F,\f{w}{r^2}\rangle=\|\f{w'}{r}\|_{L^2}^2+ \langle(\f{k^2-\f{13}{4}}{r^2}+\f{r^2}{16}-\f12)w,\f{w}{r^2}\rangle.
\end{align}

Then we treat $|k|\geq 2$ and $|k|=1$ separately.
\begin{itemize}
  \item \textbf{When $|k|\geq2$}, applying \eqref{trivial w'/r1-0}, we deduce
\begin{align}
\label{lemma-3-1}\Re\langle F,\f{w}{r^2}\rangle\geq(k^2-1)\|\f{w}{r^2}\|_{L^2}^2\gtrsim k^2\|\f{w}{r^2}\|_{L^2}^2.
\end{align}

 Combining with \eqref{lemma-3-1}, we hence obtain
\begin{align*}
&\Re\langle F,\f{w}{r^2}\rangle\gtrsim\|\f{w'}{r}\|_{L^2}^2+ \langle(\f{k^2}{r^2}+\f{r^2}{16}+|k|)w,\f{w}{r^2}\rangle.
\end{align*}

  \item \textbf{When $|k|=1$}, we write
\begin{align*}
-w''=-g^{-1}\partial_r(r^2g^2\partial_r(r^{-2}g^{-1}w))-\f{2}{r}w'-g^{-1}[g''+(\f{2}{r}g)']w
\end{align*}
for any function $g$. Let $g$ be a real function satisfying
\begin{align}\label{dif eqn for g}
\f{g''}{g}+\f{2}{r}\f{g'}{g}=-\f{2}{r^2}.
\end{align} The existence and explicit form of $g$ are guaranteed by Lemma \ref{Appendix real function g} in the Appendix. With \eqref{basic equality-1} and \eqref{dif eqn for g}, we deduce
\begin{align*}
\Re\langle F,\f{w}{r^2}\rangle=&\|rg\partial_r(r^{-2}g^{-1}w)\|_{L^2}^2-3\|\f{w}{r^2}\|_{L^2}^2-\langle g^{-1}[g''+(\f{2}{r}g)']w,\f{w}{r^2}\rangle\\
&+\langle (\f{k^2-\f14}{r^2}+\f{r^2}{16}-\f12)w,\f{w}{r^2}\rangle\\
=&\|rg\partial_r(r^{-2}g^{-1}w)\|_{L^2}^2-\langle g^{-1}(g''+\f{2}{r}g')w,\f{w}{r^2}\rangle+\langle (\f{k^2-\f54}{r^2}+\f{r^2}{16}-\f12)w,\f{w}{r^2}\rangle \\
=&\|rg\partial_r(r^{-2}g^{-1}w)\|_{L^2}^2+\langle (\f{7}{4r^2}+\f{r^2}{16}-\f12)w,\f{w}{r^2}\rangle\gtrsim\langle(\f{1}{r^2}+r^2)w,\f{w}{r^2}\rangle.
\end{align*}
Together with \eqref{lemma-3-111}, we prove
\begin{align*}
\Re\langle F,\f{w}{r^2}\rangle\gtrsim\|\f{w'}{r}\|_{L^2}^2+\langle(\f{1}{r^2}+r^2)w,\f{w}{r^2}\rangle.
\end{align*}
\end{itemize}

This completes the proof of Lemma \ref{trivial w'/r}.
\end{proof}

In a similar fashion as proving Lemma \ref{trivial w'/r}, we also obtain

\begin{lemma}\label{trivial w'/r-h1-2.4}For any $|k|\geq1$ and $w\in D_k$, it holds
  \begin{align}
  \label{trivial w'/r-h1-0}&|k|^{\f12}\|\f{w}{r}\|_{H^1}+|k|\|\f{w}{r}\|_{L^2}\lesssim\|\f{F}{r}\|_{L^2},\\
\label{trivial w'/r-h1}&\|\f{w}{r}\|_{H^1}+|k|^{\f12}\|\f{w}{r}\|_{L^2}\lesssim\|\f{F}{r}\|_{H^{-1}}.
\end{align}
\end{lemma}
\begin{proof}Applying Lemma \ref{trivial w'/r}, one has
\begin{align*}
\|\f{w'}{r}\|_{L^2}^2+|k|\|\f{w}{r}\|_{L^2}^2\lesssim\Re\langle \f{F}{r},\f{w}{r}\rangle,
\end{align*}
which yields
\begin{align*}
|k|\|\f{w}{r}\|_{L^2}\lesssim\|\f{F}{r}\|_{L^2}.
\end{align*}
Therefore, we deduce
\begin{align*}
\|\f{w}{r}\|_{H^1}^2\lesssim\|\f{w'}{r}\|_{L^2}^2+|k|\|\f{w}{r}\|_{L^2}^2\lesssim\Re\langle \f{F}{r},\f{w}{r}\rangle\lesssim|k|^{-1}\|\f{F}{r}\|_{L^2}^2.
\end{align*}
This completes the proof of \eqref{trivial w'/r-h1-0}.

Using Lemma \ref{trivial w'/r} again, we obtain
\begin{align*}
\|\f{w}{r}\|_{H^1}^2\lesssim\|\f{w'}{r}\|_{L^2}^2+|k|\|\f{w}{r}\|_{L^2}^2\lesssim\Re\langle \f{F}{r},\f{w}{r}\rangle,
\end{align*}
which gives
\begin{align*}
\|\f{w}{r}\|_{H^1}\lesssim\|\f{F}{r}\|_{H^{-1}}.
\end{align*}
Therefore, we can conclude
\begin{align*}
|k|\|\f{w}{r}\|_{L^2}^2\lesssim\Re\langle \f{F}{r},\f{w}{r}\rangle\lesssim\|\f{F}{r}\|_{H^{-1}}^2.
\end{align*}
This completes the proof of \eqref{trivial w'/r-h1}.
\end{proof}

\subsection{Resolvent estimates in $L^2$ space} We first derive the resolvent estimates from $L^2$ to $L^2$.
\begin{proposition}\label{resolvent estimate}For any $|k|\geq1$, $\lambda\in\mathbb{R}$ and $w\in D_k$, there exists a constant $C>0$ independent of $k,\beta_k,\lambda$, such that the following estimate holds
\begin{align*}
|\beta_k|^{\f16}\|w'\|_{L^2}+|\beta_k|^{\f13}\|w\|_{L^2}\leq C \|F\|_{L^2}.
\end{align*}
\end{proposition}
\begin{proof}
Let us first prove $|\beta_k|^{\f13}\|w\|_{L^2}\leq C \|F\|_{L^2}$. If $\lambda\leq0$, with Lemma \ref{trivial w' lemma} we have
  \begin{align*}
&\Re\langle F,w\rangle\gtrsim\langle(\f{k^2}{r^2}+r^2)w,w\rangle.
\end{align*}
Together with
\begin{align*}
&|\Im\langle F,w\rangle|\geq|\beta_k|\langle\f{1}{r^2}w,w\rangle,
\end{align*}
this gives
  \begin{align*}
&|\langle F,w\rangle|\gtrsim\langle(r^2+\f{|\beta_k|}{r^2})w,w\rangle\geq|\beta_k|^{\f12}\|w\|_{L^2}^2.
\end{align*}
Recall that $\beta_k=kB$ and $|B|\geq 1$. Thus $|\beta_k|=|kB|\geq 1$ for $|k|\geq 1$, and we have
 \begin{align}
\label{k.1}&|\langle F,w\rangle|\gtrsim|\beta_k|^{\f12}\|w\|_{L^2}^2 \geq |\beta_k|^{\f13}\|w\|_{L^2}^2.
\end{align}

If $\lambda>0$ and $|\beta_k|\leq k^2$, utilizing Lemma \ref{trivial w' lemma} again, one obtains
  \begin{align}
\label{k.2}|\langle F,w\rangle|\geq&\Re\langle F,w\rangle\gtrsim\langle(\f{k^2}{r^2}+r^2)w,w\rangle\geq|k|\|w\|_{L^2}^2\geq|\beta_k|^{\f12}\|w\|_{L^2}^2\geq |\beta_k|^{\f13}\|w\|_{L^2}^2.
\end{align}

It remains to deal with the scenario $\lambda>0$ and $|\beta_k|\geq k^2$. Denote $r_0:=\f{1}{\sqrt{\lambda}}$. The discussion can be divided into the below three cases:
\begin{enumerate}
\item \textbf{Case of $|\beta_k|\leq r_0^4$.}
With Lemma \ref{trivial w' lemma} we have
  \begin{align*}
&\Re\langle F,w\rangle\gtrsim\langle(\f{k^2}{r^2}+r^2)w,w\rangle.
\end{align*}
Combining with
\begin{align*}
&|\Im\langle F,w\rangle|\geq|\beta_k|\langle(\f{1}{r^2}-\f{1}{r_0^2})w,w\rangle,
\end{align*}
we deduce
  \begin{align*}
&|\langle F,w\rangle|\gtrsim\langle[(\f{k^2}{r^2}+r^2)+|\beta_k|(\f{1}{r^2}-\f{1}{r_0^2})]w,w\rangle\geq\langle[r^2+|\beta_k|(\f{1}{r^2}-\f{1}{r_0^2})]w,w\rangle.
\end{align*}
Since
\begin{align*}
r^2+|\beta_k|(\f{1}{r^2}-\f{1}{r_0^2})\geq 2|\beta_k|^{\f12}-\f{|\beta_k|}{r_0^2}\geq|\beta_k|^{\f12},
\end{align*}
we then obtain
\begin{align}
\label{k.3} |\langle F,w\rangle|\gtrsim |\beta_k|^{\f12}\|w\|_{L^2}^2\geq  |\beta_k|^{\f13}\|w\|_{L^2}^2.
\end{align} 

     \item \textbf{Case of $r_0^4\leq |\beta_k|\leq r_0^6$.} The condition $|\beta_k|\geq 1$ implies that $r_0\geq 1$. Using Lemma \ref{trivial w' lemma}, we derive the estimate for $\|w\|_{L^2(\f{r_0}{2},\infty)}^2$
  \begin{align*}
&\Re\langle F,w\rangle\gtrsim\langle(\f{k^2}{r^2}+r^2)w,w\rangle\geq\|rw\|_{L^2}^2\geq \f{r_0^2}{4}\|w\|_{L^2(\f{r_0}{2},\infty)}^2.
\end{align*}

The estimate for $\|w\|_{L^2(0,\f{r_0}{2})}^2$ is more subtle. First, we choose $r_{-}\in(\f{r_0}{2}-\f{1}{2r_0},\f{r_0}{2})$ such that the following inequality holds 
\begin{align}
\label{r-{-}}|w'(r_{-})|^2\leq 2r_0\|w'\|_{L^2}^2.
\end{align}
Notice the imaginary part of $\langle F,w\chi_{(0,r_{-})}\rangle$ obeys
 \begin{align*}
\Im\langle F,w\chi_{(0,r_{-})}\rangle=&\Im\langle- [\partial_r^2w-(\f{k^2-\f14}{r^2}+\f{r^2}{16}-\f12)w]+i\beta_k(\f{1}{r^2}-\lambda)w,w\chi_{(0,r_{-})}\rangle\\
=&\Im\langle-\partial_r^2w+i\beta_k(\f{1}{r^2}-\lambda)w,w\chi_{(0,r_{-})}\rangle.
\end{align*}
We hence prove
\begin{align*}
|\beta_k|\langle(\f{1}{r^2}-\lambda)w,w\chi_{(0,r_{-})}\rangle\leq\|F\|_{L^2}\|w\|_{L^2}+|w'(r_{-})w(r_{-})|.
\end{align*}
Using $\lambda=\f{1}{r_0^2}$ and $r_{-}\in(\f{r_0}{2}-\f{1}{2r_0},\f{r_0}{2})$, we obtain
\begin{align*}
&\|F\|_{L^2}\|w\|_{L^2}+|w'(r_{-})w(r_{-})|\geq |\beta_k|\langle(\f{1}{r^2}-\f{1}{r_0^2})w,w\chi_{(0,r_{-})}\rangle\\\geq&|\beta_k|\langle(\f{4}{r_0^2}-\f{1}{r_0^2})w,w\chi_{(0,r_{-})}\rangle\geq\f{|\beta_k|}{r_0^2}\|w\|_{L^2(0,r_{-})}^2.
\end{align*}

Therefore, $\|w\|_{L^2}$ can be bounded as below
\begin{align*}
\|w\|_{L^2}^2=&\|w\|_{L^2(0,r_{-})}^2+\|w\|_{L^2(r_{-},\f{r_0}{2})}^2+\|w\|_{L^2(\f{r_0}{2},\infty)}^2\\
\lesssim&\f{r_0^2}{|\beta_k|}(\|F\|_{L^2}\|w\|_{L^2}+|w'(r_{-})w(r_{-})|)+\f{1}{r_0}\|w\|_{L^{\infty}}^2+\f{1}{r_0^2}\|F\|_{L^2}\|w\|_{L^2}.
\end{align*}
Together with \eqref{r-{-}}, Lemma \ref{Appendix A1-1} and Lemma \ref{trivial w' lemma}, we have
\begin{align*}
\|w\|_{L^2}^2\lesssim&\f{r_0^2}{|\beta_k|}(\|F\|_{L^2}\|w\|_{L^2}+r_0^{\f12}\|w'\|_{L^2}\|w\|_{L^{\infty}})+\f{1}{r_0}\|w\|_{L^{\infty}}^2+\f{1}{r_0^2}\|F\|_{L^2}\|w\|_{L^2}\\
\lesssim&\f{r_0^2}{|\beta_k|}(\|F\|_{L^2}\|w\|_{L^2}+r_0^{\f12}\|w'\|_{L^2}^{\f32}\|w\|_{L^2}^{\f12})+\f{1}{r_0}\|w'\|_{L^2}\|w\|_{L^2}+\f{1}{r_0^2}\|F\|_{L^2}\|w\|_{L^2}\\
\lesssim&\f{r_0^2}{|\beta_k|}(\|F\|_{L^2}\|w\|_{L^2}+r_0^{\f12}\|F\|_{L^2}^{\f34}\|w\|_{L^2}^{\f54})+\f{1}{r_0}\|F\|_{L^2}^{\f12}\|w\|_{L^2}^{\f32}+\f{1}{r_0^2}\|F\|_{L^2}\|w\|_{L^2}.
\end{align*}
This yields
\begin{align*}
\|F\|_{L^2}\|w\|_{L^2}\gtrsim\min\{\f{|\beta_k|}{r_0^2},(\f{|\beta_k|}{r_0^{\f52}})^{\f43},r_0^2\}\|w\|_{L^2}^2.
\end{align*}
With $r_0^4\leq |\beta_k|\leq r_0^6$, i.e. $\f{|\beta_k|^{\f43}}{r_0^{\f{10}{3}}}\geq\f{|\beta_k|}{r_0^2}\geq r_0^2$, we conclude
\begin{align}
\label{k.4}\|F\|_{L^2}\|w\|_{L^2}\gtrsim r_0^2\|w\|_{L^2}^2\geq|\beta_k|^{\f13}\|w\|_{L^2}^2.
\end{align}

\item \textbf{Case of $|\beta_k|\geq r_0^6$.} 
We choose $r_{-}\in(r_0-\delta,r_0-\f{\delta}{2})$ and $r_{+}\in(r_0+\f{\delta}{2},r_0+\delta)$ such that the following inequality holds
\begin{align}
\label{r-{-},r-{+}}|w'(r_{-})|^2+|w'(r_{+})|^2\leq \f{4}{\delta}\|w'\|_{L^2}^2.
\end{align}
Here $\delta\in(0, r_0)$ is a constant which will be determined later.

   The following equality is direct
 \begin{align*}
&\Im\langle F,w(\chi_{(0,r_{-})}-\chi_{(r_{+},\infty)})\rangle\\
=&\Im\langle- [\partial_r^2w-(\f{k^2-\f14}{r^2}+\f{r^2}{16}-\f12)w]+i\beta_k(\f{1}{r^2}-\lambda)w,w(\chi_{(0,r_{-})}-\chi_{(r_{+},\infty)})\rangle\\
=&\Im\langle-\partial_r^2w+i\beta_k(\f{1}{r^2}-\lambda)w,w(\chi_{(0,r_{-})}-\chi_{(r_{+},\infty)})\rangle.
\end{align*}
This implies
\begin{align*}
&|\beta_k|(\int_0^{r_{-}}(\f{1}{r^2}-\f{1}{r_0^2})|w(r)|^2dr+\int_{r_{+}}^{\infty}(\f{1}{r_0^2}-\f{1}{r^2})|w(r)|^2dr)\\
\leq&\|F\|_{L^2}\|w\|_{L^2}+|w'(r_{-})w(r_{-})|+|w'(r_{+})w(r_{+})|.
\end{align*}
 With $r_{-}\in(r_0-\delta,r_0-\f{\delta}{2})$ and $r_{+}\in(r_0+\f{\delta}{2},r_0+\delta)$, it holds
\begin{align*}
&\f{1}{r^2}-\f{1}{r_0^2}=\f{(r_0-r)(r_0+r)}{r^2r_0^2}\geq\f{r_0-r}{r^2r_0}\gtrsim\f{\delta}{r^2r_0}\gtrsim\f{\delta}{r_0^3} \quad\textrm{for any} \ r\in(0,r_{-}).
\end{align*}
In a similar fashion, one obtains
\begin{align*}
&\f{1}{r_0^2}-\f{1}{r^2}=\f{(r-r_0)(r+r_0)}{r^2r_0^2}\geq\f{r-r_0}{r^2r_0}\gtrsim\f{\delta}{r^2r_0}\gtrsim\f{\delta}{r_0^3} \quad \textrm{for any} \ r\in(r_{+},2r_0),\\
&\f{1}{r_0^2}-\f{1}{r^2}=\f{(r-r_0)(r+r_0)}{r^2r_0^2}\approx \f{1}{r_0^2}\gtrsim\f{\delta}{r_0^3}\quad \textrm{for any} \ r\geq2r_0.
\end{align*}
Combining these above inequalities, we get
\begin{align*}
&\|F\|_{L^2}\|w\|_{L^2}+|w'(r_{-})w(r_{-})|+|w'(r_{+})w(r_{+})|\\
\geq&|\beta_k|(\int_0^{r_{-}}(\f{1}{r^2}-\f{1}{r_0^2})|w(r)|^2dr+\int_{r_{+}}^{\infty}(\f{1}{r_0^2}-\f{1}{r^2})|w(r)|^2dr)\\
\gtrsim&\f{|\beta_k|\delta}{r_0^3}\|w\|_{L^2\big((0,r_{-})\cup(r_{+},\infty)\big)}^2.
\end{align*}
Therefore, $\|w\|_{L^2}$ obeys the estimate
\begin{align*}
\|w\|_{L^2}^2=&\|w\|_{L^2((0,r_{-})\cup(r_{+},\infty))}^2+\|w\|_{L^2(r_{-},r_{+})}^2\\
\lesssim&\f{r_0^3}{|\beta_k|\delta}\Big(\|F\|_{L^2}\|w\|_{L^2}+|w'(r_{-})w(r_{-})|+|w'(r_{+})w(r_{+})|\Big)+\delta\|w\|_{L^{\infty}}^2.
\end{align*}
Together with \eqref{r-{-},r-{+}}, Lemma \ref{Appendix A1-1} and Lemma \ref{trivial w' lemma}, we deduce
\begin{align*}
\|w\|_{L^2}^2\lesssim&\f{r_0^3}{|\beta_k|\delta}(\|F\|_{L^2}\|w\|_{L^2}+\f{\|w'\|_{L^2}\|w\|_{L^{\infty}}}{\delta^{\f12}})+\delta\|w\|_{L^{\infty}}^2\\
\lesssim&\f{r_0^3}{|\beta_k|\delta}(\|F\|_{L^2}\|w\|_{L^2}+\f{\|w'\|_{L^2}^{\f32}\|w\|_{L^2}^{\f12}}{\delta^{\f12}})+\delta\|w'\|_{L^2}\|w\|_{L^2}\\
\lesssim&\f{r_0^3}{|\beta_k|\delta}(\|F\|_{L^2}\|w\|_{L^2}+\f{\|F\|_{L^2}^{\f34}\|w\|_{L^2}^{\f54}}{\delta^{\f12}})+\delta\|F\|_{L^2}^{\f12}\|w\|_{L^2}^{\f32}.
\end{align*}
It then follows 
\begin{align*}
\|F\|_{L^2}\|w\|_{L^2}\gtrsim\min\{\f{|\beta_k|\delta}{r_0^3},(\f{|\beta_k|\delta^{\f32}}{r_0^3})^{\f43},\delta^{-2}\}\|w\|_{L^2}^2.
\end{align*}
Take $\delta=(\f{r_0^3}{|\beta_k|})^{\f13}$. The condition $|\beta_k|\geq 1$ yields
\begin{align*}
\delta=(\f{r_0^3}{|\beta_k|})^{\f13}\leq r_0.
\end{align*}
With the basic equality $\f{|\beta_k|\delta}{r_0^3}=(\f{|\beta_k|\delta^{\f32}}{r_0^3})^{\f43}=\delta^{-2}$, we obtain
\begin{align}
\label{k.5}\|F\|_{L^2}\|w\|_{L^2}\gtrsim (\f{|\beta_k|}{r_0^3})^{\f23}\|w\|_{L^2}^2=\f{|\beta_k|^{\f23}}{r_0^2}\|w\|_{L^2}^2\geq\f{|\beta_k|^{\f23}}{|\beta_k|^{\f13}}\|w\|_{L^2}^2=|\beta_k|^{\f13}\|w\|_{L^2}^2.
\end{align}

\end{enumerate}

Combining \eqref{k.1}, \eqref{k.2}, \eqref{k.3}, \eqref{k.4} and \eqref{k.5}, we now establish the following resolvent estimate from $L^2$ to $L^2$
\begin{align}
\label{k.6-}|\beta_k|^{\f13}\|w\|_{L^2}\lesssim\|F\|_{L^2}.
\end{align}
By Lemma \ref{trivial w' lemma} we have
  \begin{align*}
&\Re\langle F,w\rangle\gtrsim\|w'\|_{L^2}^2,
\end{align*}
which along with \eqref{k.6-} implies the resolvent estimate from $H^1$ to $L^2$
  \begin{align*}
\|w'\|_{L^2}^2\lesssim\|F\|_{L^2}\|w\|_{L^2}\lesssim|\beta_k|^{-\f13}\|F\|_{L^2}^2.
\end{align*}
This completes the proof of Proposition \ref{resolvent estimate}.
\end{proof}

Now we move to establish the resolvent estimates from $L^2$ to $H^{-1}$.

\begin{proposition}\label{resolvent estimate-L2}For any $|k|\geq1$, $\lambda\in\mathbb{R}$ and $w\in D_k$, there exist  constants $C,c_2>0$ independent of $k,\beta_k,\lambda$, such that the following estimate holds
\begin{align*}
\|w\|_{H^1}+|\beta_k|^{\f16}\|w\|_{L^2}\leq C \|F-c_2|\beta_k|^{\f13}w\|_{H^{-1}}.
\end{align*}
\end{proposition}
\begin{proof}From Lemma \ref{trivial w' lemma} we obtain
  \begin{align*}
&\Re\langle F,w\rangle=\Re\langle F-c_2|\beta_k|^{\f13}w,w\rangle+c_2|\beta_k|^{\f13}\|w\|_{L^2}^2\gtrsim\|w'\|_{L^2}^2+\langle(\f{k^2}{r^2}+r^2)w,w\rangle,
\end{align*}
with $c_2>0$ to be determined later. This gives
  \begin{align*}
\|w\|_{H^1}^2\lesssim\|F-c_2|\beta_k|^{\f13}w\|_{H^{-1}}\|w\|_{H^1}+c_2|\beta_k|^{\f13}\|w\|_{L^2}^2.
\end{align*}
Thus we have
  \begin{align}
\label{kkk.0}\|w\|_{H^1}\lesssim\|F-c_2|\beta_k|^{\f13}w\|_{H^{-1}}+\sqrt{c_2}|\beta_k|^{\f16}\|w\|_{L^2}.
\end{align}

To prove $|\beta_k|^{\f16}\|w\|_{L^2}\leq C \|F-c_2|\beta_k|^{\f13}w\|_{H^{-1}}$, we proceed in the same fashion as for proving Proposition \ref{resolvent estimate}. If $\lambda\leq 0$, using Lemma \ref{trivial w' lemma}, we have
  \begin{align*}
&\Re\langle F-c_2|\beta_k|^{\f13}w,w\rangle+c_2|\beta_k|^{\f13}\|w\|_{L^2}^2\gtrsim\langle(\f{k^2}{r^2}+r^2)w,w\rangle,
\end{align*}
Combining with
\begin{align*}
&|\Im\langle F-c_2|\beta_k|^{\f13}w,w\rangle|\geq|\beta_k|\langle\f{1}{r^2}w,w\rangle,
\end{align*}
and \eqref{kkk.0}, we get
  \begin{align*}
&\|F-c_2|\beta_k|^{\f13}w\|_{H^{-1}}^2+c_2|\beta_k|^{\f13}\|w\|_{L^2}^2\gtrsim\langle(r^2+\f{|\beta_k|}{r^2})w,w\rangle\geq|\beta_k|^{\f12}\|w\|_{L^2}^2.
\end{align*}
It follows that
  \begin{align*}
&C(\|F-c_2|\beta_k|^{\f13}w\|_{H^{-1}}^2+c_2|\beta_k|^{\f13}\|w\|_{L^2}^2)\geq|\beta_k|^{\f12}\|w\|_{L^2}^2.
\end{align*}
Recall $|\beta_k|=|kB|\geq1$. Choose $c_2>0$ sufficiently small such that $Cc_2\leq\f12$, we can conclude
  \begin{align}
\label{kkk.1}&\|F-c_2|\beta_k|^{\f13}w\|_{H^{-1}}^2\gtrsim|\beta_k|^{\f12}\|w\|_{L^2}^2.
\end{align}

If $\lambda>0$ and $|\beta_k|\leq k^2$, by Lemma \ref{trivial w' lemma} and \eqref{kkk.0} we obtain
     \begin{align*}
&\|F-c_2|\beta_k|^{\f13}w\|_{H^{-1}}^2+c_2|\beta_k|^{\f13}\|w\|_{L^2}^2 \\\geq&\Re\langle F-c_2|\beta_k|^{\f13}w,w\rangle+c_2|\beta_k|^{\f13}\|w\|_{L^2}^2 \\\gtrsim&\langle(\f{k^2}{r^2}+r^2)w,w\rangle\geq|k|\|w\|_{L^2}^2\geq|\beta_k|^{\f12}\|w\|_{L^2}^2.
\end{align*}
Recall $|\beta_k|\geq1$. Choose $c_2>0$ being sufficiently small such that $Cc_2\leq\f12$, we can conclude
  \begin{align}
\label{kkk.2}&\|F-c_2|\beta_k|^{\f13}w\|_{H^{-1}}^2\gtrsim|\beta_k|^{\f12}\|w\|_{L^2}^2.
\end{align}

The remaining the scenario is when $\lambda>0$ and $|\beta_k|\geq k^2$. As the proof for Proposition \ref{resolvent estimate}, we consider the following three cases. Recall $r_0=\f{1}{\sqrt{\lambda}}$.

\begin{enumerate}
  \item \textbf{Case of $|\beta_k|\leq r_0^4$.} Utilizing Lemma \ref{trivial w' lemma} again, it holds
  \begin{align*}
\Re\langle F-c_2|\beta_k|^{\f13}w,w\rangle+c_2|\beta_k|^{\f13}\|w\|_{L^2}^2\gtrsim\langle(\f{k^2}{r^2}+r^2)w,w\rangle.
\end{align*}
Together with
\begin{align*}
&|\Im\langle F-c_2|\beta_k|^{\f13}w,w\rangle|\geq|\beta_k|\langle(\f{1}{r^2}-\f{1}{r_0^2})w,w\rangle,
\end{align*}
this implies
  \begin{align*}
&\|F-c_2|\beta_k|^{\f13}w\|_{H^{-1}}\|w\|_{H^1}+c_2|\beta_k|^{\f13}\|w\|_{L^2}^2\\
\gtrsim&\langle[(\f{k^2}{r^2}+r^2)+|\beta_k|(\f{1}{r^2}-\f{1}{r_0^2})]w,w\rangle\geq\langle[r^2+|\beta_k|(\f{1}{r^2}-\f{1}{r_0^2})]w,w\rangle.
\end{align*}
Since
\begin{align*}
r^2+|\beta_k|(\f{1}{r^2}-\f{1}{r_0^2})\geq 2|\beta_k|^{\f12}-\f{|\beta_k|}{r_0^2}\geq|\beta_k|^{\f12},
\end{align*}
we then deduce
 \begin{align*}
&C(\|F-c_2|\beta_k|^{\f13}w\|_{H^{-1}}^2+c_2|\beta_k|^{\f13}\|w\|_{L^2}^2)\gtrsim|\langle F,w\rangle|\gtrsim \langle g(r)w,w\rangle\geq|\beta_k|^{\f12}\|w\|_{L^2}^2.
\end{align*}
Recall $|\beta_k|\geq1$. Choose $c_2>0$ sufficiently small such that $Cc_2\leq\f12$, we can conclude
  \begin{align}
\label{kkk.3}&\|F-c_2|\beta_k|^{\f13}w\|_{H^{-1}}^2\gtrsim|\beta_k|^{\f12}\|w\|_{L^2}^2.
\end{align}
     \item \textbf{Case of $r_0^4\leq |\beta_k|\leq r_0^6$.} Recall $|\beta_k|\geq 1$, then we have $r_0\geq 1$. With Lemma \ref{trivial w' lemma}, we obtain the below estimate for $\|w\|_{L^2(\f{r_0}{2},\infty)}^2$ 
  \begin{align*}
&\Re\langle F-c_2|\beta_k|^{\f13}w,w\rangle+c_2|\beta_k|^{\f13}\|w\|_{L^2}^2\gtrsim\langle(\f{k^2}{r^2}+r^2)w,w\rangle\geq\|rw\|_{L^2}^2\geq \f{r_0^2}{4}\|w\|_{L^2(\f{r_0}{2},\infty)}^2.
\end{align*}
As in Proposition \ref{resolvent estimate}, the estimate for $\|w\|_{L^2(0,\f{r_0}{2})}^2$ is more subtle. First, we choose $r_{-}\in(\f{r_0}{2}-\f{1}{r_0},\f{r_0}{2})$ so that the following inequality holds
\begin{align}
\label{r-{-}3}|w'(r_{-})|^2\leq r_0\|w'\|_{L^2}^2.
\end{align}
Being slightly different from the proof of Proposition \ref{resolvent estimate}, we require the multiplier for $H^{-1}$ estimate to be of $C^1$ regularity. We introduce a cutoff function $\rho(r)$ as follows:
\begin{align*}
\rho(r)=\left\{
\begin{aligned}
&1,\quad r\in(0,\f{r_0}{2}-\f{2}{r_0}),\\
&\sin\big(\f{\pi}{2}\f{1}{r_{-}-(\f{r_0}{2}-\f{2}{r_0})}(r_{-}-r)\big),\quad r\in(\f{r_0}{2}-\f{2}{r_0},r_{-}),\\
&0,\quad r\geq r_{-}.
\end{aligned}
\right.
\end{align*}
Then we consider the following basic equality
 \begin{align*}
\Im\langle F-c_2|\beta_k|^{\f13}w,w\rho(r)\rangle=&\Im\langle- [\partial_r^2w-(\f{k^2-\f14}{r^2}+\f{r^2}{16}-\f12)w]+i\beta_k(\f{1}{r^2}-\lambda)w,w\rho(r)\rangle\\
=&\Im\langle-\partial_r^2w+i\beta_k(\f{1}{r^2}-\lambda)w,w\rho(r)\rangle.
\end{align*}
This gives
\begin{align*}
|\beta_k|\langle(\f{1}{r^2}-\lambda)w,w\rho(r)\rangle\lesssim\|F-c_2|\beta_k|^{\f13}w\|_{H^{-1}}\|\rho w\|_{H^1}+r_0\|w\|_{L^2}\|w'\|_{L^2}.
\end{align*}
Thus, with $\lambda=\f{1}{r_0^2}$ and $r_{-}\in(\f{r_0}{2}-\f{1}{r_0},\f{r_0}{2})$ we obtain
\begin{align*}
&\|F-c_2|\beta_k|^{\f13}w\|_{H^{-1}}\|\rho w\|_{H^1}+r_0\|w\|_{L^2}\|w'\|_{L^2}\\
\gtrsim&|\beta_k|\langle(\f{4}{r_0^2}-\f{1}{r_0^2})w,w\rho(r)\rangle\geq\f{|\beta_k|}{r_0^2}\|w\|_{L^2(0,\f{r_0}{2}-\f{2}{r_0})}^2.
\end{align*}
Since
\begin{align*}
\|\rho w\|_{H^1}\lesssim\|w\|_{H^1}+r_0\|w\|_{L^2},
\end{align*}
then $\|w\|_{L^2}$ can be bounded by
\begin{align*}
\|w\|_{L^2}^2=&\|w\|_{L^2(0,\f{r_0}{2}-\f{2}{r_0})}^2+\|w\|_{L^2(\f{r_0}{2}-\f{2}{r_0},\f{r_0}{2})}^2+\|w\|_{L^2(\f{r_0}{2},\infty)}^2\\
\lesssim&\f{r_0^2}{|\beta_k|}\|F-c_2|\beta_k|^{\f13}w\|_{H^{-1}}(\|w\|_{H^1}+r_0\|w\|_{L^2})+\f{1}{r_0}\|w\|_{L^{\infty}}^2\\
&+\f{1}{r_0^2}(\|F-c_2|\beta_k|^{\f13}w\|_{H^{-1}}^2+c_2|\beta_k|^{\f13}\|w\|_{L^2}^2).
\end{align*}
 Together with \eqref{r-{-}3}, Lemma \ref{Appendix A1-1} and  \eqref{kkk.0} this gives
\begin{align*}
\|w\|_{L^2}^2\lesssim&\f{r_0^2}{|\beta_k|}(\|F-c_2|\beta_k|^{\f13}w\|_{H^{-1}}\|w\|_{H^1}+r_0\|F-c_2|\beta_k|^{\f13}w\|_{H^{-1}}\|w\|_{L^2})+\f{1}{r_0}\|w\|_{L^{\infty}}^2\\
&+\f{1}{r_0^2}(\|F-c_2|\beta_k|^{\f13}w\|_{H^{-1}}^2+c_2|\beta_k|^{\f13}\|w\|_{L^2}^2)\\
\lesssim&\f{r_0^2}{|\beta_k|}(\|F-c_2|\beta_k|^{\f13}w\|_{H^{-1}}^2+c_2|\beta_k|^{\f13}\|w\|_{L^2}^2+r_0\|F-c_2|\beta_k|^{\f13}w\|_{H^{-1}}\|w\|_{L^2})\\
&+\f{1}{r_0}\|w'\|_{L^2}\|w\|_{L^2}+\f{1}{r_0^2}(\|F-c_2|\beta_k|^{\f13}w\|_{H^{-1}}^2+c_2|\beta_k|^{\f13}\|w\|_{L^2}^2)\\
\leq&C\Big(\f{r_0^2}{|\beta_k|}(\|F-c_2|\beta_k|^{\f13}w\|_{H^{-1}}^2+c_2|\beta_k|^{\f13}\|w\|_{L^2}^2+r_0\|F-c_2|\beta_k|^{\f13}w\|_{H^{-1}}\|w\|_{L^2})\\
&+\f{1}{r_0}(\|F-c_2|\beta_k|^{\f13}w\|_{H^{-1}}+\sqrt{c_2}|\beta_k|^{\f16}\|w\|_{L^2})\|w\|_{L^2}\\
&+\f{1}{r_0^2}(\|F-c_2|\beta_k|^{\f13}w\|_{H^{-1}}^2+c_2|\beta_k|^{\f13}\|w\|_{L^2}^2)\Big).
\end{align*}
Thus we obtain
\begin{align*}
C(c_2\f{r_0^2}{|\beta_k|^{\f23}}+\sqrt{c_2}\f{|\beta_k|^{\f16}}{r_0}+c_2\f{|\beta_k|^{\f13}}{r_0^2})\leq C(c_2|\beta_k|^{-\f16}+\sqrt{c_2}+c_2),
\end{align*}
Recall $|\beta_k|\geq1$. Choose $c_2>0$ satisfying
\begin{align*}
C(c_2|\beta_k|^{-\f16}+\sqrt{c_2}+c_2)\leq\f12,
\end{align*}
we can conclude
\begin{align*}
\|F-c_2|\beta_k|^{\f13}w\|_{H^{-1}}^2\gtrsim\min\{\f{|\beta_k|}{r_0^2},(\f{|\beta_k|}{r_0^3})^2,r_0^2\}\|w\|_{L^2}^2.
\end{align*}
Noting that $r_0^4\leq |\beta_k|\leq r_0^6$, i.e. $\f{|\beta_k|^{\f43}}{r_0^{\f{10}{3}}}\geq\f{|\beta_k|}{r_0^2}\geq r_0^2$, we hence arrive at
\begin{align}
\label{kkk.4}\|F-c_2|\beta_k|^{\f13}w\|_{H^{-1}}^2\gtrsim r_0^2\|w\|_{L^2}^2\geq|\beta_k|^{\f13}\|w\|_{L^2}^2.
\end{align}

     \item  \textbf{Case of $ |\beta_k|\geq r_0^6$.} Similarly, as in the previous case, we choose $r_{-}\in(r_0-\delta,r_0-\f{\delta}{2})$ and $r_{+}\in(r_0+\f{\delta}{2},r_0+\delta)$ such that the following inequality holds
\begin{align}
\label{r-{-},r-{+}3}|w'(r_{-})|^2+|w'(r_{+})|^2\leq \f{4}{\delta}\|w'\|_{L^2}^2,
\end{align}
Here $\delta>0$ is a constant which will be determined later. As in Proposition \ref{resolvent estimate}, we define a $C^1$ cutoff function $\rho$ with domain $(0, \infty)$ as follows
\begin{align*}
\rho(r)=\left\{
\begin{aligned}
&1,\quad r\in(0,r_0-2\delta),\\
&\sin\big(\f{\pi}{2}\f{1}{r_{-}-(r_0-2\delta)}(r_{-}-r)\big),\quad r\in(r_0-2\delta,r_{-}),\\
&0,\quad r\in (r_{-},r_{+}),\\
&\sin\big(\f{\pi}{2}\f{1}{r_0+2\delta-r_{+}}(r_{+}-r)\big),\quad r\in(r_{+},r_0+2\delta),\\
&-1,\quad r\geq r_0+2\delta.
\end{aligned}
\right.
\end{align*}
Observing
 \begin{align*}
\Im\langle F-c_2|\beta_k|^{\f13}w,w\rho(r)\rangle=&\Im\langle- [\partial_r^2w-(\f{k^2-\f14}{r^2}+\f{r^2}{16}-\f12)w]+i\beta_k(\f{1}{r^2}-\lambda)w,w\rho(r)\rangle\\
=&\Im\langle-\partial_r^2w+i\beta_k(\f{1}{r^2}-\lambda)w,w\rho(r)\rangle,
\end{align*}
we obtain
\begin{align*}
&|\beta_k|(\int_0^{r_0-2\delta}(\f{1}{r^2}-\f{1}{r_0^2})|w(r)|^2dr+\int_{r_0+2\delta}^{\infty}(\f{1}{r_0^2}-\f{1}{r^2})|w(r)|^2dr)\\
\lesssim&\|F-c_2|\beta_k|^{\f13}w\|_{H^{-1}}\|\rho w\|_{H^1}+\delta^{-1}\|w\|_{L^2}\|w'\|_{L^2}.
\end{align*}
With $r_{-}\in(r_0-\delta,r_0-\f{\delta}{2})$ and $r_{+}\in(r_0+\f{\delta}{2},r_0+\delta)$, it holds
\begin{align*}
&\f{1}{r^2}-\f{1}{r_0^2}=\f{(r_0-r)(r_0+r)}{r^2r_0^2}\geq\f{r_0-r}{r^2r_0}\gtrsim\f{\delta}{r^2r_0}\gtrsim\f{\delta}{r_0^3} \quad \textrm{for any} \ r\in(0,r_{-}).
\end{align*}
Likewise, we have
\begin{align*}
&\f{1}{r_0^2}-\f{1}{r^2}=\f{(r-r_0)(r+r_0)}{r^2r_0^2}\geq\f{r-r_0}{r^2r_0}\gtrsim\f{\delta}{r^2r_0}\gtrsim\f{\delta}{r_0^3},\quad \textrm{for any} \ r\in(r_{+},2r_0),\\
&\f{1}{r_0^2}-\f{1}{r^2}=\f{(r-r_0)(r+r_0)}{r^2r_0^2}\approx\f{1}{r_0^2}\gtrsim\f{\delta}{r_0^3},\quad \textrm{for any} \ r\geq2r_0.
\end{align*}
Plugging in all these above inequalities, we deduce
\begin{align*}
&\|F-c_2|\beta_k|^{\f13}w\|_{H^{-1}}\|\rho w\|_{H^1}+\delta^{-1}\|w\|_{L^2}\|w'\|_{L^2}\\
\geq&|\beta_k|(\int_0^{r_0-2\delta}(\f{1}{r^2}-\f{1}{r_0^2})|w(r)|^2dr+\int_{r_0+2\delta}^{\infty}(\f{1}{r_0^2}-\f{1}{r^2})|w(r)|^2dr)\\
\gtrsim&\f{|\beta_k|\delta}{r_0^3}\|w\|_{L^2\big((0,r_0-2\delta)\cup(r_0+2\delta,\infty)\big)}^2.
\end{align*}
Since
\begin{align*}
\|\rho w\|_{H^1}\lesssim\|w\|_{H^1}+\delta^{-1}\|w\|_{L^2},
\end{align*}
then $\|w\|_{L^2}$ can be bounded as follows
\begin{align*}
\|w\|_{L^2}^2=&\|w\|_{L^2\big((0,r_0-2\delta)\cup(r_0+2\delta,\infty)\big)}^2+\|w\|_{L^2(r_0-2\delta,r_0+2\delta)}^2\\
\lesssim&\f{r_0^3}{|\beta_k|\delta}(\|F-c_2|\beta_k|^{\f13}w\|_{H^{-1}}\|w\|_{H^1}+\delta^{-1}\|F-c_2|\beta_k|^{\f13}w\|_{H^{-1}}\|w\|_{L^2}\\
&+\delta^{-1}\|w'\|_{L^2}\|w\|_{L^2})+\delta\|w\|_{L^{\infty}}^2.
\end{align*}
Employing \eqref{r-{-},r-{+}3}, Lemma \ref{Appendix A1-1} and \eqref{kkk.0} we obtain
\begin{align*}
\|w\|_{L^2}^2\lesssim&\f{r_0^3}{|\beta_k|\delta}\Big(\|F-c_2|\beta_k|^{\f13}w\|_{H^{-1}}^2+c_2|\beta_k|^{\f13}\|w\|_{L^2}^2+\delta^{-1}\|F-c_2|\beta_k|^{\f13}w\|_{H^{-1}}\|w\|_{L^2}\\
&\delta^{-1}(\|F-c_2|\beta_k|^{\f13}w\|_{H^{-1}}+\sqrt{c_2}|\beta_k|^{\f16}\|w\|_{L^2})\|w\|_{L^2}\Big)+\delta\|w\|_{L^{\infty}}^2\\
\lesssim&\f{r_0^3}{|\beta_k|\delta}\Big(\|F-c_2|\beta_k|^{\f13}w\|_{H^{-1}}^2+c_2|\beta_k|^{\f13}\|w\|_{L^2}^2+\delta^{-1}\|F-c_2|\beta_k|^{\f13}w\|_{H^{-1}}\|w\|_{L^2}\\
&\delta^{-1}(\|F-c_2|\beta_k|^{\f13}w\|_{H^{-1}}+\sqrt{c_2}|\beta_k|^{\f16}\|w\|_{L^2})\|w\|_{L^2}\Big)+\delta\|w'\|_{L^2}\|w\|_{L^2}\\
\leq&C\Big[\f{r_0^3}{|\beta_k|\delta}\Big(\|F-c_2|\beta_k|^{\f13}w\|_{H^{-1}}^2+c_2|\beta_k|^{\f13}\|w\|_{L^2}^2+\delta^{-1}\|F-c_2|\beta_k|^{\f13}w\|_{H^{-1}}\|w\|_{L^2}\\
&\delta^{-1}(\|F-c_2|\beta_k|^{\f13}w\|_{H^{-1}}+\sqrt{c_2}|\beta_k|^{\f16}\|w\|_{L^2})\|w\|_{L^2}\Big)\\
&+\delta(\|F-c_2|\beta_k|^{\f13}w\|_{H^{-1}}+\sqrt{c_2}|\beta_k|^{\f16}\|w\|_{L^2})\|w\|_{L^2}\Big].
\end{align*}
Take $\delta=(\f{r_0^3}{|\beta_k|})^{\f13}$, then the condition $|\beta_k|\geq \max\{r_0^6,1\}$ yields
\begin{align*}
\delta=(\f{r_0^3}{|\beta_k|})^{\f13}\leq  r_0.
\end{align*}
Notice that $|\beta_k|\geq r_0^6$, we deduce
\begin{align*}
&C(c_2\f{r_0^3}{|\beta_k|^{\f23}\delta}+\sqrt{c_2}\f{r_0^3}{|\beta_k|^{\f56}\delta^2}+\sqrt{c_2}\delta|\beta_k|^{\f16})=C(c_2\f{r_0^2}{|\beta_k|^{\f13}}+2\sqrt{c_2}\f{r_0}{|\beta_k|^{\f16}})\leq C(c_2+2\sqrt{c_2}),
\end{align*}
Choose $c_2>0$ such that
\begin{align*}
C(c_2+2\sqrt{c_2})\leq\f12,
\end{align*}
we have
\begin{align*}
\|F-c_2|\beta_k|^{\f13}w\|_{H^{-1}}^2\gtrsim\min\{\f{|\beta_k|\delta}{r_0^3},(\f{|\beta_k|\delta^2}{r_0^3})^2,\delta^{-2}\}\|w\|_{L^2}^2.
\end{align*}
With the basic equality $\f{|\beta_k|\delta}{r_0^3}=(\f{|\beta_k|\delta^{\f32}}{r_0^3})^{\f43}=\delta^{-2}$, we conclude
\begin{align}
\label{kkk.5}\|F-c_2|\beta_k|^{\f13}w\|_{H^{-1}}^2\gtrsim (\f{|\beta_k|}{r_0^3})^{\f23}\|w\|_{L^2}^2=\f{|\beta_k|^{\f23}}{r_0^2}\|w\|_{L^2}^2\geq\f{|\beta_k|^{\f23}}{|\beta_k|^{\f13}}\|w\|_{L^2}^2=|\beta_k|^{\f13}\|w\|_{L^2}^2.
\end{align}

\end{enumerate}

Combining \eqref{kkk.1}, \eqref{kkk.2}, \eqref{kkk.3}, \eqref{kkk.4} and \eqref{kkk.5}, we therefore establish the following resolvent estimate from $L^2$ to $H^{-1}$
\begin{align*}
|\beta_k|^{\f16}\|w\|_{L^2}\lesssim\|F-c_2|\beta_k|^{\f13}w\|_{H^{-1}}.
\end{align*}
Applying \eqref{kkk.0}, we have
  \begin{align*}
\|w\|_{H^1}\lesssim\|F-c_2|\beta_k|^{\f13}w\|_{H^{-1}}+\sqrt{c_2}|\beta_k|^{\f16}\|w\|_{L^2}\lesssim\|F-c_2|\beta_k|^{\f13}w\|_{H^{-1}}.
\end{align*}
This completes the proof of Proposition \ref{resolvent estimate-L2}.
\end{proof}

The following proposition can be derived directly from Proposition \ref{resolvent estimate} and Proposition \ref{resolvent estimate-L2} if we choose $c_2>0$ being small enough.

\begin{proposition}\label{resolvent estimate-1}For any $|k|\geq1$, $\lambda\in\mathbb{R}$ and $w\in D_k$, there exist constants $C,c_2>0$ independent of $k,\beta_k,\lambda$, such that the following estimates hold
\begin{align*}
|\beta_k|^{\f16}\|w'\|_{L^2}+|\beta_k|^{\f13}\|w\|_{L^2}\leq C \|F-c_2|\beta_k|^{\f13}w\|_{L^2},
\end{align*}
and
\begin{align*}
\|w\|_{H^1}+|\beta_k|^{\f16}\|w\|_{L^2}\leq C \|F-c_2|\beta_k|^{\f13}w\|_{H^{-1}}.
\end{align*}
\end{proposition}

\subsection{Resolvent estimates in the weighted $L^2$ space $X$} To prove sharp decaying estimates for the nonlinear problem, we design and employ a new energy space.  Here we introduce a weighted space $L^2$ space as follows
\begin{align*}
\|w\|_{X}=(\int_0^{\infty}\f{|w|^2}{r^2}dr)^{\f12}.
\end{align*}
As we proceed in the last subsection, we start with deriving the resolvent estimates from $L^2(\f{1}{r^2})$ to $L^2(\f{1}{r^2})$.

\begin{proposition}\label{resolvent estimate-r1}For any $|k|\geq1$, $\lambda\in\mathbb{R}$ and $w\in D_k$, there exists a constant $C>0$ independent of $k,\beta_k,\lambda$, such that the following inequality holds
\begin{align*}
|\beta_k|^{\f16}\|\f{w'}{r}\|_{L^2}+|\beta_k|^{\f13}\|\f{w}{r}\|_{L^2}\leq C \|\f{F}{r}\|_{L^2}.
\end{align*}
\end{proposition}
\begin{proof}
We first prove $|\beta_k|^{\f13}\|\f{w}{r}\|_{L^2}\leq C \|\f{F}{r}\|_{L^2}$. The discussion is analogous to proofs of Proposition \ref{resolvent estimate} and Proposition \ref{resolvent estimate-L2}. If $\lambda\leq 0$, using Lemma \ref{trivial w'/r}, we have
  \begin{align*}
&\Re\langle F,\f{w}{r^2}\rangle\gtrsim\langle(\f{k^2}{r^2}+r^2)w,\f{w}{r^2}\rangle.
\end{align*}
This together with
\begin{align*}
&|\Im\langle F,\f{w}{r^2}\rangle|\geq|\beta_k|\langle\f{1}{r^2}w,\f{w}{r^2}\rangle,
\end{align*}
yields
  \begin{align}
&\label{r0.1}\|\f{F}{r}\|_{L^2}\|\f{w}{r}\|_{L^2}\gtrsim\langle(r^2+\f{|\beta_k|}{r^2})w,\f{w}{r^2}\rangle\geq|\beta_k|^{\f12}\|\f{w}{r}\|_{L^2}^2.
\end{align}

If $\lambda>0$ and $|\beta_k|\leq k^2$. With Lemma \ref{trivial w'/r}, we have
\begin{align}
\label{r1.1}&\|\f{F}{r}\|_{L^2}\gtrsim|k|\|\f{w}{r}\|_{L^2}\geq|\beta_k|^{\f12}\|\f{w}{r}\|_{L^2}.
\end{align}

If $\lambda>0$ and $|\beta_k|\geq k^2$, we consider the following three cases. Recall $r_0=\f{1}{\sqrt{\lambda}}$.
\begin{enumerate}
  \item \textbf{Case of $|\beta_k|\leq r_0^4$.} Via integration by parts we get
\begin{align*}
&\Im\langle F,\f{w}{r^2}\rangle=-2\Im\langle \f{w'}{r},\f{w}{r^2}\rangle+\beta_k\langle(\f{1}{r^2}-\f{1}{r_0^2})w,\f{w}{r^2}\rangle.
\end{align*}
Together with Lemma \ref{trivial w'/r} this gives
\begin{align*}
&|\beta_k|\langle(\f{1}{r^2}-\f{1}{r_0^2})w,\f{w}{r^2}\rangle\leq|\Im\langle F,\f{w}{r^2}\rangle|+2\|\f{w'}{r}\|_{L^2}\|\f{w}{r^2}\|_{L^2}
\lesssim\|\f{F}{r}\|_{L^2}\|\f{w}{r}\|_{L^2}.
\end{align*}
Employing Lemma \ref{trivial w'/r} again, we have
  \begin{align*}
&\|\f{F}{r}\|_{L^2}\|\f{w}{r}\|_{L^2}\gtrsim\langle[r^2+|\beta_k|(\f{1}{r^2}-\f{1}{r_0^2})]w,\f{w}{r^2}\rangle.
\end{align*}
Since
\begin{align*}
r^2+|\beta_k|(\f{1}{r^2}-\f{1}{r_0^2})\geq 2|\beta_k|^{\f12}-\f{|\beta_k|}{r_0^2}\geq|\beta_k|^{\f12},
\end{align*}
we conclude
  \begin{align}
\label{r1.2}&\|\f{F}{r}\|_{L^2}\|\f{w}{r}\|_{L^2}\gtrsim |\beta_k|^{\f12}\|\f{w}{r}\|_{L^2}^2\geq  |\beta_k|^{\f13}\|\f{w}{r}\|_{L^2}^2.
\end{align}
     \item \textbf{Case of $r_0^4\leq |\beta_k|\leq r_0^6$.} Since $|\beta_k|\geq 1$, it renders $r_0\geq 1$. Utilizing Lemma \ref{trivial w'/r}, we obtain the estimate of $\|\f{w}{r}\|_{L^2(\f{r_0}{2},\infty)}^2$ as follows
  \begin{align*}
&\|\f{F}{r}\|_{L^2}\|\f{w}{r}\|_{L^2}\gtrsim\|r\f{w}{r}\|_{L^2}^2\geq \f{r_0^2}{4}\|\f{w}{r}\|_{L^2(\f{r_0}{2},\infty)}^2.
\end{align*}
We continue to estimate $\|\f{w}{r}\|_{L^2(0,\f{r_0}{2})}^2$. Choose $r_{-}\in(\f{r_0}{2}-\f{1}{2r_0},\f{r_0}{2})$ such that the following inequality holds
\begin{align}
\label{r1-r-{-}}\f{|w'(r_{-})|^2}{r_{-}^2}\leq 2r_0\|\f{w'}{r}\|_{L^2}^2.
\end{align}
Observing the basic equality
 \begin{align*}
\Im\langle F,\f{w\chi_{(0,r_{-})}}{r^2}\rangle=&\Im\langle- [\partial_r^2w-(\f{k^2-\f14}{r^2}+\f{r^2}{16}-\f12)w]+i\beta_k(\f{1}{r^2}-\lambda)w,\f{w\chi_{(0,r_{-})}}{r^2}\rangle\\
=&\Im\langle-\partial_r^2w+i\beta_k(\f{1}{r^2}-\lambda)w,\f{w\chi_{(0,r_{-})}}{r^2}\rangle,
\end{align*}
we deduce
\begin{align*}
|\beta_k|\langle(\f{1}{r^2}-\lambda)w,\f{w\chi_{(0,r_{-})}}{r^2}\rangle\leq\|\f{F}{r}\|_{L^2}\|\f{w}{r}\|_{L^2}+\f{|w'(r_{-})w(r_{-})|}{|r_{-}|^2}.
\end{align*}
It follows from the fact that $\lambda=\f{1}{r_0^2}$ and $r_{-}\in(\f{r_0}{2}-\f{1}{r_0},\f{2r_0}{2})$
\begin{align*}
\|\f{F}{r}\|_{L^2}\|\f{w}{r}\|_{L^2}+\f{|w'(r_{-})w(r_{-})|}{|r_{-}|^2}\geq|\beta_k|\langle(\f{4}{r_0^2}-\f{1}{r_0^2})w,\f{w\chi_{(0,r_{-})}}{r^2}\rangle\geq\f{|\beta_k|}{r_0^2}\|\f{w}{r}\|_{L^2(0,r_{-})}^2.
\end{align*}
Therefore, $\|\f{w}{r}\|_{L^2}$ can be bounded by
\begin{align*}
\|\f{w}{r}\|_{L^2}^2=&\|\f{w}{r}\|_{L^2(0,r_{-})}^2+\|\f{w}{r}\|_{L^2(r_{-},\f{r_0}{2})}^2+\|\f{w}{r}\|_{L^2(\f{r_0}{2},\infty)}^2\\
\lesssim&\f{r_0^2}{|\beta_k|}(\|\f{F}{r}\|_{L^2}\|\f{w}{r}\|_{L^2}+\f{|w'(r_{-})w(r_{-})|}{|r_{-}|^2})+\f{1}{r_0}\|\f{w}{r}\|_{L^{\infty}}^2+\f{1}{r_0^2}\|\f{F}{r}\|_{L^2}\|\f{w}{r}\|_{L^2}.
\end{align*}
Together with \eqref{r1-r-{-}}, Lemma \ref{Appendix A1-1}, this yields
\begin{align*}
\|\f{w}{r}\|_{L^2}^2\lesssim&\f{r_0^2}{|\beta_k|}(\|\f{F}{r}\|_{L^2}\|\f{w}{r}\|_{L^2}+r_0^{\f12}\|\f{w'}{r}\|_{L^2}\|\f{w}{r}\|_{L^{\infty}})+\f{1}{r_0}\|\f{w}{r}\|_{L^{\infty}}^2+\f{1}{r_0^2}\|\f{F}{r}\|_{L^2}\|\f{w}{r}\|_{L^2}\\
\lesssim&\f{r_0^2}{|\beta_k|}(\|\f{F}{r}\|_{L^2}\|\f{w}{r}\|_{L^2}+r_0^{\f12}\|\f{w'}{r}\|_{L^2}\|(\f{w}{r})'\|_{L^2}^{\f12}\|\f{w}{r}\|_{L^2}^{\f12})+\f{1}{r_0}\|(\f{w}{r})'\|_{L^2}\|\f{w}{r}\|_{L^2}\\
&+\f{1}{r_0^2}\|\f{F}{r}\|_{L^2}\|\f{w}{r}\|_{L^2}.
\end{align*}
By Lemma \ref{trivial w'/r} we obtain
\begin{align*}
\|\f{w'}{r}\|_{L^2}^2+\|(\f{w}{r})'\|_{L^2}^2\lesssim\|\f{F}{r}\|_{L^2}\|\f{w}{r}\|_{L^2}, \end{align*}
which implies
\begin{align*}
\|\f{w}{r}\|_{L^2}^2\lesssim&\f{r_0^2}{|\beta_k|}(\|\f{F}{r}\|_{L^2}\|\f{w}{r}\|_{L^2}+r_0^{\f12}\|\f{w'}{r}\|_{L^2}\|(\f{w}{r})'\|_{L^2}^{\f12}\|\f{w}{r}\|_{L^2}^{\f12})+\f{1}{r_0}\|(\f{w}{r})'\|_{L^2}\|\f{w}{r}\|_{L^2}\\
&+\f{1}{r_0^2}\|\f{F}{r}\|_{L^2}\|\f{w}{r}\|_{L^2}\\
\lesssim&\f{r_0^2}{|\beta_k|}(\|\f{F}{r}\|_{L^2}\|\f{w}{r}\|_{L^2}+r_0^{\f12}\|\f{F}{r}\|_{L^2}^{\f34}\|\f{w}{r}\|_{L^2}^{\f54})+\f{1}{r_0}\|\f{F}{r}\|_{L^2}^{\f12}\|\f{w}{r}\|_{L^2}^{\f32}\\
&+\f{1}{r_0^2}\|\f{F}{r}\|_{L^2}\|\f{w}{r}\|_{L^2}.
\end{align*}
Then we conclude
\begin{align*}
\|\f{F}{r}\|_{L^2}\|\f{w}{r}\|_{L^2}\gtrsim\min\{\f{|\beta_k|}{r_0^2},(\f{|\beta_k|}{r_0^{\f52}})^{\f43},r_0^2\}\|\f{w}{r}\|_{L^2}^2.
\end{align*}
With $r_0^4\leq |\beta_k|\leq r_0^6(\Rightarrow\f{|\beta_k|^{\f43}}{r_0^{\f{10}{3}}}\geq\f{|\beta_k|}{r_0^2}\geq r_0^2)$, we deduce
\begin{align}
\label{r1.3}\|\f{F}{r}\|_{L^2}\|\f{w}{r}\|_{L^2}\gtrsim r_0^2\|\f{w}{r}\|_{L^2}^2\geq|\beta_k|^{\f13}\|\f{w}{r}\|_{L^2}^2.
\end{align}

     \item \textbf{Case of $|\beta_k|\geq r_0^6$.} Choose $r_{-}\in(r_0-\delta,r_0-\f{\delta}{2})$ and $r_{+}\in(r_0+\f{\delta}{2},r_0+\delta)$ such that it holds
\begin{align}
\label{r-1-r-{-},r-{+}}\f{|w'(r_{-})|^2}{r_{-}^2}+\f{|w'(r_{+})|^2}{r_{+}^2}\leq \f{4}{\delta}\|\f{w'}{r}\|_{L^2}^2.
\end{align}
Here $\delta\in(0, r_0)$ is a constant to be determined later. Noting
 \begin{align*}
&\Im\langle F,\f{w(\chi_{(0,r_{-})}-\chi_{(r_{+},\infty)})}{r^2}\rangle\\
=&\Im\langle- [\partial_r^2w-(\f{k^2-\f14}{r^2}+\f{r^2}{16}-\f12)w]+i\beta_k(\f{1}{r^2}-\lambda)w,\f{w(\chi_{(0,r_{-})}-\chi_{(r_{+},\infty)})}{r^2}\rangle\\
=&\Im\langle-\partial_r^2w+i\beta_k(\f{1}{r^2}-\lambda)w,\f{w(\chi_{(0,r_{-})}-\chi_{(r_{+},\infty)})}{r^2}\rangle,
\end{align*}
we obtain
\begin{align*}
&|\beta_k|(\int_0^{r_{-}}(\f{1}{r^2}-\f{1}{r_0^2})|\f{w(r)}{r}|^2dr+\int_{r_{+}}^{\infty}(\f{1}{r_0^2}-\f{1}{r^2})|\f{w(r)}{r}|^2dr)\\
\leq&\|\f{F}{r}\|_{L^2}\|\f{w}{r}\|_{L^2}+\f{|w'(r_{-})w(r_{-})|}{r_{-}^2}+\f{|w'(r_{+})w(r_{+})|}{r_{+}^2}.
\end{align*}
With $r_{-}\in(r_0-\delta,r_0-\f{\delta}{2})$ and $r_{+}\in(r_0+\f{\delta}{2},r_0+\delta)$, it holds
\begin{align*}
&\f{1}{r^2}-\f{1}{r_0^2}=\f{(r_0-r)(r_0+r)}{r^2r_0^2}\geq\f{r_0-r}{r^2r_0}\gtrsim\f{\delta}{r^2r_0}\gtrsim\f{\delta}{r_0^3},\quad \textrm{for any} \ r\in(0,r_{-}).
\end{align*}
Similarly, it can be verified that
\begin{align*}
&\f{1}{r_0^2}-\f{1}{r^2}=\f{(r-r_0)(r+r_0)}{r^2r_0^2}\geq\f{r-r_0}{r^2r_0}\gtrsim\f{\delta}{r^2r_0}\gtrsim\f{\delta}{r_0^3},\quad \textrm{for any} \ r\in(r_{+},2r_0),\\
&\f{1}{r_0^2}-\f{1}{r^2}=\f{(r-r_0)(r+r_0)}{r^2r_0^2}\approx\f{1}{r_0^2}\gtrsim\f{\delta}{r_0^3},\quad \textrm{for any} \ r\geq2r_0.
\end{align*}
Combining all inequalities above, we deduce
\begin{align*}
&\|\f{F}{r}\|_{L^2}\|\f{w}{r}\|_{L^2}+\f{|w'(r_{-})w(r_{-})|}{r_{-}^2}+\f{|w'(r_{+})w(r_{+})|}{r_{+}^2}\\
\gtrsim&\f{|\beta_k|\delta}{r_0^3}\|\f{w}{r}\|_{L^2((0,r_{-})\cup(r_{+},\infty))}^2.
\end{align*}
Hence $\|\f{w}{r}\|_{L^2}$ can be bounded by
\begin{align*}
\|\f{w}{r}\|_{L^2}^2=&\|\f{w}{r}\|_{L^2((0,r_{-})\cup(r_{+},\infty))}^2+\|\f{w}{r}\|_{L^2(r_{-},r_{+})}^2\\
\lesssim&\f{r_0^3}{|\beta_k|\delta}(\|\f{F}{r}\|_{L^2}\|\f{w}{r}\|_{L^2}+\f{|w'(r_{-})w(r_{-})|}{r_{-}^2}+\f{|w'(r_{+})w(r_{+})|}{r_{+}^2})+\delta\|\f{w}{r}\|_{L^{\infty}}^2.
\end{align*}
Together with \eqref{r-1-r-{-},r-{+}}, Lemma \ref{Appendix A1-1} and Lemma \ref{trivial w'/r}, this gives
\begin{align*}
\|\f{w}{r}\|_{L^2}^2\lesssim&\f{r_0^3}{|\beta_k|\delta}(\|\f{F}{r}\|_{L^2}\|\f{w}{r}\|_{L^2}+\delta^{-\f12}\|\f{w'}{r}\|_{L^2}\|\f{w}{r}\|_{L^{\infty}})+\delta\|\f{w}{r}\|_{L^{\infty}}^2\\
\lesssim&\f{r_0^3}{|\beta_k|\delta}(\|\f{F}{r}\|_{L^2}\|\f{w}{r}\|_{L^2}+\f{\|\f{F}{r}\|_{L^2}^{\f34}\|\f{w}{r}\|_{L^2}^{\f54}}{\delta^{\f12}})+\delta\|\f{F}{r}\|_{L^2}^{\f12}\|\f{w}{r}\|_{L^2}^{\f32}.
\end{align*}
Therefore, we obtain
\begin{align*}
\|\f{F}{r}\|_{L^2}\|\f{w}{r}\|_{L^2}\gtrsim\min\{\f{|\beta_k|\delta}{r_0^3},(\f{|\beta_k|\delta^{\f32}}{r_0^3})^{\f43},\delta^{-2}\}\|\f{w}{r}\|_{L^2}^2.
\end{align*}
Taking $\delta=(\f{r_0^3}{|\beta_k|})^{\f13}$, and noting $|\beta_k|\geq 1$, we have
\begin{align*}
\delta=(\f{r_0^3}{|\beta_k|})^{\f13}\leq r_0.
\end{align*}
With the basic equality $\f{|\beta_k|\delta}{r_0^3}=(\f{|\beta_k|\delta^{\f32}}{r_0^3})^{\f43}=\delta^{-2}$, we arrive at
\begin{align}
\label{r1.4}\|\f{F}{r}\|_{L^2}\|\f{w}{r}\|_{L^2}\gtrsim |\beta_k|^{\f13}\|\f{w}{r}\|_{L^2}^2.
\end{align}

\end{enumerate}
Plugging in \eqref{r0.1}, \eqref{r1.1}, \eqref{r1.2}, \eqref{r1.3} and \eqref{r1.4}, we therefore establish the following resolvent estimate from $L^2(\f{1}{r^2})$ to $L^2(\f{1}{r^2})$
\begin{align}
\label{r1.5-}|\beta_k|^{\f13}\|\f{w}{r}\|_{L^2}\lesssim\|\f{F}{r}\|_{L^2}.
\end{align}

  In addition, by Lemma \ref{trivial w'/r} and \eqref{r1.5-} we deduce
  \begin{align*}
\|\f{w'}{r}\|_{L^2}^2\lesssim\|\f{F}{r}\|_{L^2}\|\f{w}{r}\|_{L^2}\lesssim|\beta_k|^{-\f13}\|\f{F}{r}\|_{L^2}^2.
\end{align*}
This completes the proof of Proposition \ref{resolvent estimate-r1}.
\end{proof}

Finally, we are ready to derive the resolvent estimates from $L^2(\f{1}{r^2})$ to $H^{-1}(\f{1}{r^2})$.
\begin{proposition}\label{resolvent estimate-L2-r}For any $|k|\geq1$, $\lambda\in\mathbb{R}$ and $w\in D_k$, there exist  constants $C,c_2>0$ independent with $k,\beta_k,\lambda$, such that the following inequality holds
\begin{align*}
\|\f{w}{r}\|_{H^1}+|\beta_k|^{\f16}\|\f{w}{r}\|_{L^2}\leq C \|\f{F}{r}-c_2|\beta_k|^{\f13}\f{w}{r}\|_{H^{-1}}.
\end{align*}
\end{proposition}
\begin{proof}By Lemma \ref{trivial w'/r} we obtain
  \begin{align*}
&\Re\langle F,\f{w}{r^2}\rangle=\Re\langle F-c_2|\beta_k|^{\f13}w,\f{w}{r^2}\rangle+c_2|\beta_k|^{\f13}\|\f{w}{r}\|_{L^2}^2\gtrsim\|\f{w'}{r}\|_{L^2}^2+\langle(\f{k^2}{r^2}+r^2)w,\f{w}{r^2}\rangle\gtrsim\|\f{w}{r}\|_{H^1}^2,
\end{align*}
with $c_2>0$ to be determined later. It follows directly that
  \begin{align*}
\|\f{w}{r}\|_{H^1}^2\lesssim\|\f{F}{r}-c_2|\beta_k|^{\f13}\f{w}{r}\|_{H^{-1}}\|\f{w}{r}\|_{H^1}+c_2|\beta_k|^{\f13}\|\f{w}{r}\|_{L^2}^2.
\end{align*}
Thus we have
  \begin{align}
\label{kkk.0-r}\|\f{w}{r}\|_{H^1}\lesssim\|\f{F}{r}-c_2|\beta_k|^{\f13}\f{w}{r}\|_{H^{-1}}+\sqrt{c_2}|\beta_k|^{\f16}\|\f{w}{r}\|_{L^2}.
\end{align}

We proceed in the similar fashion as in Proposition \ref{resolvent estimate-L2} to show $|\beta_k|^{\f16}\|\f{w}{r}\|_{L^2}\leq C \|\f{F}{r}-c_2|\beta_k|^{\f13}\f{w}{r}\|_{H^{-1}}$. Likewise, we divide the proof into different cases. If $\lambda\leq0$, using Lemma \ref{trivial w'/r} we obtain
 \begin{align*}
&\Re\langle F-c_2|\beta_k|^{\f13}w,\f{w}{r^2}\rangle+c_2|\beta_k|^{\f13}\|\f{w}{r}\|_{L^2}^2\gtrsim\langle(\f{k^2}{r^2}+r^2)w,\f{w}{r^2}\rangle.
\end{align*}
This together with
\begin{align*}
&|\Im\langle F-c_2|\beta_k|^{\f13}w,\f{w}{r^2}\rangle|\geq|\beta_k|\langle\f{1}{r^2}w,\f{w}{r^2}\rangle,
\end{align*}
and \eqref{kkk.0-r} gives
  \begin{align*}
&\|\f{F-c_2|\beta_k|^{\f13}w}{r}\|_{H^{-1}}^2+c_2|\beta_k|^{\f13}\|\f{w}{r}\|_{L^2}^2\gtrsim\langle(r^2+\f{|\beta_k|}{r^2})w,\f{w}{r^2}\rangle\geq|\beta_k|^{\f12}\|\f{w}{r}\|_{L^2}^2.
\end{align*}
Then we can deduce
  \begin{align*}
&C(\|\f{F-c_2|\beta_k|^{\f13}w}{r}\|_{H^{-1}}^2+c_2|\beta_k|^{\f13}\|\f{w}{r}\|_{L^2}^2)\geq|\beta_k|^{\f12}\|\f{w}{r}\|_{L^2}^2.
\end{align*}
Notice $|\beta_k|\geq1$. Choose $c_2>0$ sufficiently small satisfying $Cc_2\leq\f12$, we conclude
  \begin{align}
\label{kkk.01-r}&\|\f{F-c_2|\beta_k|^{\f13}w}{r}\|_{H^{-1}}^2\gtrsim|\beta_k|^{\f12}\|\f{w}{r}\|_{L^2}^2.
\end{align}

If $\lambda>0$ and  $|\beta_k|\leq k^2$, it follows from Lemma \ref{trivial w'/r} and \eqref{kkk.0-r}
   \begin{align*}
\|\f{F-c_2|\beta_k|^{\f13}w}{r}\|_{H^{-1}}^2+c_2|\beta_k|^{\f13}\|\f{w}{r}\|_{L^2}^2\geq&\Re\langle F-c_2|\beta_k|^{\f13}w,\f{w}{r^2}\rangle+c_2|\beta_k|^{\f13}\|\f{w}{r}\|_{L^2}^2\\\gtrsim&\langle(\f{k^2}{r^2}+r^2)w,\f{w}{r^2}\rangle\geq|k|\|\f{w}{r}\|_{L^2}^2\geq|\beta_k|^{\f12}\|\f{w}{r}\|_{L^2}^2.
\end{align*}
Note $|\beta_k|\geq1$. Choose $c_2>0$ sufficiently small satisfying $Cc_2\leq\f12$, we conclude
  \begin{align}
\label{kkk.1-r}\|\f{F-c_2|\beta_k|^{\f13}w}{r}\|_{H^{-1}}^2\gtrsim|\beta_k|^{\f12}\|\f{w}{r}\|_{L^2}^2.
\end{align}

If $\lambda> 0$ and $|\beta_k|\geq k^2$, the discussion can be separated into three cases. Recall $r_0=\f{1}{\sqrt{\lambda}}$.
\begin{enumerate}
  \item \textbf{Case of $|\beta_k|\leq r_0^4$.} Recall that Lemma \ref{trivial w'/r} yields
  \begin{align}
\label{kkk.2-r-1}&\Re\langle F-c_2|\beta_k|^{\f13}w,\f{w}{r^2}\rangle+c_2|\beta_k|^{\f13}\|\f{w}{r}\|_{L^2}^2\gtrsim\langle(\f{k^2}{r^2}+r^2)w,\f{w}{r^2}\rangle.
\end{align}
Since
\begin{align*}
&\Im\langle F-c_2|\beta_k|^{\f13}w,\f{w}{r^2}\rangle=\Im\langle -w'',\f{w}{r^2}\rangle+\beta_k\langle(\f{1}{r^2}-\f{1}{r_0^2})w,\f{w}{r^2}\rangle\\
=&\Im\langle 2w',\f{w}{r^3}\rangle+\beta_k\langle(\f{1}{r^2}-\f{1}{r_0^2})w,\f{w}{r^2}\rangle,
\end{align*}
and note \eqref{kkk.0-r}, it holds
\begin{align*}
|\beta_k|\langle(\f{1}{r^2}-\f{1}{r_0^2})w,\f{w}{r^2}\rangle\leq& \|\f{F-c_2|\beta_k|^{\f13}w}{r}\|_{H^{-1}}\|\f{w}{r}\|_{H^1}+2\|\f{w'}{r}\|_{L^2}\|\f{w}{r^2}\|_{L^2}\\
\leq& \|\f{F-c_2|\beta_k|^{\f13}w}{r}\|_{H^{-1}}\|\f{w}{r}\|_{H^1}+2\|\f{w}{r}\|_{H^1}\|\f{w}{r}\|_{H^1}\\
\lesssim&\|\f{F-c_2|\beta_k|^{\f13}w}{r}\|_{H^{-1}}^2+c_2|\beta_k|^{\f13}\|\f{w}{r}\|_{L^2}^2.
\end{align*}
Combining with \eqref{kkk.2-r-1}, we obtain
  \begin{align*}
&\|\f{F-c_2|\beta_k|^{\f13}w}{r}\|_{H^{-1}}^2+c_2|\beta_k|^{\f13}\|\f{w}{r}\|_{L^2}^2\gtrsim\langle[r^2+|\beta_k|(\f{1}{r^2}-\f{1}{r_0^2})]w,\f{w}{r^2}\rangle.
\end{align*}
Observing
\begin{align*}
r^2+|\beta_k|(\f{1}{r^2}-\f{1}{r_0^2})\geq 2|\beta_k|^{\f12}-\f{|\beta_k|}{r_0^2}\geq|\beta_k|^{\f12},
\end{align*}
we deduce
  \begin{align*}
\|\f{F-c_2|\beta_k|^{\f13}w}{r}\|_{H^{-1}}^2+c_2|\beta_k|^{\f13}\|\f{w}{r}\|_{L^2}^2\gtrsim|\beta_k|^{\f12}\|\f{w}{r}\|_{L^2}^2.
\end{align*}
 Recall $|\beta_k|\geq1$. Choose $c_2>0$ sufficiently small so that $Cc_2\leq\f12$, we conclude
  \begin{align}
\label{kkk.2-r}&\|\f{F-c_2|\beta_k|^{\f13}w}{r}\|_{H^{-1}}^2\gtrsim|\beta_k|^{\f12}\|\f{w}{r}\|_{L^2}^2.
\end{align}
     \item \textbf{Case of $r_0^4\leq |\beta_k|\leq r_0^6$.} Then $r_0\geq1$. Utilizing Lemma \ref{trivial w'/r} and \eqref{kkk.0-r}, we obtain
  \begin{align*}
\|\f{F}{r}-c_2|\beta_k|^{\f13}\f{w}{r}\|_{H^{-1}}^2+c_2|\beta_k|^{\f13}\|\f{w}{r}\|_{L^2}^2\gtrsim&\Re\langle F-c_2|\beta_k|^{\f13}w,\f{w}{r^2}\rangle+c_2|\beta_k|^{\f13}\|\f{w}{r}\|_{L^2}^2\\
\gtrsim&\langle(\f{k^2}{r^2}+r^2)w,\f{w}{r^2}\rangle\geq\|r\f{w}{r}\|_{L^2}^2\geq \f{r_0^2}{4}\|\f{w}{r}\|_{L^2(\f{r_0}{2},\infty)}^2.
\end{align*}
We move to estimate $\|\f{w}{r}\|_{L^2(0,\f{r_0}{2})}^2$. Pick $r_{-}\in(\f{r_0}{2}-\f{1}{r_0},\f{r_0}{2})$ such that the following inequality holds
\begin{align*}
\f{|w'(r_{-})|^2}{r_{-}^2}\leq r_0\|\f{w'}{r}\|_{L^2}^2.
\end{align*}
We construct a cutoff function
\begin{align*}
\rho(r)=\left\{
\begin{aligned}
&1,\quad r\in(0,\f{r_0}{2}-\f{2}{r_0}),\\
&\sin\big(\f{\pi}{2}\f{1}{r_{-}-(\f{r_0}{2}-\f{2}{r_0})}(r_{-}-r)\big),\quad r\in(\f{r_0}{2}-\f{2}{r_0},r_{-}),\\
&0,\quad r\geq r_{-}.
\end{aligned}
\right.
\end{align*}
Starting from with the following basic equality
 \begin{align*}
\Im\langle F-c_2|\beta_k|^{\f13}w,\f{w\rho(r)}{r^2}\rangle=&\Im\langle- [\partial_r^2w-(\f{k^2-\f14}{r^2}+\f{r^2}{16}-\f12)w]+i\beta_k(\f{1}{r^2}-\lambda)w,\f{w\rho(r)}{r^2}\rangle\\
=&\Im\langle-\partial_r^2w+i\beta_k(\f{1}{r^2}-\lambda)w,\f{w\rho(r)}{r^2}\rangle,
\end{align*}
we obtain
\begin{align*}
&|\beta_k|\langle(\f{1}{r^2}-\lambda)w,\f{w\rho(r)}{r^2}\rangle\\
\lesssim&\|\f{F-c_2|\beta_k|^{\f13}w}{r}\|_{H^{-1}}\|\f{w\rho(r)}{r}\|_{H^1}+r_0\|\f{w'}{r}\|_{L^2}\|\f{w}{r}\|_{L^2}+\|\f{w'}{r}\|_{L^2}\|\f{w}{r^2}\|_{L^2}.
\end{align*}
Thus with $\lambda=\f{1}{r_0^2}$ and $r_{-}\in(\f{r_0}{2}-\f{1}{2r_0},\f{r_0}{2})$ we have
\begin{align*}
&\|\f{F-c_2|\beta_k|^{\f13}w}{r}\|_{H^{-1}}\|\f{w\rho(r)}{r}\|_{H^1}+r_0\|\f{w'}{r}\|_{L^2}\|\f{w}{r}\|_{L^2}+\|\f{w'}{r}\|_{L^2}\|\f{w}{r^2}\|_{L^2}\\
\gtrsim&|\beta_k|\langle(\f{4}{r_0^2}-\f{1}{r_0^2})w,\f{w\rho(r)}{r^2}\rangle\geq\f{|\beta_k|}{r_0^2}\|\f{w}{r}\|_{L^2(0,\f{r_0}{2}-\f{2}{r_0})}^2.
\end{align*}
Note
\begin{align*}
\|\f{\rho w}{r}\|_{H^1}\lesssim\|\f{w}{r}\|_{H^1}+r_0\|\f{w}{r}\|_{L^2}.
\end{align*}
Combining with \eqref{kkk.0-r}, we deduce
\begin{align*}
&\|\f{F-c_2|\beta_k|^{\f13}w}{r}\|_{H^{-1}}\|\f{w\rho(r)}{r}\|_{H^1}+r_0\|\f{w'}{r}\|_{L^2}\|\f{w}{r}\|_{L^2}+\|\f{w'}{r}\|_{L^2}\|\f{w}{r^2}\|_{L^2}\\
\lesssim&\|\f{F-c_2|\beta_k|^{\f13}w}{r}\|_{H^{-1}}(\|\f{w}{r}\|_{H^1}+r_0\|\f{w}{r}\|_{L^2})+r_0\|\f{w}{r}\|_{H^1}\|\f{w}{r}\|_{L^2}+\|\f{w}{r}\|_{H^1}^2\\
\lesssim&\|\f{F-c_2|\beta_k|^{\f13}w}{r}\|_{H^{-1}}(\|\f{F-c_2|\beta_k|^{\f13}w}{r}\|_{H^{-1}}+\sqrt{c_2}|\beta_k|^{\f16}\|\f{w}{r}\|_{L^2}+r_0\|\f{w}{r}\|_{L^2})\\
&+r_0(\|\f{F-c_2|\beta_k|^{\f13}w}{r}\|_{H^{-1}}+\sqrt{c_2}|\beta_k|^{\f16}\|\f{w}{r}\|_{L^2})\|\f{w}{r}\|_{L^2}\\
&+\|\f{F}{r}-c_2|\beta_k|^{\f13}\f{w}{r}\|_{H^{-1}}^2+c_2|\beta_k|^{\f13}\|\f{w}{r}\|_{L^2}^2.
\end{align*}
Therefore, with Lemma \ref{Appendix A1-1}, we now can bound $\|\f{w}{r}\|_{L^2}$ by
\begin{align*}
\|\f{w}{r}\|_{L^2}^2=&\|\f{w}{r}\|_{L^2(0,\f{r_0}{2}-\f{2}{r_0})}^2+\|\f{w}{r}\|_{L^2(\f{r_0}{2}-\f{2}{r_0},\f{r_0}{2})}^2+\|\f{w}{r}\|_{L^2(\f{r_0}{2},\infty)}^2\\
\lesssim&\f{r_0^2}{|\beta_k|}\Big[\|\f{F-c_2|\beta_k|^{\f13}w}{r}\|_{H^{-1}}(\|\f{F-c_2|\beta_k|^{\f13}w}{r}\|_{H^{-1}}+r_0\|\f{w}{r}\|_{L^2})\\
&+\sqrt{c_2}r_0|\beta_k|^{\f16}\|\f{w}{r}\|_{L^2}^2+c_2|\beta_k|^{\f13}\|\f{w}{r}\|_{L^2}^2\Big]+\f{1}{r_0}\|\f{w}{r}\|_{L^{\infty}}^2\\
&+\f{1}{r_0^2}(\|\f{F}{r}-c_2|\beta_k|^{\f13}\f{w}{r}\|_{H^{-1}}^2+c_2|\beta_k|^{\f13}\|\f{w}{r}\|_{L^2}^2)\\
\lesssim&\f{r_0^2}{|\beta_k|}\Big[\|\f{F-c_2|\beta_k|^{\f13}w}{r}\|_{H^{-1}}(\|\f{F-c_2|\beta_k|^{\f13}w}{r}\|_{H^{-1}}+r_0\|\f{w}{r}\|_{L^2})\\
&+\sqrt{c_2}r_0|\beta_k|^{\f16}\|\f{w}{r}\|_{L^2}^2+c_2|\beta_k|^{\f13}\|\f{w}{r}\|_{L^2}^2\Big]+\f{1}{r_0}\|(\f{w}{r})'\|_{L^2}\|\f{w}{r}\|_{L^2}\\
&+\f{1}{r_0^2}(\|\f{F}{r}-c_2|\beta_k|^{\f13}\f{w}{r}\|_{H^{-1}}^2+c_2|\beta_k|^{\f13}\|\f{w}{r}\|_{L^2}^2)\\
\leq&C\Big\{\f{r_0^2}{|\beta_k|}\Big[\|\f{F-c_2|\beta_k|^{\f13}w}{r}\|_{H^{-1}}(\|\f{F-c_2|\beta_k|^{\f13}w}{r}\|_{H^{-1}}+r_0\|\f{w}{r}\|_{L^2})\\
&+\sqrt{c_2}r_0|\beta_k|^{\f16}\|\f{w}{r}\|_{L^2}^2+c_2|\beta_k|^{\f13}\|\f{w}{r}\|_{L^2}^2\Big]\\
&+\f{1}{r_0}(\|\f{F-c_2|\beta_k|^{\f13}w}{r}\|_{H^{-1}}+\sqrt{c_2}|\beta_k|^{\f16}\|\f{w}{r}\|_{L^2})\|\f{w}{r}\|_{L^2}\\
&+\f{1}{r_0^2}(\|\f{F}{r}-c_2|\beta_k|^{\f13}\f{w}{r}\|_{H^{-1}}^2+c_2|\beta_k|^{\f13}\|\f{w}{r}\|_{L^2}^2)\Big\}.
\end{align*}
The assumption $r_0^4\leq |\beta_k|\leq r_0^6$ implies
\begin{align*}
&C(c_2\f{r_0^2}{|\beta_k|^{\f23}}+\sqrt{c_2}\f{r_0^3}{|\beta_k|^{\f56}}+\sqrt{c_2}\f{|\beta_k|^{\f16}}{r_0}+c_2\f{|\beta_k|^{\f13}}{r_0^2})\\
\leq& C(c_2|\beta_k|^{-\f16}+\sqrt{c_2}|\beta_k|^{-\f{1}{12}}+\sqrt{c_2}+c_2).
\end{align*}
Recalling $|\beta_k|\geq1$, we can choose $c_2>0$ small enough to obtain
\begin{align*}
C(c_2|\beta_k|^{-\f16}+\sqrt{c_2}|\beta_k|^{-\f{1}{12}}+\sqrt{c_2}+c_2)\leq\f12.
\end{align*}
Therefore, we deduce
\begin{align*}
\|\f{F-c_2|\beta_k|^{\f13}w}{r}\|_{H^{-1}}^2\gtrsim\min\{\f{|\beta_k|}{r_0^2},(\f{|\beta_k|}{r_0^3})^2,r_0^2\}\|\f{w}{r}\|_{L^2}^2.
\end{align*}
By $r_0^4\leq |\beta_k|\leq r_0^6$, i.e. $\f{|\beta_k|^2}{r_0^6}\geq\f{|\beta_k|}{r_0^2}\geq r_0^2$, we arrive at
\begin{align}
\label{kkk.3-r}\|\f{F-c_2|\beta_k|^{\f13}w}{r}\|_{H^{-1}}^2\gtrsim r_0^2\|\f{w}{r}\|_{L^2}^2\geq|\beta_k|^{\f13}\|\f{w}{r}\|_{L^2}^2.
\end{align}

     \item \textbf{Case of $|\beta_k|\geq r_0^6$.} We pick up $r_{-}\in(r_0-\delta,r_0-\f{\delta}{2})$ and $r_{+}\in(r_0+\f{\delta}{2},r_0+\delta)$ such that the following inequality holds
\begin{align*}
\f{|w'(r_{-})|^2}{r_{-}^2}+\f{|w'(r_{+})|^2}{r_{+}^2}\leq \f{4}{\delta}\|\f{w'}{r}\|_{L^2}^2,
\end{align*}
Here $\delta\in(0, r_0)$ is a constant to be determined later. As in Proposition \ref{resolvent estimate-L2}, the $H^{-1}$ estimate requires the cutoff function $\rho(r)$ is of $C^1$ regularity. Let us introduce
\begin{align*}
\rho(r)=\left\{
\begin{aligned}
&1,\quad r\in(0,r_0-2\delta),\\
&\sin\big(\f{\pi}{2}\f{1}{r_{-}-(r_0-2\delta)}(r_{-}-r)\big),\quad r\in(r_0-2\delta,r_{-}),\\
&0,\quad r\in (r_{-},r_{+}),\\
&\sin\big(\f{\pi}{2}\f{1}{r_0+2\delta-r_{+}}(r_{+}-r)\big),\quad r\in(r_{+},r_0+2\delta),\\
&-1,\quad r\geq r_0+2\delta.
\end{aligned}
\right.
\end{align*}
Write
 \begin{align*}
\Im\langle F-c_2|\beta_k|^{\f13}w,\f{w\rho(r)}{r^2}\rangle=&\Im\langle- [\partial_r^2w-(\f{k^2-\f14}{r^2}+\f{r^2}{16}-\f12)w]+i\beta_k(\f{1}{r^2}-\lambda)w,\f{w\rho(r)}{r^2}\rangle\\
=&\Im\langle-\partial_r^2w+i\beta_k(\f{1}{r^2}-\lambda)w,\f{w\rho(r)}{r^2}\rangle.
\end{align*}
This gives
\begin{align*}
&|\beta_k|(\int_0^{r_0-2\delta}(\f{1}{r^2}-\f{1}{r_0^2})\f{|w(r)|^2}{r^2}dr+\int_{r_0+2\delta}^{\infty}(\f{1}{r_0^2}-\f{1}{r^2})|w(r)|^2dr)\\
\lesssim&\|\f{F-c_2|\beta_k|^{\f13}w}{r}\|_{H^{-1}}\|\f{\rho w}{r}\|_{H^1}+\delta^{-1}\|\f{w'}{r}\|_{L^2}\|\f{w}{r}\|_{L^2}+\|\f{w'}{r}\|_{L^2}\|\f{w}{r^2}\|_{L^2}.
\end{align*}
With $r_{-}\in(r_0-\delta,r_0-\f{\delta}{2})$ and $r_{+}\in(r_0+\f{\delta}{2},r_0+\delta)$, we have
\begin{align*}
&\f{1}{r^2}-\f{1}{r_0^2}=\f{(r_0-r)(r_0+r)}{r^2r_0^2}\geq\f{r_0-r}{r^2r_0}\gtrsim\f{\delta}{r^2r_0}\gtrsim\f{\delta}{r_0^3},\quad \textrm{for any} \ r\in(0,r_{-}),
\end{align*}
Analogously, we obtain
\begin{align*}
&\f{1}{r_0^2}-\f{1}{r^2}=\f{(r-r_0)(r+r_0)}{r^2r_0^2}\geq\f{r-r_0}{r^2r_0}\gtrsim\f{\delta}{r^2r_0}\gtrsim\f{\delta}{r_0^3},\quad  \textrm{for any} \ r\in(r_{+},2r_0),\\
&\f{1}{r_0^2}-\f{1}{r^2}=\f{(r-r_0)(r+r_0)}{r^2r_0^2}\approx\f{1}{r_0^2}\gtrsim\f{\delta}{r_0^3},\quad  \textrm{for any} \ r\geq2r_0.
\end{align*}
Combining inequalities above, we deduce
\begin{align*}
&\|\f{F-c_2|\beta_k|^{\f13}w}{r}\|_{H^{-1}}\| \f{\rho w}{r}\|_{H^1}+\delta^{-1}\|\f{w'}{r}\|_{L^2}\|\f{w}{r}\|_{L^2}+\|\f{w'}{r}\|_{L^2}\|\f{w}{r^2}\|_{L^2}\\
\gtrsim&|\beta_k|(\int_0^{r_0-2\delta}(\f{1}{r^2}-\f{1}{r_0^2})|\f{w}{r}|^2dr+\int_{r_0+2\delta}^{\infty}(\f{1}{r_0^2}-\f{1}{r^2})|\f{w}{r}|^2dr)\\
\gtrsim&\f{|\beta_k|\delta}{r_0^3}\|\f{w}{r}\|_{L^2\big((0,r_0-2\delta)\cup(r_0+2\delta,\infty)\big)}^2.
\end{align*}
Noting
\begin{align*}
\|\f{\rho w}{r}\|_{H^1}\lesssim\|\f{w}{r}\|_{H^1}+\delta^{-1}\|\f{w}{r}\|_{L^2},
\end{align*}
together with \eqref{kkk.0-r}, we obtain
\begin{align*}
&\|\f{F-c_2|\beta_k|^{\f13}w}{r}\|_{H^{-1}}\|\f{\rho w}{r}\|_{H^1}+\delta^{-1}\|\f{w'}{r}\|_{L^2}\|\f{w}{r}\|_{L^2}+\|\f{w'}{r}\|_{L^2}\|\f{w}{r^2}\|_{L^2}\\
\lesssim&\|\f{F-c_2|\beta_k|^{\f13}w}{r}\|_{H^{-1}}(\|\f{w}{r}\|_{H^1}+\delta^{-1}\|\f{w}{r}\|_{L^2})+\delta^{-1}\|\f{w}{r}\|_{H^1}\|\f{w}{r}\|_{L^2}+\|\f{w}{r}\|_{H^1}^2\\
\lesssim&\|\f{F-c_2|\beta_k|^{\f13}w}{r}\|_{H^{-1}}(\|\f{F-c_2|\beta_k|^{\f13}w}{r}\|_{H^{-1}}+\sqrt{c_2}|\beta_k|^{\f16}\|\f{w}{r}\|_{L^2}+\delta^{-1}\|\f{w}{r}\|_{L^2})\\
&+\delta^{-1}(\|\f{F-c_2|\beta_k|^{\f13}w}{r}\|_{H^{-1}}+\sqrt{c_2}|\beta_k|^{\f16}\|\f{w}{r}\|_{L^2})\|\f{w}{r}\|_{L^2}\\
&+\|\f{F}{r}-c_2|\beta_k|^{\f13}\f{w}{r}\|_{H^{-1}}^2+c_2|\beta_k|^{\f13}\|\f{w}{r}\|_{L^2}^2.
\end{align*}
Therefore, with Lemma \ref{Appendix A1-1}, we can estimate $\|\f{w}{r}\|_{L^2}$ as below
\begin{align*}
\|\f{w}{r}\|_{L^2}^2=&\|\f{w}{r}\|_{L^2((0,r_0-2\delta)\cup(r_0+2\delta,\infty))}^2+\|\f{w}{r}\|_{L^2(r_0-2\delta,r_0+2\delta)}^2\\
\lesssim&\f{r_0^3}{|\beta_k|\delta}\Big[\|\f{F-c_2|\beta_k|^{\f13}w}{r}\|_{H^{-1}}(\|\f{F-c_2|\beta_k|^{\f13}w}{r}\|_{H^{-1}}+\sqrt{c_2}|\beta_k|^{\f16}\|\f{w}{r}\|_{L^2}+\delta^{-1}\|\f{w}{r}\|_{L^2})\\
&+\delta^{-1}(\|\f{F-c_2|\beta_k|^{\f13}w}{r}\|_{H^{-1}}+\sqrt{c_2}|\beta_k|^{\f16}\|\f{w}{r}\|_{L^2})\|\f{w}{r}\|_{L^2}\\
&+\|\f{F}{r}-c_2|\beta_k|^{\f13}\f{w}{r}\|_{H^{-1}}^2+c_2|\beta_k|^{\f13}\|\f{w}{r}\|_{L^2}^2\Big]+\delta\|\f{w}{r}\|_{L^{\infty}}^2\\
\lesssim&\f{r_0^3}{|\beta_k|\delta}\Big[\|\f{F-c_2|\beta_k|^{\f13}w}{r}\|_{H^{-1}}(\|\f{F-c_2|\beta_k|^{\f13}w}{r}\|_{H^{-1}}+\sqrt{c_2}|\beta_k|^{\f16}\|\f{w}{r}\|_{L^2}+\delta^{-1}\|\f{w}{r}\|_{L^2})\\
&+\delta^{-1}(\|\f{F-c_2|\beta_k|^{\f13}w}{r}\|_{H^{-1}}+\sqrt{c_2}|\beta_k|^{\f16}\|\f{w}{r}\|_{L^2})\|\f{w}{r}\|_{L^2}\\
&+\|\f{F}{r}-c_2|\beta_k|^{\f13}\f{w}{r}\|_{H^{-1}}^2+c_2|\beta_k|^{\f13}\|\f{w}{r}\|_{L^2}^2\Big]+\delta\|(\f{w}{r})'\|_{L^2}\|\f{w}{r}\|_{L^2}\\
\lesssim&\f{r_0^3}{|\beta_k|\delta}\Big[\|\f{F-c_2|\beta_k|^{\f13}w}{r}\|_{H^{-1}}(\|\f{F-c_2|\beta_k|^{\f13}w}{r}\|_{H^{-1}}+\delta^{-1}\|\f{w}{r}\|_{L^2})+\sqrt{c_2}\delta^{-1}|\beta_k|^{\f16}\|\f{w}{r}\|_{L^2}^2\\
&+c_2|\beta_k|^{\f13}\|\f{w}{r}\|_{L^2}^2\Big]+\delta(\|\f{F-c_2|\beta_k|^{\f13}w}{r}\|_{H^{-1}}+\sqrt{c_2}|\beta_k|^{\f16}\|\f{w}{r}\|_{L^2})\|\f{w}{r}\|_{L^2}\\
\leq&C\Big\{\f{r_0^3}{|\beta_k|\delta}\Big[\|\f{F-c_2|\beta_k|^{\f13}w}{r}\|_{H^{-1}}(\|\f{F-c_2|\beta_k|^{\f13}w}{r}\|_{H^{-1}}+\delta^{-1}\|\f{w}{r}\|_{L^2})+\sqrt{c_2}\delta^{-1}|\beta_k|^{\f16}\|\f{w}{r}\|_{L^2}^2\\
&+c_2|\beta_k|^{\f13}\|\f{w}{r}\|_{L^2}^2\Big]+\delta(\|\f{F-c_2|\beta_k|^{\f13}w}{r}\|_{H^{-1}}+\sqrt{c_2}|\beta_k|^{\f16}\|\f{w}{r}\|_{L^2})\|\f{w}{r}\|_{L^2}\Big\}.
\end{align*}
Take $\delta=(\f{r_0^3}{|\beta_k|})^{\f13}$, then the condition $|\beta_k|\geq \max\{r_0^6,1\}$ yields
\begin{align*}
\delta=(\f{r_0^3}{|\beta_k|})^{\f13}\leq r_0.
\end{align*}
Recall $|\beta_k|\geq 1$, we have
\begin{align*}
&C(c_2\f{r_0^3}{|\beta_k|^{\f23}\delta}+\sqrt{c_2}\f{r_0^3}{|\beta_k|^{\f56}\delta^2}+\sqrt{c_2}\delta|\beta_k|^{\f16})=C(c_2\f{r_0^2}{|\beta_k|^{\f13}}+2\sqrt{c_2}\f{r_0}{|\beta_k|^{\f16}})\leq C(c_2+2\sqrt{c_2}).
\end{align*}
If $c_2>0$ satisfies
\begin{align*}
C(c_2+2\sqrt{c_2})\leq\f12,
\end{align*}
we can conclude
\begin{align*}
\|\f{F-c_2|\beta_k|^{\f13}w}{r}\|_{H^{-1}}^2\gtrsim\min\{\f{|\beta_k|\delta}{r_0^3},(\f{|\beta_k|\delta^2}{r_0^3})^2,\delta^{-2}\}\|\f{w}{r}\|_{L^2}^2.
\end{align*}
With the basic equality $\f{|\beta_k|\delta}{r_0^3}=(\f{|\beta_k|\delta^2}{r_0^3})^2=\delta^{-2}$, we arrive at
\begin{align}
\label{kkk.4-r}\|\f{F-c_2|\beta_k|^{\f13}w}{r}\|_{H^{-1}}^2\gtrsim |\beta_k|^{\f13}\|\f{w}{r}\|_{L^2}^2.
\end{align}

\end{enumerate}
Combining \eqref{kkk.01-r}, \eqref{kkk.1-r}, \eqref{kkk.2-r}, \eqref{kkk.3-r} and \eqref{kkk.4-r}, we therefore establish the following resolvent estimate from $L^2(\f{1}{r^2})$ to $H^{-1}(\f{1}{r^2})$
\begin{align*}
|\beta_k|^{\f16}\|\f{w}{r}\|_{L^2}\lesssim\|\f{F-c_2|\beta_k|^{\f13}w}{r}\|_{H^{-1}}.
\end{align*}
Applying \eqref{kkk.0-r}, we obtain
  \begin{align*}
\|\f{w}{r}\|_{H^1}\lesssim\|\f{F}{r}-c_2|\beta_k|^{\f13}\f{w}{r}\|_{H^{-1}}+\sqrt{c_2}|\beta_k|^{\f16}\|\f{w}{r}\|_{L^2}\lesssim\|\f{F-c_2|\beta_k|^{\f13}w}{r}\|_{H^{-1}}.
\end{align*}

This completes the proof of Proposition \ref{resolvent estimate-L2-r}.
\end{proof}

The following proposition follows directly from Proposition \ref{resolvent estimate-r1} and Proposition \ref{resolvent estimate-L2-r} if we choose $c_2>0$ small enough.

\begin{proposition}\label{resolvent estimate-r1-1}For any $|k|\geq1$, $\lambda\in\mathbb{R}$ and $w\in D_k$, there exist constants $C,c_2>0$ which independent with $k,\beta_k,\lambda$, such that
\begin{align*}
|\beta_k|^{\f16}\|\f{w'}{r}\|_{L^2}+|\beta_k|^{\f13}\|\f{w}{r}\|_{L^2}\leq C \|\f{F-c_2|\beta_k|^{\f13}w}{r}\|_{L^2},
\end{align*}
and
\begin{align*}
\|\f{w}{r}\|_{H^1}+|\beta_k|^{\f16}\|\f{w}{r}\|_{L^2}\leq C \|\f{F-c_2|\beta_k|^{\f13}w}{r}\|_{H^{-1}}.
\end{align*}
\end{proposition}

\subsection{Sharpness of the resolvent bound $|\beta_k|^{\f13}$}
In the last subsection, we prove that the resolvent bound $|\beta_k|^{\f13}$ in Proposition \ref{resolvent estimate} is sharp. We have 
\begin{lemma}There exist $\lambda\in\mathbb{R}$ and non-zero $w\in C_0^{\infty}(\mathbb{R}_{+},dr)$ such that it holds
\begin{align*}
\|F\|_{L^2}\leq C|\beta_1|^{\f13}\|w\|_{L^2}.
\end{align*}
\end{lemma}
\begin{proof}Choose $\lambda=\f{1}{r_0^2}$ and $\beta_1=r_0^6$ for some $r_0\geq 1$. We pick up a function $\eta(r)$ as below
\begin{align*}
\eta(r)=\left\{
\begin{aligned}
&r(1-r),\quad0\leq r\leq1,\\
&0,\quad r\geq1.
\end{aligned}
\right.
\end{align*}
We construct
\begin{align*}
w(r)=\left\{
\begin{aligned}
&(r-r_0)(r_0+\f{1}{r_0}-r),\quad r_0\leq r\leq r_0+\f{1}{r_0},\\
&0,\quad r\in(0,r_0)\cup(r_0+\f{1}{r_0},\infty).
\end{aligned}
\right.
\end{align*}
One can verify
\begin{align*}
\|w\|_{L^2}=r_0^{-\f52}\|\eta\|_{L^2},\quad w''=-2.
\end{align*}
Then it holds
\begin{align*}
\|-[\partial_r^2-(\f34\f{1}{r^2}+\f{r^2}{16}-\f12)]w\|_{L^2}\lesssim (1+r_0^2)\|w\|_{L^2}\lesssim r_0^2\|w\|_{L^2}.
\end{align*}
In addition, we have
\begin{align*}
|\beta_1|\|(\f{1}{r_0^2}-\f{1}{r^2})w\|_{L^2}=|\beta_1|\|\f{(r-r_0)(r+r_0)}{r_0^2r^2}w\|_{L^2}\leq\f{|\beta_1|}{r_0}\|\f{r+r_0}{r_0^2r^2}w\|_{L^2}\approx\f{|\beta_1|}{r_0^4}\|w\|_{L^2}.
\end{align*}
Thus we deduce
\begin{align*}
\|-[\partial_r^2-(\f34\f{1}{r^2}+\f{r^2}{16}-\f12)]w+i\beta_1(\f{1}{r^2}-\f{1}{r_0^2})\|_{L^2}\lesssim (r_0^2+\f{|\beta_1|}{r_0^4})\|w\|_{L^2}.
\end{align*}
Together with $|\beta_1|=r_0^6$  this implies
\begin{align*}
\|-[\partial_r^2-(\f34\f{1}{r^2}+\f{r^2}{16}-\f12)]w+i\beta_1(\f{1}{r^2}-\f{1}{r_0^2})w\|_{L^2}\lesssim r_0^2\|w\|_{L^2}=|\beta_1|^{\f13}\|w\|_{L^2}.
\end{align*}
\end{proof}

\maketitle

\section{\textbf{Space-time estimate for the linearized Navier-Stokes equations}}\label{3-space time}

In this section, for the linearized 2D Navier-Stokes equation in the vorticity formulation \eqref{main equation}, we establish the space-time estimate. We consider
\begin{align*}
\left\{
\begin{aligned}
&\partial_{\tau}w_k+L_kw_k+f_1-\partial_{r}f_2=0,\\
&w_k(0)=w_k|_{\tau=0},\quad w_k|_{r=0,\infty}=0.
\end{aligned}
\right.
\end{align*}
Here 
$L_k=- [\partial_r^2-(\f{k^2-\f14}{r^2}+\f{r^2}{16}-\f12)]+i\f{kB}{r^2}$ and $f_1, f_2$ are nonlinear terms in Section \ref{1.4-Derivation of equations}.

We also introduce the space-time norm
	\begin{align*}
	\|g\|_{L^pL^2}=\big\|\|g\|_{L_r^2}\big\|_{L_t^p(0,\infty)},\quad \|g\|_{L^p X}=\big\|\|g\|_{X}\big\|_{L_t^p(0,\infty)},
	\end{align*}
	where
	\begin{align*}
	\|g\|_{L_r^2}=(\int_0^{\infty}|g(r)|^2dr)^{\f12},\quad \|g\|_{X}=(\int_0^{\infty}\f{|g(r)|^2}{r^2}dr)^{\f12}.
	\end{align*}

\subsection{Pseudospectral bound and semigroup bound}

\subsubsection{Pseudospectral bound.} Recall in \cite{Pa} an operator $L$ in a Hilbert space $H$ is called accretive if
\begin{align*}
\Re\langle Lf, f\rangle\geq0,\quad \textrm{for any}\  f\in D(L).
\end{align*}
The operator $L$ is called m-accretive if in addition all $\lambda>0$ belong to the resolvent set of $L$ (see \cite{Kato} for more details).  The \textbf{pseudospectral bound} of $L$ is defined to be
\begin{align}\label{pseudospectral bound}
\Psi(L)=\inf\{\|(L-i\lambda)f\|:f\in D(L),\lambda\in\mathbb{R},\|f\|=1\}. 
\end{align}
Recall the operator $L_k$ in \eqref{main equation} is given by
\begin{align*}
L_kw=&- [\partial_r^2-(\f{k^2-\f14}{r^2}+\f{r^2}{16}-\f12)]w+i\f{kB}{r^2}w.
\end{align*}
And we will use the weighted-$L^2$ space $X$. Its norm is expressed as $\|w\|_{X}=(\int_0^{\infty}\f{|w|^2}{r^2}dr)^{\f12}.$
Note that $-\partial_r^2$ is a operator with the compact resolvent.  Since $L_k$ is a relatively compact perturbation for $-\partial_r^2$ in the domain $D_k$, it is hence clear that the operators $L_k$ has the compact resolvent and only point spectrum. 

Employing Lemma \ref{trivial w' lemma} and Lemma \ref{trivial w'/r} we obtain
\begin{align*}
\Re\langle (L_k-i\lambda)w,w\rangle_{L^2}\gtrsim|k|\|w\|_{L^2}^2,\quad\Re\langle (L_k-i\lambda)w,w\rangle_{X}\gtrsim|k|\|w\|_{X}^2, \quad \textrm{for any} \
\lambda\in\mathbb{R}.
\end{align*}
This indicates that $L_k$ is accretive, and is also m-accretive. Recall in Proposition \ref{resolvent estimate} and Proposition \ref{resolvent estimate-r1} we establish the resolvent estimates 
\begin{align*}
&\|(L_k-i\lambda)w\|_{L^2}\geq C|kB|^{\f13}\|w\|_{L^2},\quad\|(L_k-i\lambda)w\|_{X}\geq C|kB|^{\f13}\|w\|_{X}.
\end{align*}
This implies the following lemma for the pseudospectral bounds of $L_k$:
\begin{lemma}\label{peudospectral bound}Let $\Psi$ be defined as in (\ref{pseudospectral bound}), it holds $\Psi(L_k(L^2\rightarrow L^2))\geq C|kB|^{\f13}$, $\Psi(L_k(X\rightarrow X))\geq C |kB|^{\f13}$.
\end{lemma}

\subsubsection{Semigroup bounds.} To prove decay estimates, we appeal to semigroup bounds. To obtain them from pseudospectral bounds, we employ the following Gearhart-Pr$\ddot{u}$ss type lemma established by Wei in \cite{Wei}. 
\begin{lemma}[Wei \cite{Wei}]\label{GP lemma}Let $L$ be a m-accretive operator in a Hilbert space $H$. Then we have
	\begin{align*}
	\|e^{-tL}\|_H\leq e^{-t\Psi+\f{\pi}{2}}, \quad \textrm{for any} \ t\geq0.
	\end{align*}
\end{lemma}

We are now ready to study the homogeneous linear equation
\begin{align}
\label{homogeneous linear equation}
&\partial_{\tau}w_k^l+L_kw_k^l=0,\quad w_k^l(0)=w_k(0).
\end{align}
Utilizing the language of semigroup theory, we can write
\begin{align*}
w_k^l(\tau)=e^{-\tau L_k}w_k(0).
\end{align*}
And we obtain
\begin{proposition}\label{linear exp decay}Let $w_k^l$ be a solution of \eqref{homogeneous linear equation} with $w_k(0)\in L^2$ and $w_k(0)\in X$. Then for any $|k|\geq1(k\in \mathbb{Z})$, there exist constants $C,c_3>0$ being independent of $B,k$, such that the following inequalities hold
	\begin{align}
	\label{linear exp decay-L2}&\|w_k^l(\tau)\|_{L^2}\leq Ce^{-c_3|kB|^{\f13}\tau}\|w_k(0)\|_{L^2}, \quad \textrm{for any} \ \tau\geq0,\\
	\label{linear exp decay-X}&\|w_k^l(\tau)\|_{X}\leq Ce^{-c_3|kB|^{\f13}\tau}\|w_k(0)\|_{X}, \quad \textrm{for any} \ \tau\geq0.
	\end{align}
	Moreover, for any $|k|\geq1(k\in \mathbb{Z})$ and any $c'\in(0,c_3)$, it holds
	\begin{align}
	\label{linear exp decay-L2L2}&|kB|^{\f13}\|e^{-c'|kB|^{\f13}\tau}w_k^l(\tau)\|_{L^2L^2}^2\leq C\|w_k(0)\|_{L^2}^2,\\
	\label{linear exp decay-L2X}&|kB|^{\f13}\|e^{-c'|kB|^{\f13}\tau}w_k^l(\tau)\|_{L^2X}^2\leq C\|w_k(0)\|_{X}^2.
	\end{align}
\end{proposition}
\begin{proof}To derive the semigroup bounds \eqref{linear exp decay-L2} and \eqref{linear exp decay-X}, we can directly apply Lemma \ref{peudospectral bound} and Lemma \ref{GP lemma}. We proceed to use \eqref{linear exp decay-L2}. For any $c'\in(0,c_3)$ we deduce
	\begin{align*}
	2c'|kB|^{\f13}\|e^{c'|kB|^{\f13}\tau}w_k^l(\tau)\|_{L^2}^2\leq2Cc'|kB|^{\f13}e^{-2(c_3-c')|kB|^{\f13}\tau}\|w_k(0)\|_{L^2}^2.
	\end{align*}
	Integrating with respect to $\tau$ on both sides yields
	\begin{align*}
	&2c'|kB|^{\f13}\|e^{c'|kB|^{\f13}\tau}w_k^l\|_{L^2L^2}^2\leq\int_0^{\infty}2Cc'|kB|^{\f13}e^{-2(c_2-c')|kB|^{\f13}\tau}d\tau\|w_k(0)\|_{L^2}^2\lesssim\|w_k(0)\|_{L^2}^2.
	\end{align*}
	In a similar fashion, by \eqref{linear exp decay-L2}, for any $c'\in(0,c_3)$ we have
	\begin{align*}
	&2c'|kB|^{\f13}\|e^{c'|kB|^{\f13}\tau}w_k^l(\tau)\|_{X}^2\leq2Cc'|kB|^{\f13}e^{-2(c_2-c')|kB|^{\f13}\tau}\|w_k(0)\|_{X}^2,
	\end{align*}
	which implies
	\begin{align*}
	&2c'|kB|^{\f13}\|e^{c'|kB|^{\f13}\tau}w_k^l\|_{L^2X}^2\leq\int_0^{\infty}2Cc'|kB|^{\f13}e^{-2(c_2-c')|kB|^{\f13}\tau}d\tau\|w_k(0)\|_{X}^2\lesssim\|w_k(0)\|_{X}^2.
	\end{align*}
	This completes the proof of Proposition \ref{linear exp decay}.
\end{proof}

\subsection{Space-time estimates for non-zero frequency}\label{3.2-non-zero}

Recall the full nonlinear equation \eqref{main equation} as follows
\begin{align*}
\left\{
\begin{aligned}
&\partial_{\tau}w_k+L_kw_k+f_1-\partial_{r}f_2=0,\\
&w_k(0)=w_k|_{\tau=0},\quad w_k|_{r=0,\infty}=0,
\end{aligned}
\right.
\end{align*}
where $L_k=- [\partial_r^2-(\f{k^2-\f14}{r^2}+\f{r^2}{16}-\f12)]+i\f{kB}{r^2}$ and 
\begin{align*}
&f_1=ik\sum_{l\in\mathbb{Z}}w_{l}\f{\partial_r\breve{\varphi}_{k-l}}{r}+\sum_{l\in\mathbb{Z}}i(k-l)(\f14-\f{1}{2r^2})w_l\breve{\varphi}_{k-l},\quad f_2=\sum_{l\in\mathbb{Z}}i(k-l)\f{w_l\breve{\varphi}_{k-l}}{r},
\end{align*}
with $(\partial_{r}^2+\f{1}{r}\partial_{r}-\f{k^2}{r^2})\breve{\varphi}_k=fw_k$ , $f=\f{e^{-\f{r^2}{8}}}{r^{\f12}}$.

We decompose the solution $w_k$ into two parts: let $w_k=w_k^l+w_k^{n}$ so that $w_k^l$ satisfies the below homogeneous linear equation with initial data $w_k(0)$
\begin{align*}
&\partial_{\tau}w_k^l+L_kw_k^l=0,\quad w_k^l(0)=w_k(0),
\end{align*}
and $w_k^{n}$ satisfies the following inhomogeneous linear equation with zero initial data
\begin{align}
\label{nonzero mode nonlinear}
&\partial_{\tau}w_k^{n}+L_kw_k^{n}+f_1-\partial_{r}f_2=0,\quad w_k^{n}(0)=0.
\end{align}

We first establish the following space-time estimate for the linear part $w_k^l$ in $L^2$ space.

\begin{lemma}\label{nonzero frequency linear part}Let $w_k^l$ be a solution to \eqref{homogeneous linear equation} with initial data $w_k(0)\in L^2$, $c_3$ be the same as in Proposition \ref{linear exp decay} and $c'\in(0,c_3)$ . For any $|k|\geq1(k\in\mathbb{Z})$, it holds
	\begin{align}
	\label{nonzero linear part}&\|e^{c'|kB|^{\f13}\tau}w_k^l\|_{L^{\infty}L^2}^2+|kB|^{\f13}\|e^{c'|kB|^{\f13}\tau}w_k^l\|_{L^2L^2}^2\\
	\nonumber&+\|e^{c'|kB|^{\f13}\tau}\partial_{r}w_k^l\|_{L^2L^2}^2
	+\|e^{c'|kB|^{\f13}\tau}(\f{|k|}{r}+r)w_k^l\|_{L^2L^2}^2\lesssim\|w_k(0)\|_{L^2}^2.
	\end{align}
\end{lemma}
\begin{proof}
	Employing integration by parts we obtain
	\begin{align*}
	&\Re\langle\partial_{\tau}w_k^l- [\partial_{r}^2-(\f{k^2-\f14}{r^2}+\f{r^2}{16}-\f12)]w_k^l+\f{ikB}{r^2}w_k^l,w_k^l\rangle\\
	=&\f12\partial_{\tau}\|w_k^l\|_{L^2}^2+\langle- [\partial_{r}^2-(\f{k^2-\f14}{r^2}+\f{r^2}{16}-\f12)]w_k^l,w_k^l\rangle=0.
	\end{align*}
	Together with Lemma \ref{trivial w' lemma}, this gives that there exists a constant $c_0>0$ such that
	\begin{align*}
	\partial_{\tau}\|w_k^l\|_{L^2}^2+c_0(\|\partial_{r}w_k^l\|_{L^2}^2+\|(\f{|k|}{r}+r)w_k^l\|_{L^2}^2)\leq0.
	\end{align*}
	Multiplying $e^{2c'|kB|^{\f13}\tau}$ on both sides, we deduce
	\begin{align*}
	&\partial_{\tau}\|e^{c'|kB|^{\f13}\tau}w_k^l\|_{L^2}^2+c_0(\|e^{c'|kB|^{\f13}\tau}\partial_{r}w_k^l\|_{L^2}^2+\|e^{c'|kB|^{\f13}\tau}(\f{|k|}{r}+r)w_k^l\|_{L^2}^2)\\
	\nonumber\leq &2c'|kB|^{\f13}\|e^{c'|kB|^{\f13}\tau}w_k^l\|_{L^2}^2.
	\end{align*}
	Thus we obtain the space-time estimate of $w_k^l$ as follows
	\begin{align*}
	&\|e^{c'|kB|^{\f13}\tau}w_k^l\|_{L^{\infty}L^2}^2+c_0\Big(\|e^{c'|kB|^{\f13}\tau}\partial_{r}w_k^l\|_{L^2L^2}^2+\|e^{c'|kB|^{\f13}\tau}(\f{|k|}{r}+r)w_k^l\|_{L^2L^2}^2\Big)\\
	\leq & 2c'|kB|^{\f13}\|e^{c'|kB|^{\f13}\tau}w_k^l\|_{L^2L^2}^2+\|w_k(0)\|_{L^2}^2.
	\end{align*}
	Combining with Proposition \ref{linear exp decay}, we arrive at
	\begin{align*}
	&\|e^{c'|kB|^{\f13}\tau}w_k^l\|_{L^{\infty}L^2}^2+|kB|^{\f13}\|e^{c'|kB|^{\f13}\tau}w_k^l\|_{L^2L^2}^2\\
	&+\|e^{c'|kB|^{\f13}\tau}\partial_{r}w_k^l\|_{L^2L^2}^2
	+\|e^{c'|kB|^{\f13}\tau}(\f{|k|}{r}+r)w_k^l\|_{L^2L^2}^2\lesssim\|w_k(0)\|_{L^2}^2.
	\end{align*}
	This completes the proof of Lemma \ref{nonzero frequency linear part}.
\end{proof}

Analogously, we derive the space-time estimate for the linear part $w_k^l$ in weighted $L^2$ space.

\begin{lemma}\label{nonzero frequency linear part-r}Let $w_k^l$ be a solution to \eqref{homogeneous linear equation} with initial data $w_k(0)\in X$, $c_3$ be the same as in Proposition \ref{linear exp decay} and $c'\in(0,c_3)$ . For any $|k|\geq1(k\in\mathbb{Z})$, it holds
	\begin{align}
	\label{nonzero linear part-r}&\|e^{c'|kB|^{\f13}\tau}w_k^l\|_{L^{\infty}X}^2+|kB|^{\f13}\|e^{c'|kB|^{\f13}\tau}w_k^l\|_{L^2X}^2\\
	\nonumber&+\|e^{c'|kB|^{\f13}\tau}\partial_{r}w_k^l\|_{L^2X}^2+\|e^{c'|kB|^{\f13}\tau}(\f{|k|}{r}+r)w_k^l\|_{L^2X}^2\lesssim\|w_k(0)\|_{X}^2.
	\end{align}
\end{lemma}
\begin{proof}
	Via integration by parts we have
	\begin{align*}
	&\Re\langle\partial_{\tau}w_k^l- [\partial_{r}^2-(\f{k^2-\f14}{r^2}+\f{r^2}{16}-\f12)]w_k^l+\f{ikB}{r^2}w_k^l,\f{w_k^l}{r^2}\rangle\\
	=&\f12\partial_{\tau}\|\f{w_k^l}{r}\|_{L^2}^2+\langle- [\partial_{r}^2-(\f{k^2-\f14}{r^2}+\f{r^2}{16}-\f12)]w_k^l,\f{w_k^l}{r^2}\rangle=0.
	\end{align*}
Together with Lemma \ref{trivial w'/r}, this yields that there exists a constant $c_0>0$ such that
	\begin{align*}
	\partial_{\tau}\|w_k^l\|_{X}^2+c_0\Big(\|\partial_{r}w_k^l\|_{X}^2+\|(\f{|k|}{r}+r)w_k^l\|_{X}^2\Big)\leq0.
	\end{align*}
	Multiplying $e^{2c'|kB|^{\f13}\tau}$ on both sides, the inequality above becomes
	\begin{align*}
	&\partial_{\tau}\|e^{c'|kB|^{\f13}\tau}w_k^l\|_{X}^2+c_0(\|e^{c'|kB|^{\f13}\tau}\partial_{r}w_k^l\|_{X}^2+\|e^{c'|kB|^{\f13}\tau}(\f{|k|}{r}+r)w_k^l\|_{X}^2)\\
	\nonumber\leq &2c'|kB|^{\f13}\|e^{c'|kB|^{\f13}\tau}w_k^l\|_{X}^2.
	\end{align*}
	This implies the space-time estimate of $w_k^l$ as below
	\begin{align*}
	&\|e^{c'|kB|^{\f13}\tau}w_k^l\|_{L^{\infty}X}^2+c_0(\|e^{c'|kB|^{\f13}\tau}\partial_{r}w_k^l\|_{L^2X}^2+\|e^{c'|kB|^{\f13}\tau}(\f{|k|}{r}+r)w_k^l\|_{L^2X}^2)\\
	\leq & 2c'|kB|^{\f13}\|e^{c'|kB|^{\f13}\tau}w_k^l\|_{L^2X}^2+\|w_k(0)\|_{X}^2.
	\end{align*}
	Plugging in Proposition \ref{linear exp decay}, we conclude
	\begin{align*}
	&\|e^{c'|kB|^{\f13}\tau}w_k^l\|_{L^{\infty}X}^2+|kB|^{\f13}\|e^{c'|kB|^{\f13}\tau}w_k^l\|_{L^2X}^2\\
	&+\|e^{c'|kB|^{\f13}\tau}\partial_{r}w_k^l\|_{L^2X}^2
	+\|e^{c'|kB|^{\f13}\tau}(\f{|k|}{r}+r)w_k^l\|_{L^2X}^2\lesssim\|w_k(0)\|_{X}^2.
	\end{align*}
	This completes the proof of Lemma \ref{nonzero frequency linear part-r}.
\end{proof}

Next we establish the space-time estimates for $w_k^{n}$, which satisfies the following inhomogeneous equation \eqref{nonzero mode nonlinear} with zero initial data
\begin{align}\label{zero initial data}
\partial_{\tau}w_k^{n}+L_kw_k^{n}+f_1-\partial_{r}f_2=0,\quad w_k^{n}(0)=0.
\end{align}
Multiply both sides by $e^{c_2|kB|^{\f13}\tau}$, with $c_2$ being the same constant as in Proposition \ref{resolvent estimate-1} and in Proposition \ref{resolvent estimate-r1-1}. The above inhomogeneous equation then becomes
\begin{align*}
&\partial_{\tau}(e^{c_2|kB|^{\f13}\tau}w_k^{n})+\Big(L_k-c_2|kB|^{\f13}\Big)(e^{c_2|kB|^{\f13}\tau}w_k^{n})+e^{c_2|kB|^{\f13}\tau}(f_1-\partial_{r}f_2)=0,\quad w_k^{n}(0)=0.
\end{align*}
Define $\tilde{w}_k^{n}=e^{c_2|kB|^{\f13}\tau}w_k^{n}$ and $\tilde{f}_j=e^{c_2|kB|^{\f13}\tau}f_j$ for $j=1,2$. We then derive the following space-time estimate for the nonlinear part $w_k^n$ in $L^2$ space.

\begin{lemma}\label{nonzero frequency nonlinear part2}Let $w_k^n$ be a solution to \eqref{nonzero mode nonlinear} with zero initial data. For any $|k|\geq1(k\in\mathbb{Z})$, it holds
	\begin{align*}
	&\|e^{c_2|kB|^{\f13}\tau}w_k^{n}\|_{L^{\infty}L^2}^2+|kB|^{\f13}\|e^{c_2|kB|^{\f13}\tau}w_k^{n}\|_{L^2L^2}^2+\|e^{c_2|kB|^{\f13}\tau}\partial_{r}w_k^n\|_{L^2L^2}^2\\
	&+\|e^{c_2|kB|^{\f13}\tau}(\f{|k|}{r}+r)w_k^n\|_{L^2L^2}^2\lesssim|kB|^{-\f13}\|e^{c_2|kB|^{\f13}\tau}f_1\|_{L^2L^2}^2+\|e^{c_2|kB|^{\f13}\tau}f_2\|_{L^2L^2}^2.
	\end{align*}
\end{lemma}
\begin{proof}
	We first take the Fourier transform in $\tau$ and define\footnote{Note that the zero initial condition in \eqref{zero initial data} ensures that the Fourier transform is well-defined for $\lambda\in (0,\infty)$.}
	\begin{align*}
	\hat{w}_k^n(\lambda,r)=\int_0^{\infty}\tilde{w}_k^n(\tau,r)e^{-i\tau\lambda}d\tau,\quad F_j(\lambda,r,k)=\int_0^{\infty}\tilde{f}_j(\tau,r,k)e^{-i\tau\lambda}d\tau \quad \textrm{for} \ j=1,2.
	\end{align*}
The inhomogeneous equation \eqref{nonzero mode nonlinear} is transferred into the following resolvent equation
	\begin{align*}
	(i\lambda+L_k-c_2|kB|^{\f13})\hat{w}_k^n(\lambda,r)=-F_1+\partial_{r}F_2.
	\end{align*}
	We further decompose $\hat{w}_k^n$ as $\hat{w}_k^n=\hat{w}_k^{n,1}+\hat{w}_k^{n,2}$, where $\hat{w}_k^{n,1}$ and $\hat{w}_k^{n,2}$ satisfy
	\begin{align*}
	\Big(i\lambda+L_k-c_2|kB|^{\f13}\Big)\hat{w}_k^{n,1}=-F_1 \quad \textrm{and} \quad
	\Big(i\lambda+L_k-c_2|kB|^{\f13}\Big)\hat{w}_k^{n,2}=\partial_{r}F_2.
	\end{align*}
	respectively.
	
	Applying Proposition \ref{resolvent estimate-1}, we obtain
	\begin{align*}
	&|kB|^{\f16}\|\partial_{r}\hat{w}_k^{n,1}\|_{L^2}+|kB|^{\f13}\|\hat{w}_k^{n,1}\|_{L^2}\leq C \|F_1\|_{L^2},
	\end{align*}
	and
	\begin{align*}
	&\|\partial_{r}\hat{w}_k^{n,2}\|_{L^2}+|kB|^{\f16}\|\hat{w}_k^{n,2}\|_{L^2}\leq C\|\partial_{r}F_2\|_{H^{-1}}\leq C\|F_2\|_{L^2}.
	\end{align*}
	Combining the above inequalities, we deduce
	\begin{align*}
	|kB|^{-\f16}\|\partial_{r}\hat{w}_k^{n}\|_{L^2}+\|\hat{w}_k^n\|_{L^2}\lesssim |kB|^{-\f13}\|F_1\|_{L^2}+|kB|^{-\f16}\|F_2\|_{L^2}.
	\end{align*}
		Therefore, using Plancherel's theorem, it holds
	\begin{align}
	\nonumber&|kB|^{-\f16}\|\partial_{r}\tilde{w}_k^{n}\|_{L^2L^2}+\|\tilde{w}_k^n\|_{L^2L^2}\\
		\nonumber\approx&|kB|^{-\f16}\big\|\|\partial_{r}\hat{w}_k^{n}\|_{L^2}\big\|_{L^2(\mathbb{R})}+\big\|\|\hat{w}_k^n\|_{L^2}\big\|_{L^2(\mathbb{R})}\\
\nonumber\lesssim& |kB|^{-\f13}\big\|\|F_1\|_{L^2}\big\|_{L^2(\mathbb{R})}+|kB|^{-\f16}\big\|\|F_2\|_{L^2}\big\|_{L^2(\mathbb{R})}\\
	\label{omega-n-L2}\approx&|kB|^{-\f13}\|\tilde{f}_1\|_{L^2L^2}+|kB|^{-\f16}\|\tilde{f}_2\|_{L^2L^2}.
	\end{align}
	Via integration by parts, we also have
	\begin{align*}
	&\Re\langle\partial_{\tau}w_k^n- [\partial_{r}^2-(\f{k^2-\f14}{r^2}+\f{r^2}{16}-\f12)]w_k^n+\f{i\beta_k}{r^2}w_k^n+f_1-\partial_{r}f_2,e^{2c_2|kB|^{\f13}\tau}w_k^n\rangle\\
	=&\f12\partial_{\tau}\|\tilde{w}_k^n\|_{L^2}^2+\langle- [\partial_{r}^2-(\f{k^2-\f14}{r^2}+\f{r^2}{16}-\f12)]\tilde{w}_k^n-c_2|kB|^{\f13}\tilde{w}_k^n+\tilde{f}_1-\partial_{r}\tilde{f}_2,\tilde{w}_k^n\rangle=0.
	\end{align*}
	Together with Lemma \ref{trivial w' lemma}, this yields that there exists a constant $c_0>0$ such that
	\begin{align*}
	&c_2|kB|^{\f13}\|\tilde{w}_k^n\|_{L^2}^2\\\geq&\f12\partial_{\tau}\|\tilde{w}_k^n\|_{L^2}^2+c_0(\|\partial_{r}\tilde{w}_k^n\|_{L^2}^2+\|(\f{|k|}{r}+r)\tilde{w}_k^n\|_{L^2}^2)+\Re\langle (\tilde{f}_1 -\partial_{r}\tilde{f}_2),\tilde{w}_k^n\rangle\\
	=&\f12\partial_{\tau}\|\tilde{w}_k^n\|_{L^2}^2+c_0(\|\partial_{r}\tilde{w}_k^n\|_{L^2}^2+\|(\f{|k|}{r}+r)\tilde{w}_k^n\|_{L^2}^2) +\Re\langle \tilde{f}_1 ,\tilde{w}_k^n\rangle+\Re\langle \tilde{f}_2,\partial_{r}\tilde{w}_k^n\rangle.
	\end{align*}
	By Cauchy-Schwarz inequality, we obtain
	\begin{align*}
	&\partial_{\tau}\|\tilde{w}_k^n\|_{L^2}^2+2c_0\Big(\|\partial_{r}\tilde{w}_k^n\|_{L^2}^2+\|(\f{|k|}{r}+r)\tilde{w}_k^n\|_{L^2}^2\Big) \\
	\lesssim&\Big(\|\tilde{f}_1\|_{L^2}+|kB|^{\f16}\|\tilde{f}_2\|_{L^2}\Big)\Big(\|\tilde{w}_k^n\|_{L^2}+|kB|^{-\f16}\|\partial_{r}\tilde{w}_k^n\|_{L^2}\Big)+|kB|^{\f16}\|\tilde{w}_k^n\|_{L^2}^2.
	\end{align*}
	Combining with \eqref{omega-n-L2}, we arrive at
	\begin{align*}
	&\|\tilde{w}_k^n\|_{L^{\infty}L^2}^2+\|\partial_{r}\tilde{w}_k^n\|_{L^2L^2}^2+\|(\f{|k|}{r}+r)\tilde{w}_k^n\|_{L^2L^2}^2+|kB|^{\f13}\|\tilde{w}_k^n\|_{L^2L^2}^2\\
	\lesssim&(\|\tilde{f}_1\|_{L^2L^2}+|kB|^{\f16}\|\tilde{f}_2\|_{L^2L^2})(\|\tilde{w}_k^n\|_{L^2L^2}+|kB|^{-\f16}\|\partial_{r}\tilde{w}_k^n\|_{L^2L^2})+|kB|^{\f13}\|\tilde{w}_k^n\|_{L^2L^2}^2\\
	\lesssim&|kB|^{-\f13}(\|\tilde{f}_1\|_{L^2L^2}+|kB|^{\f16}\|\tilde{f}_2\|_{L^2L^2})^2\lesssim|kB|^{-\f13}\|\tilde{f}_1\|_{L^2L^2}^2+\|\tilde{f}_2\|_{L^2L^2}^2.
	\end{align*}
	This completes the proof of Lemma \ref{nonzero frequency nonlinear part2}.
\end{proof}

We then proceed in the same fashion and derive the space-time estimate for the nonlinear part $w_k^n$ in weighted $L^2$ space.

\begin{lemma}\label{nonzero frequency nonlinear part2-r}Let $w_k^n$ be a solution to \eqref{nonzero mode nonlinear} with zero initial data. For any $|k|\geq1(k\in\mathbb{Z})$, it holds
	\begin{align*}
	&\|e^{c_2|kB|^{\f13}\tau}w_k^n\|_{L^{\infty}X}^2+|kB|^{\f13}\|e^{c_2|kB|^{\f13}\tau}w_k^n\|_{L^2X}^2\\&+\|e^{c_2|kB|^{\f13}\tau}\partial_{r}w_k^n\|_{L^2X}^2
	+\|e^{c_2|kB|^{\f13}\tau}(\f{|k|}{r}+r)w_k^n\|_{L^2X}^2\\
	\lesssim&|kB|^{-\f13}\Big(\|e^{c_2|kB|^{\f13}\tau}f_1\|_{L^2X}^2+\|e^{c_2|kB|^{\f13}\tau}\f{f_2}{r}\|_{L^2X}^2\Big)+\|e^{c_2|kB|^{\f13}\tau}f_2\|_{L^2X}^2,
	\end{align*}
	where $w_k^n$ satisfies the equation \eqref{nonzero mode nonlinear}.
\end{lemma}
\begin{proof}We still take the Fourier transform with respective to $t$ and get
	\begin{align*}
	\hat{w}_k^n(\lambda,r)=\int_0^{\infty}\tilde{w}_k^n(\tau,r)e^{-i\tau\lambda}d\tau,\quad F_j(\lambda,r,k)=\int_0^{\infty}\tilde{f}_j(\tau,r,k)e^{-i\tau\lambda}d\tau \quad \textrm{with}\  j=1,2.
	\end{align*}
The inhomogeneous equation \eqref{nonzero mode nonlinear} then becomes 
	\begin{align*}
	(i\lambda+L_k-c_2|kB|^{\f13})\hat{w}_k^n(\lambda,r)=-F_1+\partial_{r}F_2=-F_1+r\partial_{r}(\f{F_2}{r})+\f{F_2}{r}.
	\end{align*}
	As in Lemma \ref{nonzero frequency nonlinear part2}, we decompose $\hat{w}_k^n=\hat{w}_k^{n,1}+\hat{w}_k^{n,2}$, where $\hat{w}_k^{n,1}$ and $\hat{w}_k^{n,2}$ satisfy
	\begin{align*}
	(i\lambda+L_k-c_2|kB|^{\f13})\hat{w}_k^{n,1}=-F_1+\f{F_2}{r} \quad \textrm{and} \quad
	(i\lambda+L_k-c_2|kB|^{\f13})\hat{w}_k^{n,2}=r\partial_{r}(\f{F_2}{r}),
	\end{align*}
	respectively.
	It follows from Proposition \ref{resolvent estimate-r1-1} that
	\begin{align*}
	&|kB|^{\f16}\|\partial_{r}\hat{w}_k^{n,1}\|_{X}+|kB|^{\f13}\|\hat{w}_k^{n,1}\|_{X}\leq C \|F_1-\f{F_2}{r}\|_{X},
	\end{align*}
	and
	\begin{align*}
	&\|\partial_{r}\hat{w}_k^{n,2}\|_{X}+|kB|^{\f16}\|\hat{w}_k^{n,2}\|_{X}\leq C\|\partial_{r}(\f{F_2}{r})\|_{H^{-1}}\leq C\|\f{F_2}{r}\|_{L^2}\leq C\|F_2\|_{X}.
	\end{align*}
	Hence we obtain
	\begin{align*}
	|kB|^{-\f16}\|\partial_{r}\hat{w}_k^{n}\|_{X}+\|\hat{w}_k^n\|_{X}\lesssim |kB|^{-\f13}\|F_1-\f{F_2}{r}\|_{X}+|kB|^{-\f16}\|F_2\|_{X}.
	\end{align*}
	Therefore, utilizing Plancherel's theorem we deduce
	\begin{align}
	\nonumber&|kB|^{-\f16}\|\partial_{r}\tilde{w}_k^{n}\|_{L^2X}+\|\tilde{w}_k^n\|_{L^2X}\\
	\nonumber	\approx&|kB|^{-\f16}\big\|\|\partial_{r}\hat{w}_k^{n}\|_{X}\big\|_{L^2(\mathbb{R})}+\big\|\|\hat{w}_k^n\|_{X}\big\|_{L^2(\mathbb{R})}\\
	\nonumber\lesssim& |kB|^{-\f13}\big\|\|F_1-\f{F_2}{r}\|_{X}\big\|_{L^2(\mathbb{R})}+|kB|^{-\f16}\big\|\|F_2\|_{X}\big\|_{L^2(\mathbb{R})}\\
	\label{omega-n-L2-r}\approx&|kB|^{-\f13}\|\tilde{f}_1-\f{\tilde{f}_2}{r}\|_{L^2X}+|kB|^{-\f16}\|\tilde{f}_2\|_{L^2X}.
	\end{align}
	In addition, integration by parts yields
	\begin{align*}
	&\Re\langle\partial_{\tau}w_k^n- [\partial_{r}^2-(\f{k^2-\f14}{r^2}+\f{r^2}{16}-\f12)]w_k^n+\f{i\beta_k}{r^2}w_k^n+f_1-\partial_{r}f_2,e^{2c_2|kB|^{\f13}\tau}\f{w_k^n}{r^2}\rangle\\
	=&\f12\partial_{\tau}\|\tilde{w}_k^n\|_{L^2}^2+\langle- [\partial_{r}^2-(\f{k^2-\f14}{r^2}+\f{r^2}{16}-\f12)]\tilde{w}_k^n-c_2|kB|^{\f13}\tilde{w}_k^n+\tilde{f}_1-\partial_{r}\tilde{f}_2,\f{\tilde{w}_k^n}{r^2}\rangle=0.
	\end{align*}
Together with Lemma \ref{trivial w'/r}, this implies that there exists a constant $c_0>0$ such that
	\begin{align*}
	&c_2|kB|^{\f13}\|\tilde{w}_k^n\|_{X}^2\\\geq&\f12\partial_{\tau}\|\tilde{w}_k^n\|_{X}^2+c_0(\|\partial_{r}\tilde{w}_k^n\|_{X}^2+\|(\f{|k|}{r}+r)\tilde{w}_k^n\|_{X}^2)+\Re\langle (\tilde{f}_1 -\partial_{r}\tilde{f}_2),\f{\tilde{w}_k^n}{r^2}\rangle\\
	=&\f12\partial_{\tau}\|\tilde{w}_k^n\|_{X}^2+c_0(\|\partial_{r}\tilde{w}_k^n\|_{X}^2+\|(\f{|k|}{r}+r)\tilde{w}_k^n\|_{X}^2) +\Re\langle \f{\tilde{f}_1}{r} ,\f{\tilde{w}_k^n}{r}\rangle-\Re\langle \f{\tilde{f}_2}{r},\f{\partial_{r}\tilde{w}_k^n}{r}-2\f{\tilde{w}_k^n}{r^2}\rangle.
	\end{align*}
	Thus we obtain
	\begin{align*}
	&\partial_{\tau}\|\tilde{w}_k^n\|_{X}^2+2c_0(\|\partial_{r}\tilde{w}_k^n\|_{X}^2+\|(\f{|k|}{r}+r)\tilde{w}_k^n\|_{X}^2) \\ \lesssim&\|\tilde{f}_1\|_{X}\|\tilde{w}_k^n\|_{X}+\|\tilde{f}_2\|_{X}(\|\partial_{r}\tilde{w}_k^n\|_{X}+\|\f{\tilde{w}_k^n}{r}\|_{X})+|kB|^{\f16}\|\tilde{w}_k^n\|_{X}^2.
	\end{align*}
	With \eqref{omega-n-L2-r}, we conclude
	\begin{align*}
	&\|\tilde{w}_k^n\|_{L^{\infty}X}^2+|kB|^{\f13}\|\tilde{w}_k^n\|_{L^2X}^2+\|\partial_{r}\tilde{w}_k^n\|_{L^2X}^2+\|(\f{|k|}{r}+r)\tilde{w}_k^n\|_{L^2X}^2\\
	\lesssim&|kB|^{-\f13}(\|\tilde{f}_1\|_{L^2X}+\|\f{\tilde{f}_2}{r}\|_{L^2X})^2+\|\tilde{f}_2\|_{L^2X}^2.
	\end{align*}
	This completes the proof of Lemma \ref{nonzero frequency nonlinear part2-r}.
\end{proof}

Now we can state the two key conclusions of this section. We call them the space-time estimates for equation (\refeq{main equation}).
\begin{proposition}\label{nonzero nonlinear-1}Assume $w_k$ is a solution to \eqref{main equation} with $w_k(0)\in L^2$ and $e^{c'|kB|^{\f13}\tau}f_1$, $e^{c'|kB|^{\f13}\tau}f_2\in L^2L^2$ for some $c'>0$, then there exist constants $C,c>0$ independent of $B,k$ so that
	\begin{align}
	\label{nonzero nonlinear}&\|e^{c|kB|^{\f13}\tau}w_k\|_{L^{\infty}L^2}+|kB|^{\f16}\|e^{c|kB|^{\f13}\tau}w_k\|_{L^2L^2}\\&+\|e^{c|kB|^{\f13}\tau}\partial_{r}w_k\|_{L^2L^2}
	\nonumber+\|e^{c|kB|^{\f13}\tau}(\f{|k|}{r}+r)w_k\|_{L^2L^2}\\
	\nonumber\leq & C\Big[|kB|^{-\f16}\|e^{c|kB|^{\f13}\tau}f_1\|_{L^2L^2}+\|e^{c|kB|^{\f13}\tau}f_2\|_{L^2L^2}+\|w_k(0)\|_{L^2}\Big].
	\end{align}
\end{proposition}
\begin{proof}This result follows directly from our prepared Lemma \ref{nonzero frequency linear part} and Lemma \ref{nonzero frequency nonlinear part2} by taking $c=\min\{c',c_2\}$.
\end{proof}

\begin{proposition}\label{nonzero nonlinear-2}Assume $w_k$ is a solution to \eqref{main equation} with $w_k(0)\in X$ and $e^{c'|kB|^{\f13}\tau}f_1,e^{c'|kB|^{\f13}\tau}f_2, $\\$e^{c'|kB|^{\f13}\tau}\f{f_2}{r}\in L^2X$ for some $c'>0$, then there exist constants $C,c>0$ independent of $B,k$ so that
	\begin{align}
	\label{nonzero nonlinear2}&\|e^{c|kB|^{\f13}\tau}w_k\|_{L^{\infty}X}+|kB|^{\f16}\|e^{c|kB|^{\f13}\tau}w_k\|_{L^2X}\\&+\|e^{c|kB|^{\f13}\tau}\partial_{r}w_k\|_{L^2X}
	\nonumber+\|e^{c|kB|^{\f13}\tau}(\f{|k|}{r}+r)w_k\|_{L^2X}\\
	\nonumber\leq&C\Big[|kB|^{-\f16}(\|e^{c|kB|^{\f13}\tau}f_1\|_{L^2X}+\|e^{c|kB|^{\f13}\tau}\f{f_2}{r}\|_{L^2X})+\|e^{c|kB|^{\f13}\tau}f_2\|_{L^2X}+\|w_k(0)\|_{X}\Big].
	\end{align}
\end{proposition}
\begin{proof}This result follows directly from Lemma \ref{nonzero frequency linear part-r} and Lemma \ref{nonzero frequency nonlinear part2-r} by taking $c=\min\{c',c_2\}$.
\end{proof}

\subsection{Space-time estimates for zero frequency}\label{3.3-zero}

Recall from \eqref{main equation} that $w_0$ satisfies 
\begin{align*}
&\partial_\tau w_0- [\partial_r^2+\f{1}{4r^2}-\f{r^2}{16}+\f12)]w_0-\sum_{l\in\mathbb{Z}/\{0\}}il(\f14-\f{1}{2r^2})w_l\breve{\varphi}_{-l}+\sum_{l\in\mathbb{Z}/\{0\}}il\partial_{r}(\f{w_l\breve{\varphi}_{-l}}{r})=0,
\end{align*}
where $\breve{\varphi}_k$ verifies $(\partial_{r}^2+\f{1}{r}\partial_{r}-\f{k^2}{r^2})\breve{\varphi}_k=fw_k$ with $f=\f{e^{-\f{r^2}{8}}}{r^{\f12}}$. For simplicity, we denote the nonlinear terms as
\begin{align*}
N=-\sum_{l\in\mathbb{Z}/\{0\}}il(\f14-\f{1}{2r^2})w_l\breve{\varphi}_{-l}+\sum_{l\in\mathbb{Z}/\{0\}}il\partial_{r}(\f{w_l\breve{\varphi}_{-l}}{r}).
\end{align*}
And we obtain 
\begin{proposition}
\label{zero nonlinear}Let $w_0$ be the solution to \eqref{main equation} with initial data $\f{w_0(0)}{r}\in L^2$, then $w_0(\tau)$ obeys
\begin{align*}
&\|\f{w_0}{r}\|_{L^{\infty}L^2}+\|\f{\partial_rw_0}{r}\|_{L^2L^2}+\|\f{w_0}{r^2}\|_{L^2L^2}+\|w_0\|_{L^2L^2}\\
\lesssim& \sum_{l\in\mathbb{Z}/\{0\}}|l|\|w_l\breve{\varphi}_{-l}\|_{L^2L^2}+ \sum_{l\in\mathbb{Z}/\{0\}}|l|\|\f{w_l\breve{\varphi}_{-l}}{r^2}\|_{L^2L^2}+\|\f{w_0(0)}{r}\|_{L^2}.
\end{align*}
\end{proposition}
\begin{proof}
For any $C^2$ function $g$, one can check
\begin{align*}
-w_0''=-g^{-1}\partial_r[r^2g^2\partial_r(r^{-2}g^{-1}w_0)]-\f{2}{r}w_0'-g^{-1}[g''+(\f{2}{r}g)']w_0.
\end{align*}
We will need a real-valued function $g$ satisfying
\begin{align}\label{dif eqn for g-2}
\f{g''}{g}+\f{2}{r}\f{g'}{g}=-\f{3}{r^2}.
\end{align} The existence of $g$ is guaranteed by Lemma \ref{Appendix real function g} in the Appendix (with $A=2, B=-3$). Employing \eqref{basic equality-1} and \eqref{dif eqn for g-2}, we deduce
\begin{align*}
0=&\langle \partial_\tau w_0- [\partial_r^2+\f{1}{4r^2}-\f{r^2}{16}+\f12)]w_0+N,\f{w_0}{r^2}\rangle\\
=&\f12\partial_\tau\|\f{w_0}{r}\|_{L^2}^2+\|rg\partial_r(r^{-2}g^{-1}w_0)\|_{L^2}^2-3\|\f{w_0}{r^2}\|_{L^2}^2-\langle g^{-1}[g''+(\f{2}{r}g)']w_0,\f{w_0}{r^2}\rangle\\
&+\langle (-\f{1}{4r^2}+\f{r^2}{16}-\f12)w_0,\f{w_0}{r^2}\rangle+\Re\langle N,\f{w_0}{r^2}\rangle\\
=&\f12\partial_\tau\|\f{w_0}{r}\|_{L^2}^2+\|rg\partial_r(r^{-2}g^{-1}w_0)\|_{L^2}^2-\langle g^{-1}(g''+\f{2}{r}g')w_0,\f{w_0}{r^2}\rangle\\
&+\langle (-\f{5}{4r^2}+\f{r^2}{16}-\f12)w_0,\f{w_0}{r^2}\rangle+\Re\langle N,\f{w_0}{r^2}\rangle\\
=&\f12\partial_\tau\|\f{w_0}{r}\|_{L^2}^2+\|rg\partial_r(r^{-2}g^{-1}w_0)\|_{L^2}^2+\langle (\f{7}{4r^2}+\f{r^2}{16}-\f12)w_0,\f{w_0}{r^2}\rangle\\
&-\Re\langle \sum_{l\in\mathbb{Z}/\{0\}}il(\f14-\f{1}{2r^2})w_l\breve{\varphi}_{-l},\f{w_0}{r^2}\rangle-\Re\langle \sum_{l\in\mathbb{Z}/\{0\}}il\f{w_l\breve{\varphi}_{-l}}{r},\f{\partial_{r}w_0}{r^2}-\f{2}{r^3}w_0\rangle.
\end{align*}
Note that there exists $c_0>0$ such that $\f{7}{4r^2}+\f{r^2}{16}-\f12\geq c_0 (\f{1}{r^2}+r^2)$, we hence obtain
\begin{align}
\label{zero mode-1}0\geq&\f12\partial_\tau\|\f{w_0}{r}\|_{L^2}^2+\|rg\partial_r(r^{-2}g^{-1}w_0)\|_{L^2}^2+c_0\langle (\f{1}{r^2}+r^2)w_0,\f{w_0}{r^2}\rangle\\
\nonumber&-\Re\langle \sum_{l\in\mathbb{Z}/\{0\}}il(\f14-\f{1}{2r^2})w_l\breve{\varphi}_{-l},\f{w_0}{r^2}\rangle-\Re\langle \sum_{l\in\mathbb{Z}/\{0\}}il\f{w_l\breve{\varphi}_{-l}}{r},\f{\partial_{r}w_0}{r^2}-\f{2}{r^3}w_0\rangle.
\end{align}
In addition, we also find
\begin{align}\label{zero mode-1.1}
0=&\langle \partial_\tau w_0- [\partial_r^2+\f{1}{4r^2}-\f{r^2}{16}+\f12)]w_0,\f{w_0}{r^2}\rangle\\
\nonumber&-\Re\langle \sum_{l\in\mathbb{Z}/\{0\}}il(\f14-\f{1}{2r^2})w_l\breve{\varphi}_{-l},\f{w_0}{r^2}\rangle-\Re\langle \sum_{l\in\mathbb{Z}/\{0\}}il\f{w_l\breve{\varphi}_{-l}}{r},\f{\partial_{r}w_0}{r^2}-\f{2}{r^3}w_0\rangle\\
\nonumber=&\f12\partial_\tau\|\f{w_0}{r}\|_{L^2}^2+\|\f{\partial_rw_0}{r}\|_{L^2}^2-\f{13}{4}\|\f{w_0}{r^2}\|_{L^2}^2+\f{1}{16}\|w_0\|_{L^2}^2-\f12\|\f{w_0}{r}\|_{L^2}^2\\
\nonumber&-\Re\langle \sum_{l\in\mathbb{Z}/\{0\}}il(\f14-\f{1}{2r^2})w_l\breve{\varphi}_{-l},\f{w_0}{r^2}\rangle-\Re\langle \sum_{l\in\mathbb{Z}/\{0\}}il\f{w_l\breve{\varphi}_{-l}}{r},\f{\partial_{r}w_0}{r^2}-\f{2}{r^3}w_0\rangle.
\end{align}
For any $C_0>0$, adding $C_0\times\eqref{zero mode-1}$  and \eqref{zero mode-1.1}, one obtains
\begin{align*}
&\f12C_0\partial_\tau\|\f{w_0}{r}\|_{L^2}^2+C_0\|rg\partial_r(r^{-2}g^{-1}w_0)\|_{L^2}^2+C_0c_0\langle (\f{1}{r^2}+r^2)w_0,\f{w_0}{r^2}\rangle\\
&+\f12\partial_\tau\|\f{w_0}{r}\|_{L^2}^2+\|\f{\partial_rw_0}{r}\|_{L^2}^2-\f{13}{4}\|\f{w_0}{r^2}\|_{L^2}^2+\f{1}{16}\|w_0\|_{L^2}^2-\f12\|\f{w_0}{r}\|_{L^2}^2\\
\leq&(C_0+1)\Big(\Re\langle \sum_{l\in\mathbb{Z}/\{0\}}il(\f14-\f{1}{2r^2})w_l\breve{\varphi}_{-l},\f{w_0}{r^2}\rangle+\Re\langle \sum_{l\in\mathbb{Z}/\{0\}}il\f{w_l\breve{\varphi}_{-l}}{r},\f{\partial_{r}w_0}{r^2}-\f{2}{r^3}w_0\rangle\Big).
\end{align*}
Choose $C_0$ being large enough such that 
\begin{align*}
    \f{C_0c_0}{2}\langle (\f{1}{r^2}+r^2)w_0,\f{w_0}{r^2}\rangle \geq \f{13}{4}\|\f{w_0}{r^2}\|_{L^2}^2+\f12\|\f{w_0}{r}\|_{L^2}^2.
\end{align*}
Then we deduce
\begin{align*}
&\partial_\tau\|\f{w_0}{r}\|_{L^2}^2+\|\f{\partial_rw_0}{r}\|_{L^2}^2+\langle (\f{1}{r^2}+r^2)w_0,\f{w_0}{r^2}\rangle\\
\lesssim&|\langle \sum_{l\in\mathbb{Z}/\{0\}}il(\f14-\f{1}{2r^2})w_l\breve{\varphi}_{-l},\f{w_0}{r^2}\rangle|+|\langle \sum_{l\in\mathbb{Z}/\{0\}}il\f{w_l\breve{\varphi}_{-l}}{r},\f{\partial_{r}w_0}{r^2}-\f{2}{r^3}w_0\rangle|.
\end{align*}
This yields
\begin{align*}
&\partial_\tau\|\f{w_0}{r}\|_{L^2}^2+\|\f{\partial_rw_0}{r}\|_{L^2}^2+\langle (\f{1}{r^2}+r^2)w_0,\f{w_0}{r^2}\rangle\\
\lesssim&\| \sum_{l\in\mathbb{Z}/\{0\}}il(\f14-\f{1}{2r^2})w_l\breve{\varphi}_{-l}\|_{L^2}^2+\| \sum_{l\in\mathbb{Z}/\{0\}}il\f{w_l\breve{\varphi}_{-l}}{r^2}\|_{L^2}^2.
\end{align*}
Hence we obtain
\begin{align*}
&\|\f{w_0}{r}\|_{L^{\infty}L^2}^2+\|\f{\partial_rw_0}{r}\|_{L^2L^2}^2+\|\f{w_0}{r^2}\|_{L^2L^2}^2+\|w_0\|_{L^2L^2}^2\\
\lesssim& \sum_{l\in\mathbb{Z}/\{0\}}|l|\|w_l\breve{\varphi}_{-l}\|_{L^2L^2}^2+ \sum_{l\in\mathbb{Z}/\{0\}}|l|\|\f{w_l\breve{\varphi}_{-l}}{r^2}\|_{L^2L^2}^2+\|\f{w_0(0)}{r}\|_{L^2}^2.
\end{align*}
And this completes the proof of Proposition \ref{zero nonlinear}.
\end{proof}

\section{\textbf{Nonlinear stability}}\label{4-Nonlinear stability}
This section is devoted to the proof of the main Theorem \ref{Nonlinear stability}. For the 2D Navier-Stokes equation, it is well-known that smooth initial data could lead to smooth global solutions. Here we are interested in solutions' asymptotically behaviors, especially for initial data close to the Taylor-Couette flow. We first recall the energy functional $E_k$ \eqref{basic energy}:
\begin{align*}
E_k=&\|e^{c|kB|^{\f13}\tau}w_k\|_{L^{\infty}L^2}+|kB|^{\f16}\Big(\|e^{c|kB|^{\f13}\tau}w_k\|_{L^2L^2}+\|e^{c|kB|^{\f13}\tau}\f{w_k}{r}\|_{L^{\infty}L^2}\\
\nonumber&+|k|\|e^{c|kB|^{\f13}\tau}\f{w_k}{r^2}\|_{L^2L^2}+|k|^{\f12}\|e^{c|kB|^{\f13}\tau}\f{w_k}{r^{\f32}}\|_{L^2L^{\infty}}\Big)
+|kB|^{\f13}\|e^{c|kB|^{\f13}\tau}\f{w_k}{r}\|_{L^2L^2},\quad|k|\geq1,\\
E_0=&|B|^{\f16}(\|\f{w_0}{r}\|_{L^{\infty}L^2}+\|\f{w_0}{r^{\f32}}\|_{L^2L^{\infty}}+\|\f{w_0}{r^2}\|_{L^2L^2}+\|w_0\|_{L^2L^2}),\quad k=0,
\end{align*}
where $w_k$ is a solution of the nonlinear vorticity formulation \eqref{main equation}.

Before we provide the proof, let us recall the  nonlinear vorticity formulation \eqref{main equation} from Section \ref{1.4-Derivation of equations}:
\begin{align*}
\left\{
\begin{aligned}
&\partial_{\tau}w_k+L_kw_k+f_1-\partial_{r}f_2=0,\\
&w_k(0)=w_k|_{\tau=0},\quad w_k|_{r=0,\infty}=0,
\end{aligned}
\right.
\end{align*}
where $L_k=- [\partial_r^2-(\f{k^2-\f14}{r^2}+\f{r^2}{16}-\f12)]+i\f{\beta_k}{r^2}$ with $\beta_k=kB$ and
\begin{align*}
&f_1=ik\sum_{l\in\mathbb{Z}}w_{l}\f{\partial_r\breve{\varphi}_{k-l}}{r}+\sum_{l\in\mathbb{Z}}i(k-l)(\f14-\f{1}{2r^2})w_l\breve{\varphi}_{k-l},\quad f_2=\sum_{l\in\mathbb{Z}}i(k-l)\f{w_l\breve{\varphi}_{k-l}}{r},
\end{align*}
with $\breve{\varphi}_k$ and $\varphi_k$ satisfy $(\partial_{r}^2+\f{1}{r}\partial_{r}-\f{k^2}{r^2})\breve{\varphi}_k=\f{e^{-\f{r^2}{8}}}{r^{\f12}}w_k$ and $\varphi_k:=r^{\f12}\breve{\varphi}_k$. For $\varphi_k$, we further deduce
\begin{align*}
r^{-\f12}(\partial_{r}^2-\f{k^2-\f14}{r^2})\varphi_k=r^{-\f12}e^{-\f{r^2}{8}}w_k,
\end{align*}
which yields
\begin{align}
\label{5-11111}(\partial_{r}^2-\f{k^2-\f14}{r^2})\varphi_k=e^{-\f{r^2}{8}}w_k.
\end{align}
Employing \eqref{5-11111}, the nonlinear terms $f_1,f_2$ can be further expressed as
\begin{align*}
&f_1=ik\sum_{l\in\mathbb{Z}}w_{l}\f{\partial_r(r^{-\f12}\varphi_{k-l})}{r}+\sum_{l\in\mathbb{Z}}i(k-l)(\f14-\f{1}{2r^2})\f{w_l\varphi_{k-l}}{r^{\f12}},\quad f_2=\sum_{l\in\mathbb{Z}}i(k-l)\f{w_l\varphi_{k-l}}{r^{\f32}}.
\end{align*}

The following is the main result of this paper: 
\begin{theorem}\label{Nonlinear stability}For any $|B|\geq1$, there exist constants $c_0, C>0$ independent of $B$ so that if the initial data satisfies \begin{align*}
 \sum_{k\in\mathbb{Z}/\{0\}}\|w_k(0)\|_{L^2}+\sum_{k\in\mathbb{Z}/\{0\}}|kB|^{\f16}\|\f{w_k(0)}{r}\|_{L^2}+|B|^{\f16}\|\f{w_0(0)}{r}\|_{L^2}\leq c_0|B|^\f13,
  \end{align*}
  then the solution to the nonlinear vorticity formulation \eqref{main equation} exists globally in time, and there exists a constant $C>0$ independent of $B$ such that
\begin{align*}
&\sum_{k\in\mathbb{Z}}E_k\leq C (\sum_{k\in\mathbb{Z}/\{0\}}\|w_k(0)\|_{L^2}+\sum_{k\in\mathbb{Z}/\{0\}}|kB|^{\f16}\|\f{w_k(0)}{r}\|_{L^2}+|B|^{\f16}\|\f{w_0(0)}{r}\|_{L^2})\leq Cc_0|B|^\f13.
\end{align*}
\end{theorem}

\subsection{Interactions between zero frequency and non-zero frequency}
Now we establish the estimate involving the interaction between zero frequency and other frequency. Let's start with the terms including $w_0\varphi_{k}$.
\begin{lemma}\label{zero-Nonlinear estimate1}For any $k\in \mathbb{Z}\backslash \{0\}$, it holds that
\begin{align*}
|k|\Big(&\|e^{c|kB|^{\f13}\tau}\f{w_0\varphi_{k}}{r^{\f12}}\|_{L^2L^2}+\|e^{c|kB|^{\f13}\tau}\f{w_0\varphi_{k}}{r^{\f32}}\|_{L^2L^2}\\&+\|e^{c|kB|^{\f13}\tau}\f{w_0\varphi_{k}}{r^{\f52}}\|_{L^2L^2}
+\|e^{c|kB|^{\f13}\tau}\f{w_0\varphi_{k}}{r^{\f72}}\|_{L^2L^2}\Big)
\lesssim|B|^{-\f12}|k|^{-\f56}E_0E_k.
\end{align*}
\end{lemma}
\begin{proof}Applying Hölder's inequalities for the nonlinear terms in $L^2L^2$ norm, we obtain
\begin{align*}
&|k|\Big(\|e^{c|kB|^{\f13}\tau}\f{w_0\varphi_{k}}{r^{\f12}}\|_{L^2L^2}+\|e^{c|kB|^{\f13}\tau}\f{w_0\varphi_{k}}{r^{\f32}}\|_{L^2L^2}\\&+\|e^{c|kB|^{\f13}\tau}\f{w_0\varphi_{k}}{r^{\f52}}\|_{L^2L^2}
+\|e^{c|kB|^{\f13}\tau}\f{w_0\varphi_{k}}{r^{\f72}}\|_{L^2L^2}\Big)\\
\leq&|k|\|\f{w_0}{r}\|_{L^\infty L^2}\Big(\|e^{c|kB|^{\f13}\tau}r^{\f12}\varphi_{k}\|_{L^2 L^\infty}+\|e^{c|kB|^{\f13}\tau}\f{\varphi_{k}}{r^{\f12}}\|_{L^2 L^\infty}\\\qquad&+\|e^{c|kB|^{\f13}\tau}\f{\varphi_{k}}{r^{\f32}}\|_{L^2 L^\infty}
+\|e^{c|kB|^{\f13}\tau}\f{\varphi_{k}}{r^{\f52}}\|_{L^2 L^\infty}\Big).
\end{align*}
Combining above inequalities with \eqref{5-11111} and the calculation Lemma \ref{Appendix A3} proved in appendix (with $\beta=-2, -1, 0, 1$), we derive
\begin{align*}
&|k|\Big(\|e^{c|kB|^{\f13}\tau}r^{\f12}\varphi_{k}\|_{L^2 L^\infty}+\|e^{c|kB|^{\f13}\tau}\f{\varphi_{k}}{r^{\f12}}\|_{L^2 L^\infty}\\&+\|e^{c|kB|^{\f13}\tau}\f{\varphi_{k}}{r^{\f32}}\|_{L^2 L^\infty}
+\|e^{c|kB|^{\f13}\tau}\f{\varphi_{k}}{r^{\f52}}\|_{L^2 L^\infty}\Big)\\
\lesssim&|k|^{-\f12}\Big(\|e^{c|kB|^{\f13}\tau}r^2e^{-\f{r^2}{8}}w_{k}\|_{L^2 L^2}+\|e^{c|kB|^{\f13}\tau}re^{-\f{r^2}{8}}w_{k}\|_{L^2 L^2}\\&+\|e^{c|kB|^{\f13}\tau}e^{-\f{r^2}{8}}w_{k}\|_{L^2 L^2}
+\|e^{c|kB|^{\f13}\tau}e^{-\f{r^2}{8}}\f{w_{k}}{r}\|_{L^2 L^2}\Big)\\
\lesssim&|k|^{-\f12}\|e^{c|kB|^{\f13}\tau}\f{w_{k}}{r}\|_{L^2L^2}.
\end{align*}
Therefore, we deduce
\begin{align*}
&|k|\Big(\|e^{c|kB|^{\f13}\tau}\f{w_0\varphi_{k}}{r^{\f12}}\|_{L^2L^2}+\|e^{c|kB|^{\f13}\tau}\f{w_0\varphi_{k}}{r^{\f32}}\|_{L^2L^2}\\&+\|e^{c|kB|^{\f13}\tau}\f{w_0\varphi_{k}}{r^{\f52}}\|_{L^2L^2}+\|e^{c|kB|^{\f13}\tau}\f{w_0\varphi_{k}}{r^{\f72}}\|_{L^2L^2}\Big)\\
\lesssim&|k|^{-\f12}\|\f{w_0}{r}\|_{L^\infty L^2}\|e^{c|kB|^{\f13}\tau}\f{w_{k}}{r}\|_{L^2L^2}\leq|k|^{-\f12}|B|^{-\f16}E_0|kB|^{-\f13}E_k\\=&|B|^{-\f12}|k|^{-\f56}E_0E_k.
\end{align*}
\end{proof}
Then we treat $w_0\varphi_{k}'$ terms.
\begin{lemma}\label{zero-Nonlinear estimate2}For any $k\in \mathbb{Z}\backslash \{0\}$, we have
\begin{align*}
&|k|\Big(\|e^{c|kB|^{\f13}\tau}\f{w_0\varphi_{k}'}{r^{\f52}}\|_{L^2L^2}+\|e^{c|kB|^{\f13}\tau}\f{w_0\varphi_{k}'}{r^{\f32}}\|_{L^2L^2}\Big)\lesssim|k|^{-\f16}|B|^{-\f13}E_0E_k.
\end{align*}
\end{lemma}
\begin{proof}Employing Hölder's inequalities, we can prove
\begin{align*}
&|k|\Big(\|e^{c|kB|^{\f13}\tau}\f{w_0\varphi_{k}'}{r^{\f52}}\|_{L^2L^2}+\|e^{c|kB|^{\f13}\tau}\f{w_0\varphi_{k}'}{r^{\f32}}\|_{L^2L^2}\Big)\\
\leq&|k|\|\f{w_0}{r^{\f32}}\|_{L^2L^{\infty}}\Big(\|e^{c|kB|^{\f13}\tau}\f{\varphi_{k}'}{r}\|_{L^{\infty}L^2}+\|e^{c|kB|^{\f13}\tau}\varphi_{k}'\|_{L^{\infty}L^2}\Big).
\end{align*}
This together with \eqref{5-11111} and Lemma \ref{Appendix A3} yields
\begin{align*}
|k|\Big(\|e^{c|kB|^{\f13}\tau}\f{\varphi_{k}'}{r}\|_{L^{\infty}L^2}+\|e^{c|kB|^{\f13}\tau}\varphi_{k}'\|_{L^{\infty}L^2}\Big)\lesssim\|e^{c|kB|^{\f13}\tau}\f{w_k}{r}\|_{L^{\infty}L^2}.
\end{align*}
Hence we deduce
\begin{align*}
&|k|\Big(\|e^{c|kB|^{\f13}\tau}\f{w_0\varphi_{k}'}{r^{\f52}}\|_{L^2L^2}+\|e^{c|kB|^{\f13}\tau}\f{w_0\varphi_{k}'}{r^{\f32}}\|_{L^2L^2}\Big)\\
\lesssim&\|\f{w_0}{r^{\f32}}\|_{L^2L^{\infty}}\|e^{c|kB|^{\f13}\tau}\f{w_k}{r}\|_{L^{\infty}L^2}\leq |B|^{-\f16}E_0|kB|^{-\f16}E_k=|k|^{-\f16}|B|^{-\f13}E_0E_k.
\end{align*}
\end{proof}

For quadratic terms containing $w_k\partial_r(r^{-\f12}\varphi_0)$, we have
\begin{lemma}\label{zero-Nonlinear estimate3}For any $k\in \mathbb{Z}\backslash \{0\}$, the following inequality holds
\begin{align*}
&|k|\Big(\|e^{c|kB|^{\f13}\tau}\f{w_k\partial_r(r^{-\f12}\varphi_0)}{r}\|_{L^2L^2}+\|e^{c|kB|^{\f13}\tau}\f{w_k\partial_r(r^{-\f12}\varphi_0)}{r^2}\|_{L^2L^2}\Big)\lesssim|k|^{-\f16}|B|^{-\f13}E_0E_k.
\end{align*}
\end{lemma}
\begin{proof}Recall that $\varphi_0$ and $w_0$ satisfy 
\begin{align*}
(\partial_{r}^2+\f{1}{r}\partial_{r})(r^{-\f12}\varphi_0)=\f{e^{-\f{r^2}{8}}}{r^{\f12}}w_0.
\end{align*}
Construct $G:=\partial_{r}(r^{-\f12}\varphi_0)$, one can check $\partial_r(r^{-\f12}\varphi_0)=G$.
Observing the inequality as below,
\begin{align*}
&\|G\|_{L^{\infty}}+\|\f{G}{r}\|_{L^{\infty}}=\|\f{\int_0^rs^{\f12}e^{-\f{s^2}{8}}w_0ds}{r}\|_{L^{\infty}}+\|\f{\int_0^rs^{\f12}e^{-\f{s^2}{8}}w_0ds}{r^2}\|_{L^{\infty}}\\
\leq&\Big\|\f{(\int_0^rs^3e^{-\f{s^2}{4}}ds)^{\f12}}{r}\|\f{w_0}{r}\|_{L^2}\Big\|_{L^{\infty}}+\Big\|\f{(\int_0^rs^3e^{-\f{s^2}{4}}ds)^{\f12}}{r^2}\|\f{w_0}{r}\|_{L^2}\Big\|_{L^{\infty}}\\
\leq&\Big(\|\f{(\int_0^rs^3e^{-\f{s^2}{4}}ds)^{\f12}}{r}\|_{L^{\infty}}+\|\f{(\int_0^rs^3e^{-\f{s^2}{4}}ds)^{\f12}}{r^2}\|_{L^{\infty}}\Big)\|\f{w_0}{r}\|_{L^2}\lesssim\|\f{w_0}{r}\|_{L^2}.
\end{align*}
Together with Hölder's inequality, we conclude
\begin{align*}
&|k|\Big(\|e^{c|kB|^{\f13}\tau}\f{w_k\partial_r(r^{-\f12}\varphi_0)}{r}\|_{L^2L^2}+\|e^{c|kB|^{\f13}\tau}\f{w_k\partial_r(r^{-\f12}\varphi_0)}{r^2}\|_{L^2L^2}\Big)\\
\lesssim&|k|\|e^{c|kB|^{\f13}\tau}\f{w_k}{r}\|_{L^2L^2}(\|G\|_{L^{\infty}L^{\infty}}+\|\f{G}{r}\|_{L^{\infty}L^{\infty}})\lesssim|k|\|e^{c|kB|^{\f13}\tau}\f{w_k}{r}\|_{L^2L^2}\|\f{w_0}{r}\|_{L^{\infty}L^2}\\
\leq&|kB|^{-\f16}|B|^{-\f16}E_kE_0=|k|^{-\f16}|B|^{-\f13}E_kE_0.
\end{align*}
\end{proof}

\subsection{Interactions between non-zero frequency spaces}
We now move to control nonlinear terms containing only non-zero frequency.
Firstly, observing the following inequality
\begin{align}
\label{basic 1}|kB|^{\f13}\leq|lB|^{\f13}+|(k-l)B|^{\f13} \quad \textrm{for any}\  k, l\in\mathbb{Z},
\end{align}
we obtain

\begin{lemma}\label{Nonlinear estimate1}For any $k\in \mathbb{Z}\backslash \{0\}$, we have
\begin{align*}
&\|e^{c|kB|^{\f13}\tau}\sum_{l\in\mathbb{Z}\backslash\{0,k\}}i(k-l)\f{w_l\varphi_{k-l}}{r^{\f12}}\|_{L^2L^2}+\|e^{c|kB|^{\f13}\tau}\sum_{l\in\mathbb{Z}\backslash\{0,k\}}i(k-l)\f{w_l\varphi_{k-l}}{r^{\f32}}\|_{L^2L^2}\\
&+\|e^{c|kB|^{\f13}\tau}\sum_{l\in\mathbb{Z}\backslash\{0,k\}}i(k-l)\f{w_l\varphi_{k-l}}{r^{\f52}}\|_{L^2L^2}+\|e^{c|kB|^{\f13}\tau}\sum_{l\in\mathbb{Z}\backslash\{0,k\}}i(k-l)\f{w_l\varphi_{k-l}}{r^{\f72}}\|_{L^2L^2}\\
\lesssim &|B|^{-\f12}\sum_{l\in\mathbb{Z}\backslash\{0,k\}}|l|^{-\f13}|k-l|^{-\f23}E_lE_{k-l}.
\end{align*}
\end{lemma}
\begin{proof}Utilizing \eqref{basic 1} we obtain
\begin{align*}
&|k-l|\Big(\|e^{c|kB|^{\f13}\tau}\f{w_l\varphi_{k-l}}{r^{\f12}}\|_{L^2L^2}+\|e^{c|kB|^{\f13}\tau}\f{w_l\varphi_{k-l}}{r^{\f32}}\|_{L^2L^2}\\
&+\|e^{c|kB|^{\f13}\tau}\f{w_l\varphi_{k-l}}{r^{\f52}}\|_{L^2L^2}+\|e^{c|kB|^{\f13}\tau}\f{w_l\varphi_{k-l}}{r^{\f72}}\|_{L^2L^2}\Big)\\
\leq&|k-l|\|e^{c|lB|^{\f13}\tau}\f{w_l}{r}\|_{L^2L^2}\Big(\|e^{c|(k-l)B|^{\f13}\tau}r^{\f12}\varphi_{k-l}\|_{L^\infty L^\infty}+\|e^{c|(k-l)B|^{\f13}\tau}\f{\varphi_{k-l}}{r^{\f12}}\|_{L^\infty L^\infty}\\
&+\|e^{c|(k-l)B|^{\f13}\tau}\f{\varphi_{k-l}}{r^{\f32}}\|_{L^\infty L^\infty}+\|e^{c|(k-l)B|^{\f13}\tau}\f{\varphi_{k-l}}{r^{\f52}}\|_{L^\infty L^\infty}\Big).
\end{align*}
Combining \eqref{5-11111} and the calculation Lemma \ref{Appendix A3} proved in appendix (with $\beta=-4, -2, 0, 2$), it follows
\begin{align*}
&|k-l|\Big(\|e^{c|(k-l)B|^{\f13}\tau}r^{\f12}\varphi_{k-l}\|_{L^\infty L^\infty}+\|e^{c|(k-l)B|^{\f13}\tau}\f{\varphi_{k-l}}{r^{\f12}}\|_{L^\infty L^\infty}\\
&+\|e^{c|(k-l)B|^{\f13}\tau}\f{\varphi_{k-l}}{r^{\f32}}\|_{L^\infty L^\infty}+\|e^{c|(k-l)B|^{\f13}\tau}\f{\varphi_{k-l}}{r^{\f52}}\|_{L^\infty L^\infty}\Big)\\
\lesssim&|k-l|^{-\f12}\Big(\|e^{c|(k-l)B|^{\f13}\tau}r^2e^{-\f{r^2}{8}}w_{k-l}\|_{L^\infty L^2}+\|e^{c|(k-l)B|^{\f13}\tau}re^{-\f{r^2}{8}}w_{k-l}\|_{L^\infty L^2}\\
&+\|e^{c|(k-l)B|^{\f13}\tau}e^{-\f{r^2}{8}}w_{k-l}\|_{L^\infty L^2}+\|e^{c|(k-l)B|^{\f13}\tau}e^{-\f{r^2}{8}}\f{w_{k-l}}{r}\|_{L^\infty L^2}\Big)\\
\lesssim&|k-l|^{-\f12}\|e^{c|(k-l)B|^{\f13}\tau}\f{w_{k-l}}{r}\|_{L^\infty L^2}.
\end{align*}
Therefore, we deduce
\begin{align*}
&\|e^{c|kB|^{\f13}\tau}\sum_{l\in\mathbb{Z}\backslash\{0,k\}}i(k-l)\f{w_l\varphi_{k-l}}{r^{\f12}}\|_{L^2L^2}+\|e^{c|kB|^{\f13}\tau}\sum_{l\in\mathbb{Z}\backslash\{0,k\}}i(k-l)\f{w_l\varphi_{k-l}}{r^{\f32}}\|_{L^2L^2}\\
&+\|e^{c|kB|^{\f13}\tau}\sum_{l\in\mathbb{Z}\backslash\{0,k\}}i(k-l)\f{w_l\varphi_{k-l}}{r^{\f52}}\|_{L^2L^2}+\|e^{c|kB|^{\f13}\tau}\sum_{l\in\mathbb{Z}\backslash\{0,k\}}i(k-l)\f{w_l\varphi_{k-l}}{r^{\f52}}\|_{L^2L^2}\\
\lesssim&\sum_{l\in\mathbb{Z}\backslash\{0,k\}}\|e^{c|lB|^{\f13}\tau}\f{w_l}{r}\|_{L^2L^2}|k-l|^{-\f12}\|e^{c|(k-l)B|^{\f13}\tau}\f{w_{k-l}}{r}\|_{L^\infty L^2}\\
\leq&\sum_{l\in\mathbb{Z}\backslash\{0,k\}}|lB|^{-\f13}E_l|k-l|^{-\f12}|(k-l)B|^{-\f16}E_{k-l}\\=&|B|^{-\f12}\sum_{l\in\mathbb{Z}\backslash\{0,k\}}|l|^{-\f13}|k-l|^{-\f23}E_lE_{k-l}.
\end{align*}
This completes the proof of the lemma.
\end{proof}

We also have

\begin{lemma}\label{Nonlinear estimate2}For any $k\in \mathbb{Z}\backslash \{0\}$, it holds
\begin{align*}
&|k|\Big(\|e^{c|kB|^{\f13}\tau}\sum_{l\in\mathbb{Z}\backslash\{0,k\}}\f{w_l\varphi_{k-l}'}{r^{\f52}}\|_{L^2L^2}+\|e^{c|kB|^{\f13}\tau}\sum_{l\in\mathbb{Z}\backslash\{0,k\}}\f{w_l\varphi_{k-l}'}{r^{\f32}}\|_{L^2L^2}\\
&+\|e^{c|kB|^{\f13}\tau}\sum_{l\in\mathbb{Z}\backslash\{0,k\}}\f{w_l\varphi_{k-l}}{r^{\f52}}\|_{L^2L^2}+\|e^{c|kB|^{\f13}\tau}\sum_{l\in\mathbb{Z}\backslash\{0,k\}}\f{w_l\varphi_{k-l}}{r^{\f72}}\|_{L^2L^2}\Big)\\
\lesssim&|B|^{-\f13}\sum_{l\in\mathbb{Z}\backslash\{0,k\}}\Big(|l|^{-\f23}|k-l|^{-\f16}E_lE_{k-l}+|l|^{-\f16}|k-l|^{-\f23}E_lE_{k-l}\Big).
\end{align*}
\end{lemma}
\begin{proof}Applying \eqref{basic 1} we have
\begin{align*}
&|k-l|\Big(\|e^{c|kB|^{\f13}\tau}\f{w_l\varphi'_{k-l}}{r^{\f52}}\|_{L^2L^2}+\|e^{c|kB|^{\f13}\tau}\f{w_l\varphi'_{k-l}}{r^{\f32}}\|_{L^2L^2}+\|e^{c|kB|^{\f13}\tau}\f{w_l\varphi_{k-l}}{r^{\f52}}\|_{L^2L^2}\\
&+\|e^{c|kB|^{\f13}\tau}\f{w_l\varphi_{k-l}}{r^{\f72}}\|_{L^2L^2}\Big)\\
\leq&|k-l|\|e^{c|lB|^{\f13}\tau}\f{w_l}{r^{\f32}}\|_{L^2L^{\infty}}\Big(\|e^{c|(k-l)B|^{\f13}\tau}\f{\varphi'_{k-l}}{r}\|_{L^\infty L^2}+\|e^{c|(k-l)B|^{\f13}\tau}\varphi'_{k-l}\|_{L^\infty L^2}\\
&+\|e^{c|(k-l)B|^{\f13}\tau}\f{\varphi_{k-l}}{r}\|_{L^\infty L^2}+\|e^{c|(k-l)B|^{\f13}\tau}\f{\varphi_{k-l}}{r^2}\|_{L^\infty L^2}\Big).
\end{align*}
Together with \eqref{5-11111}, the calculation Lemma \ref{Appendix A3} proved in appendix (with $\beta$ chosen to be $-2$ and $0$) this gives
\begin{align*}
&|k-l|\Big(\|e^{c|kB|^{\f13}\tau}\f{w_l\varphi'_{k-l}}{r^{\f52}}\|_{L^2L^2}+\|e^{c|kB|^{\f13}\tau}\f{w_l\varphi'_{k-l}}{r^{\f32}}\|_{L^2L^2}+\|e^{c|kB|^{\f13}\tau}\f{w_l\varphi_{k-l}}{r^{\f52}}\|_{L^2L^2}\\
&+\|e^{c|kB|^{\f13}\tau}\f{w_l\varphi_{k-l}}{r^{\f72}}\|_{L^2L^2}\Big)\\
\lesssim&\|e^{c|lB|^{\f13}\tau}\f{w_l}{r^{\f32}}\|_{L^2L^{\infty}}\Big(\|e^{c|(k-l)B|^{\f13}\tau}e^{-\f{r^2}{8}}w_{k-l}\|_{L^\infty L^2}+\|e^{c|(k-l)B|^{\f13}\tau}re^{-\f{r^2}{8}}w_{k-l}\|_{L^\infty L^2}\Big)\\
\lesssim&\|e^{c|lB|^{\f13}\tau}\f{w_l}{r^{\f32}}\|_{L^2L^{\infty}}\|e^{c|(k-l)B|^{\f13}\tau}\f{w_{k-l}}{r}\|_{L^\infty L^2}\\
\leq&|lB|^{-\f16}|l|^{-\f12}E_l|(k-l)B|^{-\f16}E_{k-l}=|B|^{-\f13}|l|^{-\f23}|(k-l)|^{-\f16}E_lE_{k-l}.
\end{align*}

Also we find
\begin{align*}
&|l|\Big(\|e^{c|kB|^{\f13}\tau}\f{w_l\varphi'_{k-l}}{r^{\f52}}\|_{L^2L^2}+\|e^{c|kB|^{\f13}\tau}\f{w_l\varphi'_{k-l}}{r^{\f32}}\|_{L^2L^2}+\|e^{c|kB|^{\f13}\tau}\f{w_l\varphi_{k-l}}{r^{\f52}}\|_{L^2L^2}\\
&+\|e^{c|kB|^{\f13}\tau}\f{w_l\varphi_{k-l}}{r^{\f72}}\|_{L^2L^2}\Big)\\
\leq&|l|\|e^{c|lB|^{\f13}\tau}\f{w_l}{r^2}\|_{L^2L^2}\Big(\|e^{c|(k-l)B|^{\f13}\tau}\f{\varphi'_{k-l}}{r^{\f12}}\|_{L^\infty L^\infty}+\|e^{c|(k-l)B|^{\f13}\tau}r^{\f12}\varphi'_{k-l}\|_{L^\infty L^\infty}\\
&+\|e^{c|(k-l)B|^{\f13}\tau}\f{\varphi_{k-l}}{r^{\f12}}\|_{L^\infty L^\infty}+\|e^{c|(k-l)B|^{\f13}\tau}\f{\varphi_{k-l}}{r^{\f32}}\|_{L^\infty L^\infty}\Big).
\end{align*}
Employing \eqref{5-11111} and calculation Lemma \ref{Appendix A3} proved in appendix (with $\beta=-2, 0$), we obtain
\begin{align*}
&|l|\Big(\|e^{c|kB|^{\f13}\tau}\f{w_l\varphi'_{k-l}}{r^{\f52}}\|_{L^2L^2}+\|e^{c|kB|^{\f13}\tau}\f{w_l\varphi'_{k-l}}{r^{\f32}}\|_{L^2L^2}+\|e^{c|kB|^{\f13}\tau}\f{w_l\varphi_{k-l}}{r^{\f52}}\|_{L^2L^2}\\
&+\|e^{c|kB|^{\f13}\tau}\f{w_l\varphi_{k-l}}{r^{\f72}}\|_{L^2L^2}\Big)\\
\lesssim&|l|\|e^{c|lB|^{\f13}\tau}\f{w_l}{r^2}\|_{L^2L^2}|k-l|^{-\f12}\Big(\|e^{c|(k-l)B|^{\f13}\tau}e^{-\f{r^2}{8}}w_{k-l}\|_{L^\infty L^2}+\|e^{c|(k-l)B|^{\f13}\tau}re^{-\f{r^2}{8}}w_{k-l}\|_{L^\infty L^2}\Big)\\
\lesssim&|l|\|e^{c|lB|^{\f13}\tau}\f{w_l}{r^2}\|_{L^2L^2}|k-l|^{-\f12}\|e^{c|(k-l)B|^{\f13}\tau}\f{w_{k-l}}{r}\|_{L^\infty L^2}\\
\leq&|lB|^{-\f16}E_l|k-l|^{-\f12}|(k-l)B|^{-\f16}E_{k-l}\\=&|B|^{-\f13}|l|^{-\f16}|k-l|^{-\f23}E_lE_{k-l}.
\end{align*}
Therefore, with applying the basic inequality $|k|\lesssim|k-l||l|$ for $l\neq0,k$, we conclude
\begin{align*}
&|k|\Big(\|e^{c|kB|^{\f13}\tau}\sum_{l\in\mathbb{Z}\backslash\{0,k\}}\f{w_l\varphi_{k-l}'}{r^{\f52}}\|_{L^2L^2}+\|e^{c|kB|^{\f13}\tau}\sum_{l\in\mathbb{Z}\backslash\{0,k\}}\f{w_l\varphi_{k-l}'}{r^{\f32}}\|_{L^2L^2}\\
&+\|e^{c|kB|^{\f13}\tau}\sum_{l\in\mathbb{Z}\backslash\{0,k\}}\f{w_l\varphi_{k-l}}{r^{\f52}}\|_{L^2L^2}+\|e^{c|kB|^{\f13}\tau}\sum_{l\in\mathbb{Z}\backslash\{0,k\}}\f{w_l\varphi_{k-l}}{r^{\f72}}\|_{L^2L^2}\Big)\\
\lesssim&|B|^{-\f13}\sum_{l\in\mathbb{Z}\backslash\{0,k\}}\Big(|l|^{-\f23}|k-l|^{-\f16}E_lE_{k-l}+|l|^{-\f16}|k-l|^{-\f23}E_lE_{k-l}\Big).
\end{align*}
This completes the proof of this lemma.
\end{proof}

Finally, we prove the following bound for terms involving $w_l \breve{\varphi}_{-l}$.
\begin{lemma}\label{Nonlinear estimate3}For nonlinear terms involving $w_l \breve{\varphi}_{-l}$ ($l\neq 0$), we have the bounds as below
\begin{align*}
 \sum_{l\in\mathbb{Z}/\{0\}}|l|\|w_l\breve{\varphi}_{-l}\|_{L^2L^2}+ \sum_{l\in\mathbb{Z}/\{0\}}|l|\|\f{w_l\breve{\varphi}_{-l}}{r^2}\|_{L^2L^2}\lesssim|B|^{-\f12}\sum_{l\in\mathbb{Z}/\{0\}}|l|^{-1}E_lE_{-l}.
\end{align*}

\end{lemma}
\begin{proof}Recall \eqref{5-11111} and for $\varphi_k, \breve{\varphi}_k$ we have $breve{\varphi}_k=\f{\varphi_k}{r^{\f12}}$, and
\begin{align*}
(\partial_{r}^2-\f{k^2-\f14}{r^2})\varphi_k=e^{-\f{r^2}{8}}w_k.
\end{align*}
Employing Hölder's inequalities, one has
\begin{align*}
&\sum_{l\in\mathbb{Z}/\{0\}}|l|\| w_l\breve{\varphi}_{-l}\|_{L^2L^2}+ \sum_{l\in\mathbb{Z}/\{0\}}|l|\| \f{w_l\breve{\varphi}_{-l}}{r^2}\|_{L^2L^2}\\
\lesssim&\sum_{l\in\mathbb{Z}/\{0\}}|l|\| \f{w_l}{r}\|_{L^2L^2}(\|\f{\varphi_{-l}}{r^{\f12}}\|_{L^{\infty}L^{\infty}}+\|\f{\varphi_{-l}}{r^{\f32}}\|_{L^{\infty}L^{\infty}}),
\end{align*}
Combining with \eqref{5-11111} and Lemma \ref{Appendix A3}, we derive
\begin{align*}
&|l|\Big(\|\f{\varphi_{-l}}{r^{\f12}}\|_{L^\infty L^\infty}+\|\f{\varphi_{-l}}{r^{\f32}}\|_{L^\infty L^\infty}\Big)\\
\lesssim&|l|^{-\f12}\Big(\|re^{-\f{r^2}{8}}w_{-l}\|_{L^\infty L^2}+\|e^{-\f{r^2}{8}}w_{-l}\|_{L^\infty L^2}\Big)\\\lesssim&|l|^{-\f12}\|\f{w_{-l}}{r}\|_{L^\infty L^2}.
\end{align*}
Hence we conclude
\begin{align*}
&\| \sum_{l\in\mathbb{Z}/\{0\}}ilw_l\breve{\varphi}_{-l}\|_{L^2L^2}+\| \sum_{l\in\mathbb{Z}/\{0\}}il\f{w_l\breve{\varphi}_{-l}}{r^2}\|_{L^2L^2}\\\lesssim&\sum_{l\in\mathbb{Z}/\{0\}}\| \f{w_l}{r}\|_{L^2L^2}|l|^{-\f12}\|\f{w_{-l}}{r}\|_{L^\infty L^2}\\
\lesssim&\sum_{l\in\mathbb{Z}/\{0\}}|l|^{-\f12}|lB|^{-\f13}E_l|lB|^{-\f16}E_{-l}\\=&|B|^{-\f12}\sum_{l\in\mathbb{Z}/\{0\}}|l|^{-1}E_lE_{-l}.
\end{align*}
\end{proof}

\subsection{Proof of Theorem \ref{Nonlinear stability}} With above estimates at hand, we are now ready to prove the main theorem of this paper.
\begin{proof}[Proof of Theorem \ref{Nonlinear stability}]
Recall that in Proposition \ref{nonzero nonlinear-1} and  Proposition \ref{nonzero nonlinear-2} we have derived the following space-time estimates
\begin{align*}
&\|e^{c|kB|^{\f13}\tau}w_k\|_{L^{\infty}L^2}+|kB|^{\f16}\|e^{c|kB|^{\f13}\tau}w_k\|_{L^2L^2}\\
&+\|e^{c|kB|^{\f13}\tau}\partial_{r}w_k\|_{L^2L^2}+\|e^{c|kB|^{\f13}\tau}(\f{|k|}{r}+r)w_k\|_{L^2L^2} \\
\lesssim&|kB|^{-\f16}\|e^{c|kB|^{\f13}\tau}f_1\|_{L^2L^2}+\|e^{c|kB|^{\f13}\tau}f_2\|_{L^2L^2}+\|w_k(0)\|_{L^2},
\end{align*}
and 
\begin{align*}
&|kB|^{\f16}\|e^{c|kB|^{\f13}\tau}w_k\|_{L^{\infty}X}+|kB|^{\f13}\|e^{c|kB|^{\f13}\tau}w_k\|_{L^2X}\\
&+|kB|^{\f16}\|e^{c|kB|^{\f13}\tau}\partial_{r}w_k\|_{L^2X}+|kB|^{\f16}\|e^{c|kB|^{\f13}\tau}(\f{|k|}{r}+r)w_k\|_{L^2X}\\
\lesssim&(\|e^{c|kB|^{\f13}\tau}f_1\|_{L^2X}+\|e^{c|kB|^{\f13}\tau}\f{f_2}{r}\|_{L^2X})+|kB|^{\f16}\|e^{c|kB|^{\f13}\tau}f_2\|_{L^2X}+|kB|^{\f16}\|w_k(0)\|_{X},
\end{align*}
where $X=L^2(\f{1}{r^2})$ is the weighted $L^2$ space we construct, and $f_1, f_2$ are nonlinear terms given by 
\begin{align*}
&f_1=ik\sum_{l\in\mathbb{Z}}w_{l}\f{\partial_r(r^{-\f12}\varphi_{k-l})}{r}+\sum_{l\in\mathbb{Z}}i(k-l)(\f14-\f{1}{2r^2})\f{w_l\varphi_{k-l}}{r^{\f12}},\quad f_2=\sum_{l\in\mathbb{Z}}i(k-l)\f{w_l\varphi_{k-l}}{r^{\f32}}.
\end{align*}
Using the calculation Lemma \ref{Appendix A1} proved in appendix (with $\alpha=\f{3}{2}$), we obtain
\begin{align*}
&|k|\|\f{w_k}{r^{\f32}}\|_{L^{\infty}}^2\lesssim|k|\|\f{w_k'}{r}\|_{L^2}\|\f{w_k}{r^2}\|_{L^2}+|k|\|\f{w_k}{r^2}\|_{L^2}^2\lesssim\|\f{w_k'}{r}\|_{L^2}^2+k^2\|\f{w_k}{r^2}\|_{L^2}^2,\\
&\|\f{w_0}{r^{\f32}}\|_{L^{\infty}}^2\lesssim\|\f{w_0'}{r}\|_{L^2}\|\f{w_0}{r^2}\|_{L^2}+\|\f{w_0}{r^2}\|_{L^2}^2\lesssim\|\f{w_0'}{r}\|_{L^2}^2+\|\f{w_0}{r^2}\|_{L^2}^2.
\end{align*}
This implies
\begin{align*}
&|k|\|e^{c|kB|^{\f13}\tau}\f{w_k}{r^{\f32}}\|_{L^2L^{\infty}}^2\lesssim\|e^{c|kB|^{\f13}\tau}\f{w_k'}{r}\|_{L^2L^2}^2+k^2\|e^{c|kB|^{\f13}\tau}\f{w_k}{r^2}\|_{L^2L^2}^2,\\
&\|\f{w_0}{r^{\f32}}\|_{L^2L^{\infty}}^2\lesssim\|\f{w_0'}{r}\|_{L^2L^2}^2+\|\f{w_0}{r^2}\|_{L^2L^2}^2.
\end{align*}
Therefore, following Proposition \ref{nonzero nonlinear-1}, Proposition \ref{nonzero nonlinear-2}, and Proposition \ref{zero nonlinear}, we deduce
\begin{align*}
E_k=&\|e^{c|kB|^{\f13}\tau}w_k\|_{L^{\infty}L^2}+|kB|^{\f16}\Big(\|e^{c|kB|^{\f13}\tau}w_k\|_{L^2L^2}+\|e^{c|kB|^{\f13}\tau}\f{w_k}{r}\|_{L^{\infty}L^2}\\
&+|k|\|e^{c|kB|^{\f13}\tau}\f{w_k}{r^2}\|_{L^2L^2}+|k|^{\f12}\|e^{c|kB|^{\f13}\tau}\f{w_k}{r^{\f32}}\|_{L^2L^{\infty}}\Big)
+|kB|^{\f13}\|e^{c|kB|^{\f13}\tau}\f{w_k}{r}\|_{L^2L^2}\\
\lesssim&|kB|^{-\f16}\|e^{c|kB|^{\f13}\tau}f_1\|_{L^2L^2}+\|e^{c|kB|^{\f13}\tau}f_2\|_{L^2L^2}+\|e^{c|kB|^{\f13}\tau}f_1\|_{L^2X}+\|e^{c|kB|^{\f13}\tau}\f{f_2}{r}\|_{L^2X}\\
&+|kB|^{\f16}\|e^{c|kB|^{\f13}\tau}f_2\|_{L^2X}+\|w_k(0)\|_{L^2}+|kB|^{\f16}\|w_k(0)\|_{X},\quad |k|\geq1,\\
E_0=&|B|^{\f16}(\|\f{w_0}{r}\|_{L^{\infty}L^2}+\|\f{w_0}{r^{\f32}}\|_{L^2L^{\infty}}+\|\f{w_0}{r^2}\|_{L^2L^2}+\|w_0\|_{L^2L^2})\\
\lesssim&|B|^{\f16}(\| \sum_{l\in\mathbb{Z}/\{0\}}ilw_l\breve{\varphi}_{-l}\|_{L^2L^2}+\| \sum_{l\in\mathbb{Z}/\{0\}}il\f{w_l\breve{\varphi}_{-l}}{r^2}\|_{L^2L^2}+\|\f{w_0(0)}{r}\|_{L^2}),\quad k=0.
\end{align*}
To bound the terms on the right hand side, we utilize the following inequality
\begin{align}\label{basic-4.1-1}
|k|\lesssim|k-l||l| \quad \mbox{if} \ k\neq0,l.
\end{align}
With \eqref{basic-4.1-1}, Lemma \ref{zero-Nonlinear estimate1}, Lemma \ref{zero-Nonlinear estimate2}, Lemma \ref{zero-Nonlinear estimate3}, Lemma \ref{Nonlinear estimate1} and Lemma \ref{Nonlinear estimate2}, we have that the first term $|kB|^{-\f16}\|e^{c|kB|^{\f13}\tau}f_1\|_{L^2L^2}$ can be controlled via
\begin{align*}
&|kB|^{-\f16}\|e^{c|kB|^{\f13}\tau}f_1\|_{L^2L^2}\\
\lesssim&|kB|^{-\f16}\Big(|k|\|e^{c|kB|^{\f13}\tau}\sum_{l\in\mathbb{Z}\backslash\{0,k\}}\f{w_{l}\partial_r\varphi_{k-l}}{r^{\f32}}\|_{L^2L^2}+|k|\|e^{c|kB|^{\f13}\tau}\sum_{l\in\mathbb{Z}\backslash\{0,k\}}\f{w_{l}\varphi_{k-l}}{r^{\f52}}\|_{L^2L^2}\\
&+|k|\|e^{c|kB|^{\f13}\tau}\f{w_{0}\partial_r\varphi_{k}}{r^{\f32}}\|_{L^2L^2}+|k|\|e^{c|kB|^{\f13}\tau}\f{w_{0}\varphi_{k}}{r^{\f52}}\|_{L^2L^2}+|k|\|e^{c|kB|^{\f13}\tau}\f{w_{k}\partial_r(r^{-\f12}\varphi_{0})}{r}\|_{L^2L^2}\\
&+\|e^{c|kB|^{\f13}\tau}\sum_{l\in\mathbb{Z}\backslash\{0,k\}}i(k-l)\f{w_l\varphi_{k-l}}{r^{\f12}}\|_{L^2L^2}+\|e^{c|kB|^{\f13}\tau}\sum_{l\in\mathbb{Z}\backslash\{0,k\}}i(k-l)\f{w_l\varphi_{k-l}}{r^{\f52}}\|_{L^2L^2}\\
&+\|e^{c|kB|^{\f13}\tau}ik\f{w_0\varphi_{k}}{r^{\f12}}\|_{L^2L^2}+\|e^{c|kB|^{\f13}\tau}ik\f{w_0\varphi_{k}}{r^{\f52}}\|_{L^2L^2}\Big)\\
\lesssim&|kB|^{-\f16}(|B|^{-\f13}|k|^{-\f16}\sum_{l\in\mathbb{Z}\backslash\{0,k\}}E_lE_{k-l}+|B|^{-\f13}|k|^{-\f16}E_0E_{k})\\
\lesssim&|B|^{-\f12}|k|^{-\f13}\sum_{l\in\mathbb{Z}}E_lE_{k-l}.
\end{align*}
For the estimate of $f_2$, we appeal to \eqref{basic-4.1-1}, Lemma \ref{zero-Nonlinear estimate1} and Lemma \ref{Nonlinear estimate1} again. And we obtain
\begin{align*}
&\|e^{c|kB|^{\f13}\tau}f_2\|_{L^2L^2}+\|e^{c|kB|^{\f13}\tau}\f{f_2}{r}\|_{L^2X}+|kB|^{\f16}\|e^{c|kB|^{\f13}\tau}f_2\|_{L^2X}\\
\lesssim&|kB|^{\f16}(|B|^{-\f12}\sum_{l\in\mathbb{Z}\backslash\{0,k\}}|l|^{-\f13}|k-l|^{-\f23}E_lE_{k-l}+|B|^{-\f12}|k|^{-\f56}E_0E_{k})\\
\lesssim&|B|^{-\f13}|k|^{-\f16}\sum_{l\in\mathbb{Z}}E_lE_{k-l}.
\end{align*}
It remains to handle the item $\|e^{c|kB|^{\f13}\tau}f_1\|_{L^2X}$. Applying \eqref{basic-4.1-1}, Lemma \ref{zero-Nonlinear estimate1}, Lemma \ref{zero-Nonlinear estimate2}, Lemma \ref{zero-Nonlinear estimate3}, Lemma \ref{Nonlinear estimate1} and Lemma \ref{Nonlinear estimate2}, we obtain
\begin{align*}
&\|e^{c|kB|^{\f13}\tau}f_1\|_{L^2X}\\
\lesssim&\Big(\|e^{c|kB|^{\f13}\tau}\sum_{l\in\mathbb{Z}\backslash\{0,k\}}ik\f{w_l\varphi_{k-l}'}{r^{\f52}}\|_{L^2L^2}+\|e^{c|kB|^{\f13}\tau}\sum_{l\in\mathbb{Z}\backslash\{0,k\}}ik\f{w_l\varphi_{k-l}}{r^{\f72}}\|_{L^2L^2}\\
&+|k|\|e^{c|kB|^{\f13}\tau}\f{w_{0}\partial_r\varphi_{k}}{r^{\f52}}\|_{L^2L^2}+|k|\|e^{c|kB|^{\f13}\tau}\f{w_{0}\varphi_{k}}{r^{\f72}}\|_{L^2L^2}+|k|\|e^{c|kB|^{\f13}\tau}\f{w_{k}\partial_r(r^{-\f12}\varphi_{0})}{r^2}\|_{L^2L^2}\\
&+\|e^{c|kB|^{\f13}\tau}\sum_{l\in\mathbb{Z}\backslash\{0,k\}}i(k-l)\f{w_l\varphi_{k-l}}{r^{\f32}}\|_{L^2L^2}+\|e^{c|kB|^{\f13}\tau}\sum_{l\in\mathbb{Z}\backslash\{0,k\}}i(k-l)\f{w_l\varphi_{k-l}}{r^{\f72}}\|_{L^2L^2}\\
&+\|e^{c|kB|^{\f13}\tau}ik\f{w_0\varphi_{k}}{r^{\f12}}\|_{L^2L^2}+\|e^{c|kB|^{\f13}\tau}ik\f{w_0\varphi_{k}}{r^{\f52}}\|_{L^2L^2}\Big)\\
\lesssim&|B|^{-\f13}|k|^{-\f16}(\sum_{l\in\mathbb{Z}\backslash\{0,k\}}E_lE_{k-l}+E_0E_{k})\lesssim|B|^{-\f13}|k|^{-\f16}\sum_{l\in\mathbb{Z}}E_lE_{k-l}.
\end{align*}
Finally, we arrive at
\begin{align*}
E_k\lesssim&|kB|^{-\f16}\|e^{c|kB|^{\f13}\tau}f_1\|_{L^2L^2}+\|e^{c|kB|^{\f13}\tau}f_2\|_{L^2L^2}+\|e^{c|kB|^{\f13}\tau}f_1\|_{L^2X}+\|e^{c|kB|^{\f13}\tau}\f{f_2}{r}\|_{L^2X}\\
&+|kB|^{\f16}\|e^{c|kB|^{\f13}\tau}f_2\|_{L^2X}+\|w_k(0)\|_{L^2}+|kB|^{\f16}\|w_k(0)\|_{X}\\
\lesssim &|k|^{-\f16}|B|^{-\f13}\sum_{l\in\mathbb{Z}}E_lE_{k-l}+\|w_k(0)\|_{L^2}+|kB|^{\f16}\|w_k(0)\|_{X},
\end{align*}
In addition, using Lemma \ref{Nonlinear estimate3}, we also find
\begin{align*}
E_0\lesssim&|B|^{-\f13}\sum_{l\in\mathbb{Z}/\{0\}}E_lE_{-l}+|B|^{\f16}\|\f{w_0(0)}{r}\|_{L^2}.
\end{align*}
With the bootstrap assumption
\begin{align*}
\sum_{k\in\mathbb{Z}}E_k\leq c_0 B^{\f13},
\end{align*}
we obtain
\begin{align*}
&\sum_{k\in\mathbb{Z}}E_k\\
\leq & C\Big[|B|^{-\f13}(\sum_{k\in\mathbb{Z}}E_k)^2+\sum_{k\in\mathbb{Z}/\{0\}}\|w_k(0)\|_{L^2}+\sum_{k\in\mathbb{Z}/\{0\}}|kB|^{\f16}\|w_k(0)\|_{X}+|B|^{\f16}\|\f{w_0(0)}{r}\|_{L^2}\Big]\\
\leq & C\Big[c_0\sum_{k\in\mathbb{Z}}E_k+\sum_{k\in\mathbb{Z}/\{0\}}\|w_k(0)\|_{L^2}+\sum_{k\in\mathbb{Z}/\{0\}}|kB|^{\f16}\|\f{w_k(0)}{r}\|_{L^2}+|B|^{\f16}\|\f{w_0(0)}{r}\|_{L^2}\Big].
\end{align*}
Choose $c_0\leq\f12 C^{-1}$, we hence prove
\begin{align*}
&\sum_{k\in\mathbb{Z}}E_k\lesssim \sum_{k\in\mathbb{Z}/\{0\}}\|w_k(0)\|_{L^2}+\sum_{k\in\mathbb{Z}/\{0\}}|kB|^{\f16}\|\f{w_k(0)}{r}\|_{L^2}+|B|^{\f16}\|\f{w_0(0)}{r}\|_{L^2}.
\end{align*}
The above completes the proof of Theorem \ref{Nonlinear stability}.
\end{proof}

\subsection{Proof of Theorem \ref{improved transition threshold}}

As a corollary of the above main result, we also obtain
an improved transition threshold for the original equations \eqref{full nonlinear equation ns vor}. Let $\tilde{\omega}:=\f{w-2A_1t}{t}$. We have
\begin{align*}
\int_{\mathbb{R}^2}|\tilde{\omega}(t,x)|dx=&\nu t\int_{\mathbb{R}^2}|\f{w-2A_1t}{t}|d\xi=\nu\int_{\mathbb{R}^2}|w(\tau,\xi)-2A_1e^{\tau}|d\xi\\
=&\nu\int_0^{2\pi}\int_0^{\infty}r|w(\tau,r,\theta)-2A_1e^{\tau}|drd\theta.
\end{align*}
Note that $w(\tau,\xi)=w(\tau,r,\theta)=\sum\limits_{k\in\mathbb{Z}} \omega_k(\tau,r)e^{ik\theta}$. And recall the following estimate for $\tau\geq0$ in Theorem \ref{Main result-physical space}
\begin{align*}
 \mathcal{E}(\tau)=|\f{A_2}{\nu}|^{\f16}\|\f{\omega_0(\tau)-2A_1e^{\tau}}{r}\|_{\mathcal{M}}+\sum_{k\in\mathbb{Z}/\{0\}}(\|\omega_k(\tau)\|_{\mathcal{M}}+|k|^{\f16}|\f{A_2}{\nu}|^{\f16}\|\f{\omega_k(\tau)}{r}\|_{\mathcal{M}})\leq Cc_0|\f{A_2}{\nu}|^\f13.
  \end{align*}
  We thus obtain 
\begin{align*}
\int_{\mathbb{R}^2}|\tilde{\omega}(t,x)|dx=&\nu\int_0^{2\pi}\int_0^{\infty}r|w(\tau,r,\theta)-2A_1e^{\tau}|drd\theta\\
\leq&2\pi\nu\int_0^{\infty}r(|\omega_0(\tau)-2A_1e^{\tau}|+\sum_{k\in\mathbb{Z}/\{0\}}|\omega_k|)dr\\
\leq&2\pi\nu(\int_0^{\infty}r^3e^{-\f{r^2}{4}}dr)^{\f12}(\int_0^{\infty}r^{-1}e^{\f{r^2}{4}}|\omega_0(\tau)-2A_1e^{\tau}|^2dr)^{\f12}\\
&+2\pi\nu\sum_{k\in\mathbb{Z}/\{0\}}(\int_0^{\infty}re^{-\f{r^2}{4}}dr)^{\f12}(\int_0^{\infty}re^{\f{r^2}{4}}|\omega_k(\tau)|^2dr)^{\f12}\\
=&2\pi\nu(\int_0^{\infty}r^3e^{-\f{r^2}{4}}dr)^{\f12}\|\f{\omega_0(\tau)-2A_1e^{\tau}}{r}\|_{\mathcal{M}}\\
&+2\pi\nu(\int_0^{\infty}re^{-\f{r^2}{4}}dr)^{\f12}\sum_{k\in\mathbb{Z}/\{0\}}\|\omega_k(\tau)\|_{\mathcal{M}}\\
\leq&C\nu \mathcal{E}(\tau)\leq Cc_0|\f{A_2}{\nu}|^{\f13}\nu=Cc_0|A_2|^{\f13}\nu^{\f23}.
\end{align*}
Hence, for initial data at $t=1$, if the initial perturbation satisfies
\begin{align*}
&\int_{\mathbb{R}^2}|\tilde{\omega}(1,x)|dx\lesssim \nu \mathcal{E}(0)\leq Cc_0 |A_2|^{\f13}\nu^{\f23},
\end{align*}
then we proof
\begin{align*}
&\int_{\mathbb{R}^2}|\tilde{\omega}(t,x)|dx\lesssim \nu \mathcal{E}(\tau)\lesssim \nu \mathcal{E}(0)\leq Cc_0 |A_2|^{\f13}\nu^{\f23},\quad \ for\ all\ t\geq1.
\end{align*}
This is corresponding to the improved transition threshold stated in the Theorem \ref{improved transition threshold}.

\appendix
\section{\textbf{Calculation lemmas}}
\begin{lemma}\label{Appendix real function g}For any $A, B\in \mathbb{R}$, the below ODE admits a global solution (with variable $r\in (0, \infty)$)
		\begin{align}\label{A.1}
		\f{g''}{g}+\f{A}{r}\f{g'}{g}=\f{B}{r^2}.
		\end{align}
	\end{lemma}
	\begin{proof}
		Let $G:=\f{g'}{g}$. By noting that
		\begin{align*}
		G'=\f{g''g-g'^2}{g^2}=\f{g''}{g}-G^2,
		\end{align*}
		 we can transfer the above equation \eqref{A.1} to
		\begin{align*}
		G'+G^2+\f{A}{r}G=\f{B}{r^2}.
		\end{align*}
		Denote $F=G+\frac{A}{2r}$, we deduce
		$F'+F^2=\f{C}{r^2}$
		with $C=\f{A^2}{4}+\f{A}{2}+B$.
		Dividing both sides by $F^2$ and via introducing $H=\frac{1}{F}$, one obtains
		\begin{equation*}
		-H'+1=C(\f{H}{r})^2,
		\end{equation*}
		Let $K=\frac{H}{r}$, we then have
		\begin{equation*}
		-(rK'+K)+1=CK^2,
		\end{equation*}
		or equivalently
		\begin{equation*}
		K'=\f{-CK^2-K+1}{r}.
		\end{equation*}
		The above separable equation can be solved via
		\begin{equation*}\label{eqn for K}
		\int \f{dK}{-CK^2-K+1}=\log r.
		\end{equation*}
		Therefore, we deduce
		\begin{equation*}
		g(r, A, B)=C_0 \exp(\int_{0}^{r}G(s)ds)=C_0 \exp(\int_{0}^{r}\f{1}{sK(s)}-\f{A}{2s}ds),
		\end{equation*}
\end{proof}
\begin{lemma}\label{Appendix A1-1}For any $w\in L^2(\Gamma), w|_{r\in\partial\Gamma}=0$, it holds
\begin{align*}
&\|w\|_{L^{\infty}}^2\leq 2\|w\|_{L^2}\|w'\|_{L^2}.
\end{align*}
\end{lemma}
\begin{proof}
Since $w|_{r\in\partial\Gamma}=0$, we have
\begin{align*} |w(r)|^2=&\int_0^r\partial_s(|w(s)|^2)ds=\int_0^r(w'(s)\overline{w}(s)+w(s)\overline{w}'(s))ds\leq 2\|w\|_{L^2}\|w'\|_{L^2},
\end{align*}
\end{proof}

\begin{lemma}\label{Appendix A1}For any $w\in L^2(\Gamma)$ and $\alpha\in \mathbb{R}$, it holds
\begin{align*}
\|\f{w}{r^{\alpha}}\|_{L^{\infty}}^2\lesssim\|\f{w'}{r^{\alpha-\f12}}\|_{L^2}\|\f{w}{r^{\alpha+\f12}}\|_{L^2}+\|\f{w}{r^{\alpha+\f12}}\|_{L^2}^2.
\end{align*}
\end{lemma}
\begin{proof}
Choose $r_1\in(0,1)$ such that $|\f{|w(r_1)|^2}{r_1^{2\alpha}}|\leq\|\f{w}{r^{\alpha}}\|_{L^2(0,1)}^2\leq\|\f{w}{r^{\alpha+\f12}}\|_{L^2(0,1)}^2$. Then we have
\begin{align*} \f{|w(r)|^2}{r^{2\alpha}}\leq&\int_{r_1}^r\partial_s(\f{|w(s)|^2}{s^{\alpha}})ds+\|\f{w}{r^{\alpha+\f12}}\|_{L^2(0,1)}^2\\
=&\int_{r_1}^r\f{w'(s)\overline{w}(s)+w(s)\overline{w}'(s)}{s^{2\alpha}}-2\alpha\f{|w(s)|^2}{s^{2\alpha+1}}ds+\|\f{w}{r^{\alpha+\f12}}\|_{L^2(0,1)}^2\\
\lesssim& \|\f{w'}{r^{\alpha-\f12}}\|_{L^2}\|\f{w}{r^{\alpha+\f12}}\|_{L^2}+\|\f{w}{r^{\alpha+\f12}}\|_{L^2}^2.
\end{align*}
This completes the proof of this lemma.
\end{proof}

\begin{lemma}\label{Appendix A2}Let $w=\varphi''-\f{k^2-\f14}{r^2}\varphi\in L^2(\Gamma)$ with $w|_{r\in\partial\Gamma}=0, \varphi|_{r\in\partial\Gamma}=0$. Then for any $|k|\geq 1$ and $\beta\in\mathbb{R}$ it holds
		\begin{align*}
			\Re \langle -w, r^\beta \varphi\rangle \gtrsim  \|r^{\f{\beta}{2}}\varphi'\|_{L^2}^2+k^2\|r^{\f{\beta}{2}-1}\varphi\|_{L^2}^2.
		\end{align*}
	\end{lemma}
\begin{proof}
	For any $C^2$ real function $g$ with variable $r$, we have the following identity:
	\begin{align*}
		-\varphi''=-g^{-1}\partial_r(r^{\beta}g^2\partial_r(r^{-\beta}g^{-1}\varphi))+\f{\beta}{r}\varphi'-g^{-1}[g''-(\f{\beta}{r}g)']\varphi.
	\end{align*}
 Via integration by parts we obtain
	\begin{align*}
		\langle-w,r^\beta \varphi \rangle=&\langle-g^{-1}\partial_r(r^\beta g^2\partial_r(r^{-\beta}g^{-1}\varphi))+\f{\beta}{r}\varphi'-g^{-1}[g''-(\f{\beta}{r}g)']\varphi+\f{k^2-\f14}{r^2}\varphi,\f{\varphi}{r^\beta}\rangle\\
		=&\|r^{-\f{\beta}{2}}g\partial_r(r^{\beta}g^{-1}\varphi)\|_{L^2}^2+\beta\langle \varphi', r^{\beta-1}\varphi\rangle-\langle g^{-1}[g''-(\f{\beta}{r}g)']\varphi, r^{\beta}\varphi\rangle\\&+(k^2-\f14)\|r^{\f{\beta}{2}-1}\varphi\|_{L^2}^2.
	\end{align*}
	This together with
	\begin{align}
		\label{lemma-A4-1} \Re\langle \varphi', r^{\beta-1}\varphi\rangle=-\f12\int_{\Gamma}r^{\beta-1}d|\varphi|^2=-\f{\beta-1}{2}\int_{\Gamma}r^{\beta-2}|\varphi|^2 dr,
	\end{align}
	gives
	\begin{align*}
		\Re\langle-w, r^\beta \varphi \rangle=&\|r^{-\f{\beta}{2}}g\partial_r(r^{\beta}g^{-1}\varphi)\|_{L^2}^2-\langle g^{-1}(g''-\f{\beta}{r}g')\varphi, r^\beta \varphi\rangle\\&+(k^2-\f14-\f{\beta(\beta+1)}{2})\|r^{\f{\beta}{2}-1}\varphi \|_{L^2}^2.
	\end{align*}
	Choosing $g=g(r, A, B)$ as in Lemma \ref{Appendix real function g} with $A=-\beta, B=\f14+\f{\beta(\beta+1)}{2}$, then $\f{g''}{g}-\f{\beta}{r}\f{g'}{g}=\f{\f14+\f{\beta(\beta+1)}{2}}{r^2}$. Hence we deduce
	\begin{align}
		\label{lemma-A4-2}\Re\langle-w,r^\beta \varphi\rangle=\|r^{-\f{\beta}{2}}g\partial_r(r^{\beta}g^{-1}\varphi)\|_{L^2}^2+k^2 \|r^{\f{\beta}{2}-1}\varphi \|_{L^2}^2\geq k^2 \|r^{\f{\beta}{2}-1}\varphi \|_{L^2}^2.
	\end{align}
	Employing integration by parts and (\ref{lemma-A4-1}), we also have
	\begin{align}
	\label{lemma-A4-7}	\Re\langle-w, r^\beta \varphi\rangle=&\langle-\varphi''+\f{k^2-\f14}{r^2}\varphi,r^\beta \varphi\rangle\ \nonumber\\=&\|r^{\f{\beta}{2}}\varphi'\|_{L^2}^2+\beta\Re\langle \varphi',r^\beta \varphi \rangle+(k^2-\f14)\|r^{\f{\beta}{2}-1}\varphi\|_{L^2}^2 \\
\nonumber	=&\|r^{\f{\beta}{2}}\varphi'\|_{L^2}^2+(k^2-\f14-\f{\beta(\beta-1)}{2})\|r^{\f{\beta}{2}-1}\varphi\|_{L^2}^2 .
	\end{align}
	Combining  \eqref{lemma-A4-7} and \eqref{lemma-A4-2}, we now conclude
	\begin{align}
		\label{lemma-A4-01}\Re\langle-w,r^\beta \varphi \rangle\gtrsim\|r^{\f{\beta}{2}}\varphi'\|_{L^2}^2+k^2\|r^{\f{\beta}{2}-1}\varphi\|_{L^2}^2.
	\end{align}
\end{proof}

\begin{lemma}\label{Appendix A3}
		Under the same conditions of Lemma \ref{Appendix A2}, we have
		\begin{align*}
		&\|r^{\f{\beta}{2}+1}\varphi''\|_{L^2}+|k|^{\f12}\|r^{\f{\beta+1}{2}}\varphi'\|_{L^{\infty}}+|k|\|r^{\f{\beta}{2}}\varphi'\|_{L^2}+|k|^{\f32}\|r^{\f{\beta-1}{2}}\varphi\|_{L^{\infty}}+k^2\|r^{\f{\beta}{2}-1}\varphi\|_{L^2}\lesssim\|r^{\f{\beta}{2}+1}w\|_{L^2}.
		\end{align*}
	\end{lemma}
\begin{proof}
	Recall in Lemma \ref{Appendix A2}, we have the following inequality
		\begin{align*}
		\Re\langle-w,r^\beta \varphi \rangle\gtrsim\|r^{\f{\beta}{2}}\varphi'\|_{L^2}^2+k^2\|r^{\f{\beta}{2}-1}\varphi\|_{L^2}^2.
	\end{align*}
	Combining with
	\begin{align*}
	\Re\langle-w,r^\beta \varphi \rangle \leq \|r^{\f{\beta}{2}+1}w\|_{L^2}\|r^{\f{\beta}{2}-1}\varphi\|_{L^2} \leq |k|^{-1}\|r^{\f{\beta}{2}+1}w\|_{L^2}( \|r^{\f{\beta}{2}}\varphi'\|_{L^2}+|k| \|r^{\f{\beta}{2}-1}\varphi\|_{L^2})
	\end{align*}
	and
	\begin{align*}
	\|r^{\f{\beta}{2}}\varphi'\|_{L^2}^2+k^2\|r^{\f{\beta}{2}-1}\varphi\|_{L^2}^2\gtrsim( \|r^{\f{\beta}{2}}\varphi'\|_{L^2}+|k| \|r^{\f{\beta}{2}-1}\varphi\|_{L^2})^2,
	\end{align*}
	we obtain
	\begin{align}\label{lemma A2-01}
	|k| \|r^{\f{\beta}{2}}\varphi'\|_{L^2}+k^2 \|r^{\f{\beta}{2}-1}\varphi\|_{L^2}\lesssim \|r^{\f{\beta}{2}+1}w\|_{L^2}.
	\end{align}
	Employing Lemma \ref{Appendix A1} (with $\alpha=-\f{\beta-1}{2}$) and \eqref{lemma A2-01} we get
	\begin{align*}
	\|r^{\f{\beta-1}{2}}\varphi\|_{L^{\infty}}^2\lesssim \|r^{\f{\beta}{2}}\varphi'\|_{L^2}\|r^{\f{\beta}{2}-1}\varphi\|_{L^2}+\|r^{\f{\beta}{2}-1}\varphi\|_{L^2}^2\lesssim |k|^{-3} \|r^{\f{\beta}{2}+1}w\|_{L^2}^2.
	\end{align*}
	To derive the bound for $\|r^{\f{\beta}{2}+1}\varphi''\|_{L^2}$, we apply the following equality
	\begin{align*}
	\Re\langle w, r^{\beta+2} \varphi'' \rangle=\|r^{\f{\beta}{2}+1}\varphi''\|^2_{L^2}-(k^2-\f14)\Re\langle \varphi, r^{\beta+2} \varphi''\rangle.
	\end{align*}
	By  Hölder's inequality we have
	\begin{align*}
	\|r^{\f{\beta}{2}+1}\varphi''\|_{L^2}^2\leq \|r^{\f{\beta}{2}+1}\varphi''\|_{L^2}\|r^{\f{\beta}{2}+1}w\|_{L^2}+k^2 \|r^{\f{\beta}{2}+1}\varphi''\|_{L^2}\|r^{\f{\beta}{2}-1}\varphi\|_{L^2},
	\end{align*}
	which yields
	\begin{align}\label{lemma A2-02}
		\|r^{\f{\beta}{2}+1}\varphi''\|_{L^2}\leq \|r^{\f{\beta}{2}+1}w\|_{L^2}+k^2 \|r^{\f{\beta}{2}-1}\varphi\|_{L^2} \lesssim \|r^{\f{\beta}{2}+1}w\|_{L^2}.
	\end{align}
	Here in the second inequality we appeal to \eqref{lemma A2-01}.
	
\noindent	Finally, we utilize Lemma \ref{Appendix A1} again with $\alpha=-\f{\beta+1}{2}$. Together with \eqref{lemma A2-01} and \eqref{lemma A2-02}, we deduce
	\begin{align*}
	\|r^{\f{\beta+1}{2}}\varphi'\|_{L^{\infty}}^2\lesssim \|r^{\f{\beta}{2}+1}\varphi''\|_{L^2}\|r^{\f{\beta}{2}}\varphi'\|_{L^2}+\|r^{\f{\beta}{2}}\varphi'\|_{L^2}^2\lesssim |k|^{-2} \|r^{\f{\beta}{2}+1}w\|_{L^2}^2.
	\end{align*}
\end{proof}

\end{CJK*}

\bigskip



\begin{thebibliography}{99}


\bibitem{BW}M. Beck and C. E.Wayne, {\it Metastability and rapid convergence to quasi-stationary bar states for the two-dimensional Navier-Stokes equations}, Proc. Roy. Soc. Edinburgh Sect. A 143 (2013), no. 5, 905-927.

\bibitem{BGM3} J. Bedrossian, P. Germain and N. Masmoudi, {\it On the stability threshold for the 3D Couette flow in Sobolev regularity}, Annals of Math., 185(2017), 541-608.


\bibitem{BGM-bams}J. Bedrossian, P. Germain and N. Masmoudi, {\it Stability of the Couette flow at high Reynolds number in
2D and 3D}, Bull. Amer. Math. Soc. (N.S.), 56 (2019), 373-414.

\bibitem{BM}J. Bedrossian and N. Masmoudi, {\it Inviscid damping and the asymptotic stability of planar shear flows in
the 2D Euler equations}, Publ. Math. Inst. Hautes  $\acute{E}$tudes Sci., 122(2015), 195-300.


\bibitem{BMV} J. Bedrossian, N. Masmoudi and V. Vicol,  {\it Enhanced dissipation and inviscid damping in the inviscid limit of the Navier-Stokes equations near the two dimensional Couette flow},  Arch. Ration. Mech. Anal., 219(2016), 1087-1159.

\bibitem{BWV}J. Bedrossian, F. Wang and V. Vicol, {\it The Sobolev stability threshold for 2D shear flows near Couette}, J. Nonlinear Sci.,  28 (2018),  2051-2075.

\bibitem{BZ}J. Bedrossian and M. Zelati, {\it Enhanced dissipation, hypoellipticity, and anomalous small noise inviscid limits in shear flows,} Arch. Ration. Mech. Anal. 224 (2017), no. 3, 1161-1204.

\bibitem{Ch}S. J. Chapman, {\it  Subcritical transition in channel flows}, J. Fluid Mech., 451(2002), 35-97.

\bibitem{CLWZ-2D-C}Q. Chen, T. Li, D. Wei and Z. Zhang, {\it  Transition threshold for the 2-D Couette flow in a finite channel}, Arch. Ration. Mech. Anal., 238 (2020), 125-183.

\bibitem{CWZ-2D-ided} Q. Chen, D. Wei and Z. Zhang, {\it Linear inviscid damping and enhanced dissipation for monotone shear flows}, preprint.

\bibitem{CWZ-2D-pp} Q. Chen, D. Wei and Z. Zhang, {\it Linear stability of pipe Poiseuille flow at high Reynolds number}, arXiv:1910.14245.

\bibitem{CWZ-3D-C} Q. Chen, D. Wei and Z. Zhang, {\it Transition threshold for the 3-D Couette flow in a finite channel}, arXiv:2006.00721.


\bibitem{DHB}F. Daviaud, J. Hagseth and P. Berg$\acute{e}$, {\it  Subcritical transition to turbulence in plane Couette flow}, Phys. Rev. Lett., 69(1992), 2511-2514.

\bibitem{Deng}W. Deng, {\it Resolvent estimates for a two-dimensional non-self-adjoint operator}, Commun. Pure Appl. Anal. 12 (2013), no. 1, 547-596.

\bibitem{DWZ}W. Deng, J. Wu and P. Zhang, {\it Stability of Couette flow for 2D Boussinesq system with vertical dissipation}, Journal of Functional Analysis, Volume 281, Issue 12 (2021).


\bibitem{DW}P. Drazin and W. Reid, {\it  Hydrodynamic Stability}, Cambridge Monographs Mech. Appl. Math., Cambridge Univ. Press, New York, 1981.

\bibitem{FGLT}M. Fei, C. Gao, Z. Lin, and T. Tao, {\it  Prandtl-batchelor flows on an annulus}, arXiv:2111.07114v1.

\bibitem{Ga0}T. Gallay, {\it  Stability and interaction of vortices in two-dimensional viscous flows}, Discr. Cont. Dyn. Systems Ser. S 5 (2012), 1091-1131.

\bibitem{Ga}T. Gallay, {\it  Enhanced dissipation and axisymmetrization of two-dimensional viscous vortices}, Archive for Rational Mechanics and Analysis, volume 230 (2018), 939-975.

\bibitem{GR}T. Gallay and L. M. Rodrigues, {\it  Sur le temps de vie de la turbulence bidimensionnelle (in French)}, Ann. Fac. Sci. Toulouse 17 (2008), 719-733.

\bibitem{GV}T. Gallay and V. $\check{S}$ver$\acute{a}$k, {\it  Arnold's variational principle and its application to the stability of planar vortices}, arXiv:2110.13739.

\bibitem{GW0}T. Gallay and C. E. Wayne, {\it  Invariant manifolds and the long-time asymptotics of the Navier-Stokes and
vorticity equations on $\mathbb{R}^2$}, Arch. Ration. Mech. Anal., 163 (2002), 209-258.

\bibitem{GW}T. Gallay and C. E. Wayne, {\it  Global stability of vortex solutions of the two-dimensional Navier-Stokes equation}, Comm. Math. Phys., 255 (2005), 97-129.

\bibitem{IMM}S. Ibrahim, Y. Maekawa, and N. Masmoudi, {\it On pseudospectral bound for non-selfadjoint operators and its application to stability of Kolmogorov flows},  Annals of PDE, 5(2019), 5:2. 

\bibitem{IJ-NonE}A. Ionescu and H. Jia, {\it Nonlinear inviscid damping near monotonic shear flows}, arXiv:2001.03087, to appear in Acta Mathematica.


\bibitem{IJ-E}A. Ionescu and H. Jia, {\it Axi-symmetrization near point vortex solutions for the 2D Euler equation}, arXiv 1904.09170.


\bibitem{Kato}T. Kato, {\it Perturbation Theory for Linear Operators}. Grundlehren der Mathematischen Wissenschaften 132. Berlin: Springer, 1966.

\bibitem{Kel}L. Kelvin, {\it Stability of fluid motion-rectilinear motion of viscous fluid between two parallel plates}, Phil. Mag., 24(1887), 188-196.

\bibitem{LWZ-O} T. Li, D. Wei and Z. Zhang, {\it Pseudospectral and spectral bounds for the Oseen vortices operator}, Annales Scientifiques of Ecole Normale Sup\'erieure, 53 (2020),  993-1035.

\bibitem{LX}Z. Lin and M. Xu, {\it Metastability of Kolmogorov flows and inviscid damping of shear flows}, Arch. Rat. Mech. Anal.,  231(2019), 1811-1852. 

\bibitem{LZ1}Z. Lin and C. Zeng, {\it Inviscid dynamic structures near Couette flow}, Arch. Rat. Mech. Anal., 200(2011), 1075-1097.

\bibitem{LZ2}Z. Lin and C. Zeng, {\it Unstable manifolds of Euler equations}, Comm. Pure Appl. Math., 66(2013), no. 11, 1803-1836.

\bibitem{LZ3}Z. Lin and C. Zeng, {\it Instability, index theorem, and exponential trichotomy for Linear Hamiltonian PDEs,}  Memoirs of the American Mathematical Society,  275(2022), no. 1347.


\bibitem{Pa}A. Pazy, {\it Semigroups of linear operators and applications to partial differential equations}, Applied Mathematical Sciences 44, Springer-Verlag, New York, 1983.

\bibitem{Rey}O. Reynolds, {\it An experimental investigation of the circumstances which determine whether the motion of water shall be direct or sinuous, and of the law of resistance in parallel channels}, Proc. R. Soc. Lond., 35(1883), 84.


\bibitem{SH}P. Schmid and D. Henningson, {\it Stability and Transition in Shear Flows}, Applied Mathematical Sciences 142, Springer-Verlag, New York, 2001.

\bibitem{TA}N. Tillmark and P. H. Alfredsson, {\it Experiments on transition in plane Couette flow}, J. Fluid Mech., 235(1992), 89-102.

\bibitem{TTRD}L. Trefethen, A. Trefethen, S. Reddy and T. Driscoll, {\it Hydrodynamic stability without eigenvalues}, Science, 261(1993), 578-584.

\bibitem{Wei}D. Wei, {\it Diffusion and mixing in fluid flow via the resolvent estimate}, Science China Mathematics, 64 (2021), 507-518.

\bibitem{WZ-3C}D. Wei and Z. Zhang, {\it Transition threshold for the 3D Couette flow in Sobolev space}, Comm. Pure Appl. Math., 74(2021), 2398-2479.


\bibitem{WZZ-E}D. Wei, Z. Zhang, and W. Zhao, {\it Linear inviscid damping for a class of monotone shear flow in Sobolev spaces}, Comm. Pure Appl. Math., 71(2018), 617-687.

\bibitem{WZZ-E2}D. Wei, Z. Zhang and W. Zhao, {\it Linear inviscid damping and vorticity depletion for shear flows}, Annals of PDE, 5(2019), 5:3.

\bibitem{WZZ-NS}D. Wei, Z. Zhang and W. Zhao, {\it Linear inviscid damping and enhanced dissipation for the Kolmogorov
flow}, Adv. Math., 362 (2020), 106963.


\bibitem{Ya}A. Yaglom, {\it Hydrodynamic instability and transition to turbulence}, Fluid Mech. Appl. 100, Springer-Verlag, New York, 2012.


\bibitem{ZDE}M. Zelati, M. Delgadino and T. Elgindi, {\it On the Relation between Enhanced Dissipation Timescales and Mixing Rates}, Comm. Pure Appl. Math., 73(2020), 1205-1244.

\bibitem{ZZ}M. Zelati, C. Zillinger, {\it On degenerate circular and shear flows: the point vortex and power law circular flows}, Communications in Partial Differential Equations, 44 (2019), 110-155.

\bibitem{Z1}C. Zillinger, {\it On circular flows: linear stability and damping}, J. Differential Equations, 263 (2017), no. 11, 7856-7899.




\end{thebibliography}
\end{document}